\documentclass[10pt,reqno]{amsart}  \usepackage{amsmath, amsthm, amssymb, enumerate} \usepackage{times} \usepackage{graphicx} \usepackage[usenames,dvipsnames,svgnames,table]{xcolor} \usepackage[colorlinks=true, pdfstartview=FitV, linkcolor=blue, citecolor=blue, urlcolor=blue]{hyperref} \usepackage{tikz}  \usepackage{color} \usepackage{xcolor} \usepackage{datetime} \usepackage{fancyhdr,lastpage} \chardef\coloryes=0  \chardef\isitdraft=0 \ifnum\isitdraft=1   \textwidth 16truecm \textheight 8.4in\oddsidemargin0.2truecm\evensidemargin0.7truecm\voffset-.9truecm         \def\eqref#1{({\ref{#1}})}                   \def\startnewsection#1#2{\section{#1}\label{#2}\setcounter{equation}{0}}      \textwidth 16truecm \textheight 8.4in\oddsidemargin0.2truecm\evensidemargin0.7truecm\voffset-.9truecm   \def\nnewpage{}  \else   \def\startnewsection#1#2{\section{#1}\label{#2}\setcounter{equation}{0}}      \textwidth 16.2truecm \textheight 8.6in\oddsidemargin0.2truecm\evensidemargin0.7truecm\voffset-.9truecm \def\nnewpage{}  \fi \begin{document} \def\ques{{\colr \underline{??????}\colb}} \def\nto#1{{\colC \footnote{\em \colC #1}}} \def\fractext#1#2{{#1}/{#2}} \def\fracsm#1#2{{\textstyle{\frac{#1}{#2}}}}    \def\nnonumber{} \def\les{\lesssim} \def\plusdelta{+\delta} \def\colr{{}} \def\colg{{}} \def\colb{{}} \def\colu{{}} \def\cole{{}} \def\colA{{}} \def\colB{{}} \def\colC{{}} \def\colD{{}} \def\colE{{}} \def\colF{{}} \ifnum\coloryes=1   \definecolor{coloraaaa}{rgb}{0.1,0.2,0.8}   \definecolor{colorbbbb}{rgb}{0.1,0.7,0.1}   \definecolor{colorcccc}{rgb}{0.8,0.3,0.9}   \definecolor{colordddd}{rgb}{0.0,.5,0.0}   \definecolor{coloreeee}{rgb}{0.8,0.3,0.9}   \definecolor{colorffff}{rgb}{0.8,0.9,0.9}   \definecolor{colorgggg}{rgb}{0.5,0.0,0.4}  \def\colb{\color{black}}  \def\colr{\color{red}}  \def\cole{\color{colorgggg}}  \def\colu{\color{blue}}  \def\colg{\color{colordddd}}  \def\colgray{\color{colorffff}}  \def\colA{\color{coloraaaa}}  \def\colB{\color{colorbbbb}}  \def\colC{\color{colorcccc}}  \def\colD{\color{colordddd}}  \def\colE{\color{coloreeee}}  \def\colF{\color{colorffff}}  \def\colG{\color{colorgggg}} \fi \ifnum\isitdraft=1    \chardef\coloryes=1     \baselineskip=17.6pt \pagestyle{myheadings} \reversemarginpar \def\const{\mathop{\rm const}\nolimits}   \def\diam{\mathop{\rm diam}\nolimits}     \def\rref#1{{\ref{#1}{\rm \tiny \fbox{\tiny #1}}}} \def\theequation{\fbox{\bf \thesection.\arabic{equation}}} \def\plusdelta{+\delta} \def\startnewsection#1#2{\newpage\colg \section{#1}\colb\label{#2} \setcounter{equation}{0} \pagestyle{fancy} \lhead{\colb Section~\ref{#2}, #1 } \cfoot{} \rfoot{\thepage\ of \pageref{LastPage}} \lfoot{\colb{\today,~\currenttime}~(dkt1)}} \chead{} \rhead{\thepage} \def\nnewpage{\newpage} \newcounter{startcurrpage} \newcounter{currpage} \def\llll#1{{\rm\tiny\fbox{#1}}}    \def\blackdot{{\color{red}{\hskip-.0truecm\rule[-1mm]{4mm}{4mm}\hskip.2truecm}}\hskip-.3truecm}    \def\bluedot{{\colC {\hskip-.0truecm\rule[-1mm]{4mm}{4mm}\hskip.2truecm}}\hskip-.3truecm}    \def\purpledot{{\colA{\rule[0mm]{4mm}{4mm}}\colb}}    \def\pdot{\purpledot} \else    \baselineskip=12.8pt    \def\blackdot{{\color{red}{\hskip-.0truecm\rule[-1mm]{4mm}{4mm}\hskip.2truecm}}\hskip-.3truecm}    \def\purpledot{{\rule[-3mm]{8mm}{8mm}}}    \def\pdot{} \fi \def\KK{K} \def\ema#1{{#1}} \def\emb#1{#1} \ifnum\isitdraft=1   \def\llabel#1{\nonumber} \else   \def\llabel#1{\nonumber} \fi \def\tepsilon{\tilde\epsilon} \def\restr{\bigm|} \def\into{\int_{\Omega}} \def\intu{\int_{\Gamma_1}} \def\intl{\int_{\Gamma_0}} \def\tpar{\tilde\partial} \def\bpar{\,|\nabla_2|} \def\barpar{\bar\partial} \def\qq{{p_0}} \def\FF{F} \def\gdot{{\color{green}{\hskip-.0truecm\rule[-1mm]{4mm}{4mm}\hskip.2truecm}}\hskip-.3truecm} \def\bdot{{\color{blue}{\hskip-.0truecm\rule[-1mm]{4mm}{4mm}\hskip.2truecm}}\hskip-.3truecm} \def\cydot{{\color{cyan} {\hskip-.0truecm\rule[-1mm]{4mm}{4mm}\hskip.2truecm}}\hskip-.3truecm} \def\rdot{{\color{red} {\hskip-.0truecm\rule[-1mm]{4mm}{4mm}\hskip.2truecm}}\hskip-.3truecm} \def\tdot{\fbox{\fbox{\bf\color{blue}\tiny I'm here; \today \ \currenttime}}} \def\nts#1{{\color{blue}\hbox{\bf ~#1~}}}  \def\ntsr#1{\vskip.0truecm{\color{red}\hbox{\bf ~#1~}}\vskip0truecm}  \def\ntsf#1{\footnote{\hbox{\bf ~#1~}}}  \def\ntsf#1{\footnote{\color{blue}\hbox{\bf ~#1~}}}  \def\bigline#1{~\\\hskip2truecm~~~~{#1}{#1}{#1}{#1}{#1}{#1}{#1}{#1}{#1}{#1}{#1}{#1}{#1}{#1}{#1}{#1}{#1}{#1}{#1}{#1}{#1}\\} \def\biglineb{\bigline{$\downarrow\,$ $\downarrow\,$}} \def\biglinem{\bigline{---}} \def\biglinee{\bigline{$\uparrow\,$ $\uparrow\,$}} \def\ceil#1{\lceil #1 \rceil} \def\gdot{{\color{green}{\hskip-.0truecm\rule[-1mm]{4mm}{4mm}\hskip.2truecm}}\hskip-.3truecm} \def\bluedot{{\color{blue} {\hskip-.0truecm\rule[-1mm]{4mm}{4mm}\hskip.2truecm}}\hskip-.3truecm} \def\rdot{{\color{red} {\hskip-.0truecm\rule[-1mm]{4mm}{4mm}\hskip.2truecm}}\hskip-.3truecm} \def\dbar{\bar{\partial}} \newtheorem{Theorem}{Theorem}[section] \newtheorem{Corollary}[Theorem]{Corollary} \newtheorem{Proposition}[Theorem]{Proposition} \newtheorem{Lemma}[Theorem]{Lemma} \newtheorem{Remark}[Theorem]{Remark} \newtheorem{definition}{Definition}[section] \def\theequation{\thesection.\arabic{equation}} \def\sqrtg{\sqrt{g}} \def\DD{{\mathcal D}} \def\OO{\tilde\Omega} \def\EE{{\mathcal E}} \def\lot{{\rm l.o.t.}}                        \def\inon#1{~~~\hbox{#1}}                 \def\endproof{\hfill$\Box$\\} \def\square{\hfill$\Box$\\} \def\inon#1{~~~\hbox{#1}}                 \def\comma{ {\rm ,\qquad{}} }             \def\commaone{ {\rm ,\qquad{}} }          \def\dist{\mathop{\rm dist}\nolimits}     \def\sgn{\mathop{\rm sgn\,}\nolimits}     \def\Tr{\mathop{\rm Tr}\nolimits}     \def\dive{\mathop{\rm div}\nolimits}     \def\grad{\mathop{\rm grad}\nolimits}    \def\curl{\mathop{\rm curl}\nolimits}    \def\det{\mathop{\rm det}\nolimits}     \def\supp{\mathop{\rm supp}\nolimits}   \def\re{\mathop{\rm {\mathbb R}e}\nolimits} \def\wb{\bar{\omega}} \def\Wb{\bar{W}} \def\indeq{\qquad{}}                      \def\indeqtimes{\indeq\indeq\times}  \def\period{.}                            \def\semicolon{\,;}                       \newcommand{\cD}{\mathcal{D}} \newcommand{\cH}{\mathcal{H}} \newcommand{\al}{\alpha} \def\gsdfgsdfga{\par} \newcommand{\be}{\beta} \newcommand{\ga}{\gamma} \newcommand{\de}{\delta} \newcommand{\ep}{\epsilon} \newcommand{\si}{\sigma} \def\gsdkfjg{\gsdfgsdfga} \newcommand{\Si}{\Sigma} \newcommand{\vfi}{\varphi} \newcommand{\om}{\omega} \newcommand{\Om}{\Omega} \newcommand{\cqd}{\hfill $\qed$\\ \medskip} \newcommand{\imp}{\Rightarrow} \newcommand{\tr}{\operatorname{tr}} \newcommand{\vol}{\operatorname{vol}} \newcommand{\id}{\operatorname{id}} \def\dghendjsai{\gsdkfjg} \def\gsdfgdsfgdsfg{\dghendjsai} \newcommand{\p}{\parallel} \newcommand{\norm}[1]{\Vert#1\Vert} \newcommand{\abs}[1]{\vert#1\vert} \newcommand{\nnorm}[1]{\left\Vert#1\right\Vert} \newcommand{\aabs}[1]{\left\vert#1\right\vert} \title[Lagrangian regularity]{A Lagrangian interior regularity result for the incompressible free boundary Euler equation with surface tension} \author[Disconzi]{Marcelo M. Disconzi} \address{Department of Mathematics\\ Vanderbilt University\\ Nashville, TN, USA} \email{marcelo.disconzi@vanderbilt.edu} \thanks{MMD is partially supported by the NSF grant DMS-1812826,  a Sloan Research Fellowship provided by the Alfred P.~Sloan foundation, and from a Discovery grant administered by Vanderbilt University.} \author[Kukavica]{Igor Kukavica} \address{Department of Mathematics\\ University of Southern California\\ Los Angeles, CA 90089} \email{kukavica@usc.edu} \thanks{IK is partially supported by the NSF grants DMS-1615239 and DMS-1907992.} \author[Tuffaha]{Amjad Tuffaha} \address{Department of Mathematics and Statistics\\ American University of Sharjah\\ Sharjah, UAE} \email{atufaha@aus.edu} \begin{abstract} We consider the three-dimensional  incompressible free-boundary Euler equations in a bounded domain and with  surface tension. Using Lagrangian coordinates, we establish a priori estimates for solutions with minimal regularity assumptions on the initial data. \end{abstract} \maketitle \gsdfgdsfgdsfg \startnewsection{Introduction}{sec01}  We consider the free boundary Euler equation of incompressible flow defined on a moving three dimensional domain $\Omega(t) \subseteq \mathbb{R}^{3}$, which read   \begin{align}    \label{ZZ01}    &    u_{t} + (u \cdot \nabla) u  + \nabla p = 0     \inon{in~$\cD$}    \\    &    \label{ZZ02}    \dive u =0     \inon{in~$\cD$}  \\    \label{ZZ03}   & p = \sigma \cH       \inon{on~$\partial\cD$} \\    \label{ZZ04}   & (\partial_{t} + u^{\al} \partial_{x_{\al}}) |_{\partial \cD} \in T \partial \cD   \end{align} where $\cD = \bigcup_{0 \leq t \leq T}  \{ t \} \times \Omega(t)$, $u$ is the fluid's velocity and $p$ its pressure. The symbol $\sigma \geq 0$  denotes the surface tension parameter and $\cH$ is, for each $t$, the  mean curvature of the boundary $\partial \Om(t)$ embedded into $\mathbb{R}^3$.  Also,  $T \partial \cD$ stands for  the tangent bundle of $\partial \cD$ and \eqref{ZZ04} expresses the condition  that the boundary moves with the speed equal to the normal component of $u$. The initial data are given by   \begin{align}    \label{ZZ05}    &   u(\cdot,0) = u_{0} \\    \label{ZZ06} & \Omega(0) = \Omega    .   \end{align} Our aim in this paper is to obtain a~priori estimates for a  local-in-time existence result of solutions to this system with minimal regularity assumptions on the initial data and when $\si > 0$. \gsdfgdsfgdsfg  The first existence results for \eqref{ZZ01}--\eqref{ZZ06} are those of Nalimov \cite{NalimovCauchyPoisson} and Yosihara \cite{YosiharaGravity}, who considered regular irrotational data. In the case of zero surface tension, i.e., $\sigma=0$, Ebin has shown  in  \cite{Ebin_ill-posed} that the problem is ill-posed without the Rayleigh-Taylor stability condition. The problem of well-posedness under the Rayleigh-Taylor condition and in the case of zero surface tension was solved by Wu \cite{WuWaterWaves2d,  WuWaterWaves}. Regarding optimal regularity of the initial data, Wang~et~al obtained in  \cite{WZZZ} the local existence under the sharp Sobolev regularity $H^{2.5+\delta}$ for the zero surface tension case, extending the previous result of Alazard~et~al \cite{AlazardetalLowregularity}, who considered irrotational data. For the Euler equations in ${\mathbb R}^2$ or ${\mathbb R}^{3}$, the sharpness of the exponent $2.5+\delta$ was shown in \cite{BourgainLi15}. \gsdfgdsfgdsfg The well-posedness of the non-zero surface tension problem, although requiring no additional stability condition, is challenging {on} its own right and has to be approached differently.  While the surface tension has a regularizing effect, the boundary evolution contributes  to the energy estimates at top order. Controlling such top order boundary terms, which would automatically vanish in the $\si=0$ case, requires an intricate analysis of several boundary terms that express the coupling of the boundary geometry with the interior evolution. Such analysis is particularly delicate in low regularity spaces in that the ellipticity provided by the mean curvature  cannot be exploited to same extent as in higher regularity due to the presence of rough coefficients in the mean curvature equation. \gsdfgdsfgdsfg Consequently, currently, one does not have estimates that close in spaces near the threshold $H^{2.5+\delta}$ in the case $\si > 0$, with exception of the simpler situation of irrotational data,  for which Alazard, Burq, and  Zuily established a full  local-wellposedness result with optimal regularity \cite{AlazardWaterWaveSurfaceTension}. \gsdfgdsfgdsfg Regarding rotational fluids with $\si > 0$, Schweizer \cite{SchweizerFreeEuler} constructed solutions with rotational data in $H^{4.5}$ with an additional vorticity condition at the surface. Coutand and Shkoller \cite{CoutandShkollerFreeBoundary} used the Lagrangian formulation and  constructed solutions with   $H^{4.5}$ initial data without this restriction. At the same time,  Shatah and Zeng obtained in \cite{ShatahZengGeometry}  a~priori estimates for $H^{3}$ data in Eulerian coordinates using techniques of infinite dimensional geometry in the spirit of Ebin and Marsden \cite{EbinMarsden} (see also \cite{ShatahZengInterface}, where the authors showed how to use their a priori estimates to obtain a local existence result). Ignatova and the second author obtained  in \cite{IgorMihaelaSurfaceTension}  a~priori estimates with interior regularity in $H^{3.5}$, using the Lagrangian (direct) approach, while Ebin and the first author established a local-existence result in $H^{3.5 + \de}$ using a combination of the Lagrangian approach, infinite-dimensional geometry, and semi-group theory \cite{DisconziEbinFreeBoundary3d}. \gsdfgdsfgdsfg For other results on irrotational fluids with surface tension see~\cite{  AlazardStabilizationSurfaceTension,  AlazardCapillaryWaterWaves,  AlazardDispersiveSurfaceTension,  AmbroseVortexSheets,  AmbroseMasmoudiWaterWaves,  Beale_et_al_Growth,  FeffermanetallSplashSurfaceTension,  Iguchi_well_posedness_capillary_gravity,  IonescuGlobal3dCapillary,  GermainMasmoudiShatahGlobalCapillary,  IfrimTataru2dCapillary,  IonescuPusateriGlobal2dwaterSurfaceTension,  YosiharaGravitySurfaceTension}.  Further related results with non-zero surface tension, including the case of rotational fluids, vortex sheets, two-phase fluids, and singular limits, are \cite{ ShkollerVortexSheets,  CoutandShkollerSplash,  Disconzilineardynamic, DisconziEbinFreeBoundary2d, FeffermanIonescuLie,  IonescuPusateriGlobal2dwaterModel,  Ogawa-Tani_FreeBoundarySurfaceTension,  PusateriTwoPhaseOnePhaseLimitSurfaceTension}.  Free-boundary problems constitute a very active and fast-growing area of research, and a complete, or even thorough review of prior works is beyond the scope of this paper. A partial list of references relevant to the above discussion and the results of this paper is \cite{ AlazardAboutGlobalExistence, AlazardCollectionWaterWaves, AlazardCauchyWaterWaves, AlazardCauchyTheoryWaterWaves, AlazardDelortGlobal2dWater, AlazardetalLowregularity,  BieriWu1, FeffermanetallSplash, FeffermanStructural, FeffermanetallMuskat, ShkollerElliptic,  ChristodoulouLindbladFree,  CoutandSingularity, MR2660719, CraigHamiltonianWaterWaves, PoryferreEmergingBottom,  GermainMasmoudiShatahGlobalWaterWaves3D, IfrimHunterTataru, IfrimTataruGlobalWater, IfrimTataruGravityConstant, Iguchi_et_al_FreeBoundary, IonescuPusateriWaterWaves2d, KukavicaTuffaha-Free2dEuler, KukavicaTuffaha-RegularityFreeEuler, KukavicaTuffahaVicol-3dFreeEuler, LannesWaterWaves,  LannesWaterWavesBook, LindbladFree1, Lindblad-LinearizedFreeBoundary, LindbladNordgren-AprioriFreeBoundary, Ogawa-Tani_FiniteDepth, WuAlmostGlobal, WuGlobal}. \gsdfgdsfgdsfg In this manuscript, we use the Lagrangian variables and derive a priori estimates assuming that the initial velocity is in $H^{2.5+\delta}$, where $0 < \delta < 0.5$. Some further minimal assumptions on the data are also necessary in order to  obtain that the second time derivative of the velocity is in $L^2$ (cf.~Remark \ref{R01} below). \gsdfgdsfgdsfg Unlike in the zero surface tension case, when $\si > 0$ the interface regularity is driven by the regularity of the pressure, which can be controlled as a solution to an elliptic problem with  Neumann boundary condition, in terms of the velocity time derivative. The control of the velocity and its time derivatives  is established using a combination of time and tangential energy estimates. Such time and tangential estimates for the velocity lead to some crucial boundary terms whose control is technically challenging (we stress that such boundary terms are absent if $\si=0$). Exploiting the non-linear structure of the  equations and of the boundary condition, we are able to obtain an estimate that reads schematically as \begin{align*} \frac{d}{dt} \Vert \partial_t^2 v \Vert^2_{L^2(\Om)}  + \frac{d}{dt} \Vert \overline{\partial} \partial_t (v\cdot N) \Vert^2_{L^2(\partial \Om)} + \frac{d}{dt} \Vert \overline{\partial}^{2+\delta}  (v\cdot N) \Vert^2_{L^2(\partial \Om)} \lesssim \frac{d}{dt}\int_\Om P_1 + \frac{d}{dt} \int_{\partial \Om} P_2 + P_{3}, \end{align*} where $v$ is the Lagrangian velocity, $P_1$ and $P_2$  are polynomial expressions on the   Lagrangian velocity, the Lagrangian pressure, and their time derivatives, $P_3$ is a polynomial in several norms of the fluid variables, $\overline{\partial}$ are derivatives tangent to the boundary, $N$ is the unit outer normal to $\partial \Omega$, and $\delta$ is a small number. Upon time-integration, the term $P_{3}$ is treated by a standard Gronwall argument. The  remaining two terms on the right hand side, however, do not have a definite sign.  To control such terms we need to show that they can be bounded by lower order terms plus top order terms with small coefficients. Unfortunately, it turns out that this does not seem possible. \gsdfgdsfgdsfg However, if we define ``non-linear energies" that involve powers of the velocity and its derivatives, we arrive at \begin{align*} \frac{d}{dt} \Vert \partial^2_t v \Vert^a_{L^2(\Om)}  + \frac{d}{dt} \Vert \overline{\partial} \partial_t (v\cdot N) \Vert^b_{L^2(\partial \Om)} + \frac{d}{dt} \Vert \overline{\partial}^{2+\delta}  (v\cdot N) \Vert^c_{L^2(\partial \Om)} \lesssim \frac{d}{dt}\int_\Om P_1 + \frac{d}{dt} \int_{\partial \Om} P_2 + P_{3}, \end{align*} for certain $a,b,c > 0$ (and possibly different $P_1$, $P_2$, and $P_3$). Now, using a combination of interpolation, Sobolev embeddings, and Young's inequality, we obtain that, after time integration, the right hand side is bounded by \begin{align*} \ep_0( \Vert \partial^2_t v \Vert^\alpha_{L^2(\Om)}  + \Vert \overline{\partial} \partial_t (v\cdot N) \Vert^\beta_{L^2(\partial \Om)} +  \Vert \overline{\partial}^{2+\delta}  (v\cdot N) \Vert^\gamma_{(\partial \Om)}) + \int_0^t P, \end{align*} where $\ep_0$ is a small number and $\alpha$, $\beta$, and $\gamma$ depend on  $a$, $b$, $c$ and $\delta$.  The problem then reduces to the algebraic question of whether it is possible to choose $a$, $b$, $c$, so that the powers on both sides match. This turns out to be possible\footnote{In the presentation of the  results it is not necessary to work with such general $a$, $b$, and $c$. Having found the correct exponents, we already define our energy with them; see \eqref{EQ00}.} precisely when $0 < \delta < 0.5$ which, unwrapping all definitions, corresponds to estimating $v$ in $H^{2.5+ \delta}$ (this can be seen explicitly in the last estimate, see equations \eqref{EQ197} and \eqref{EQ196}). Note that these energies control, aside from $\partial_t^2 v$ in the interior, only tangential derivatives of the normal component of the velocity on the boundary. But once these have been controlled, a bound for the full norms of $v$ is obtained via div-curl estimates, with  control of the divergence coming from the divergence-free condition, 
control of the curl from the Cauchy invariance, and control of the normal components given by the  above energy estimate. \gsdfgdsfgdsfg To treat the case of a general bounded domain, we employ local coordinates near the boundary and suitably chosen cut-off functions. Such localization techniques are not straightforwardly adapted to the framework of fractional derivatives that we need to employ to obtain estimates in $H^{2.5 +\delta}$. Therefore, we consider the problem in two steps. First, we take the initial domain $\Omega$ to have the simpler topology   \begin{equation}    \Omega= {\mathbb T}^{2} \times [0,1]       \llabel{8Th sw ELzX U3X7 Ebd1Kd Z7 v 1rN 3ZZ32}   \end{equation} and denote its bottom and top boundaries by $\Gamma_{0}$ and $\Gamma_{1}$, respectively. Assume that  the lower boundary   \begin{equation}    \Gamma_{0} = {\mathbb T}^{2} \times \{ 0\}      \llabel{Gi irR XG KWK0 99ov BM0FDJ Cv k opZZ11}   \end{equation} is rigid, while the upper boundary $\Gamma_1(t)$ evolves in time according to the unknown flow map   \begin{equation}    \eta(x_{1},x_{2}, 1,t) \colon \Gamma_{1} \to \Gamma_{1}(t)    \llabel{Y NQ2 aN9 4Z 7k0U nUKa mE3OjU 8D FZZ18}   \end{equation} and is such that $\Gamma_1(0)$ equals   \begin{equation}    \Gamma_{1} = {\mathbb T}^{2} \times \{1 \}    .    \llabel{ YFF okb SI2 J9 V9gV lM8A LWThDP nZZ12}   \end{equation} We then establish our result for this type of domains, see Theorem \ref{T01}. This simplified setting already presents all the main difficulties of the problem, but makes it easier to focus on its core aspects without being distracted by the technicalities caused by the use of fractional derivatives in local charts and their interaction with cut-off functions. Then, we show how to adapt the estimates leading to Theorem \ref{T01} to a general domain, stated as  Theorems~\ref{T02} and~\ref{T03}. \gsdfgdsfgdsfg \gsdfgdsfgdsfg \startnewsection{The Lagrangian variables and the main statement}{sec02} We assume that $\Omega(t)$ is initially the 1-periodic channel    \begin{equation}      \Omega(0) = \Omega =   \mathbb{T}^{2} \times [0,1],       \label{EQ181}   \end{equation} with the rigid bottom boundary  $       \Gamma_{0} = \mathbb{T}^{2} \times \{ 0 \} $. The top boundary $\Gamma_{1}(t)$ evolves and is initially equal to $      \Gamma_{1} = \mathbb{T}^{2} \times \{1 \} $. (The general case is discussed in Section~\ref{sec12} below.) We use $\eta$ to denote the Lagrangian variable and $a$ the inverse of the matrix $ \nabla \eta$. The Lagrangian formulation of the Euler equations then reads   \begin{align}    &    \partial_{t} v^{\alpha}       + a^{\mu \alpha} \partial_{\mu} q =0     \inon{in $\Omega \times [0,T]$}       \label{ZZ13}    \\    \label{ZZ14} & a^{\alpha \beta} \partial_{\alpha} v_{\beta} =0    \inon{in $\Omega \times [0,T]$}  \\&    \partial_{t} \eta = v \inon{in $\Omega \times [0,T]$}    \label{EQ06}   \\    \label{ZZ15} & \partial_{t} a^{\alpha \beta} + a^{\alpha \gamma} \partial_{\mu} v_{\gamma} a^{\mu \beta} = 0     \inon{in $\Omega \times [0,T] $}   & \\    \label{ZZ16} &   a^{\mu \alpha} N_{\mu} q       + \sigma | a^{T}N| \Delta_{g} \eta^{\alpha} =0    \inon{on $\Gamma_{1} \times [0,T]$}   \\    \label{ZZ17} & v^{\mu} N_{\mu} =0    \inon{on $\Gamma_{0} \times [0,T]$}    ,   \end{align} where $N$ is the unit outward normal vector to $\partial \Omega$ and $\Delta_{g}$ is the Laplacian induced on $\partial \Omega(t)$ by $    \eta|_{\Gamma_1} $ i.e.,   \begin{align}    \Delta_{g} (\cdot) = \frac{1}{\sqrt{g}} \partial_{i} (\sqrt{g} g^{ij} \partial_{j}(\cdot))    ,    \label{EQ07}   \end{align} where     \begin{align}    g_{ij} = \partial_{i} \eta^{\mu} \partial_{j} \eta_{\mu},    \label{EQ08}   \end{align} while $g$ is the determinant of the matrix $[g_{ij}]_{i,j=1,2}$. Above and in the sequel, we use the summation convention on repeated indices.  The Greek letters run from 1 to 3, while the Latin go from 1 to 2. \gsdfgdsfgdsfg \gsdfgdsfgdsfg \gsdfgdsfgdsfg \gsdfgdsfgdsfg \gsdfgdsfgdsfg The following is the main statement in which we establish a~priori estimates for the local existence of solutions with initial data $v_0=(v_0^{1},v_0^{2},v_0^{3})\in H^{2.5+\epsilon}$, where $\epsilon\in(0,1/2)$, \gsdfgdsfgdsfg \cole \begin{Theorem} \label{T01} Let $\sigma >0$ and $\epsilon\in(0,1/2)$. Assume that $v_0$ is a smooth divergence-free vector field on $\Om$. Then there exist constants $C_*,T_* > 0$ such that any smooth solution $(v,q)$ to  \eqref{ZZ13}--\eqref{ZZ17}  with initial condition $v_0$ defined on the time interval $[0,T_*]$, satisfies   \begin{align}    \begin{split}     \Vert v\Vert_{H^{2.5+\epsilon}}     + \Vert \partial_{t}v\Vert_{H^{1.5}}     + \Vert \partial_{t}^2v\Vert_{L^2}     + \Vert q\Vert_{H^{2.25+\epsilon/2}}     + \Vert \partial_{t} q\Vert_{H^{1}}     \leq C_*     ,    \end{split}    \label{EQ12a}   \end{align} where $T_{*}, C_*>0$ depend only on  $\Vert v_0\Vert_{H^{2.5+\epsilon}}$, $\Vert v_0^{3}|_{\Gamma_1}\Vert_{H^{2.5}(\Gamma_1)}$, and $\sigma>0$. \end{Theorem} \colb \gsdfgdsfgdsfg Above and in the sequel, if the domain of the Sobolev space is not designated, it is understood to be $\Omega$, while other domains (typically $\Gamma_1$, $\Gamma_0$, and $\partial\Omega$) are explicitly noted. \gsdfgdsfgdsfg In  Remark~\ref{R01} below we show that the condition $\Vert v_0|_{\Gamma_1}\Vert_{H^{2.5}(\Gamma_1)}<\infty$ can be replaced by $\Vert \Delta_2 v_0^3|_{\Gamma_1}\Vert_{H^{0.5}(\Gamma_1)}<\infty$, where $\Delta_2$ is the boundary Laplacian. This last condition is not only sufficient but  is also \emph{necessary} for $\partial_{t}^2v_0\in L^2$. \gsdfgdsfgdsfg Instead of working with $\epsilon>0$, we introduce, for simplicity of notation, the parameter $\nu=1/2-\epsilon$  and thus consider   \begin{equation}      v_0\in H^{3-\nu}    \label{EQ169}   \end{equation} where we assume   \begin{equation}    \nu\in[0,0.5)    .    \llabel{P u 3EL 7HP D2V Da ZTgg zcCC mbvc7EQ184}   \end{equation} By introducing $\nu$, many exponents and Sobolev parameters have simpler forms. Note that we include also the value $\nu=0$ since all the results below hold for this borderline value as well. \gsdfgdsfgdsfg The proof consists of a series of estimates on $v$ and $q$ involving the energies   \begin{equation}     E_{0}        =   \Vert v\Vert_{H^{2.5-\nu/2}}    \label{EQ13}   \end{equation} and   \begin{equation}     E_{1}        =   \Vert \partial_{t}v\Vert_{H^{1.5}}         + \Vert \partial_{t}^2 v\Vert_{L^2}    \label{EQ13a}    .   \end{equation} It is also convenient to introduce the total energy   \begin{equation}    E    = E_{0}^2      + E_{1}      + 1    .    \label{EQ00}   \end{equation} Note that in \eqref{EQ00} $E_0$ is squared while $E_1$ is not. \gsdfgdsfgdsfg Since $\sigma>0$ does not vary, we set $    \sigma=1 $ from here on. \gsdfgdsfgdsfg As usual, in what follows, the symbol $a\les b$ stands for $a \le C b$, where $C$ is a constant. \gsdfgdsfgdsfg \gsdfgdsfgdsfg \gsdfgdsfgdsfg \gsdfgdsfgdsfg \startnewsection{Preliminary estimates}{sec03}  \gsdfgdsfgdsfg In the first lemma, we collect a~priori estimates on the map $\eta$ and the cofactor matrix $a=(\nabla \eta)^{-1}$. \gsdfgdsfgdsfg \cole \begin{Lemma}    \label{L01} Assume that  $    \Vert v\Vert_{L^{\infty}([0,T];H^{3-\nu}(\Omega))} \leq M $. If    \begin{equation}    T \le \frac{1}{C M}    \label{ZZ43}   \end{equation} where $C\ge1$ is a sufficiently large constant, then  the following statements hold:\\ \noindent (i)      $\Vert\eta \Vert_{H^{3-\nu}}              \leq C               $ for $t\in[0,T]$,\\ \noindent (ii)   $\Vert a \Vert_{H^{2-\nu}}              \leq C  $,\\ \noindent (iii)   $\Vert \partial_{t}a \Vert_{H^{s}}       \le        C       \Vert \nabla v \Vert_{H^{s}}$ for  $0\le s\le 2-\nu$ with\\ $\Vert \partial_{t}^2a \Vert_{H^{0.5-\nu/2}}              \leq               C (\Vert \partial_{t}v \Vert_{H^{1.5-\nu/2}}              +                           \Vert v \Vert_{H^{2}}              \Vert v \Vert_{H^{2-\nu/2}})  $ and $\Vert \partial_{t}^2a \Vert_{H^{\nu/2}}              \leq               C (\Vert \partial_{t}v \Vert_{H^{1+\nu/2}}              +                           \Vert v \Vert_{H^{2}}              \Vert v \Vert_{H^{1.5+\nu/2}}) $, and \\ \noindent (iv)  For every $\epsilon_0\in(0,1]$, we have $\Vert a-I\Vert_{H^{2-\nu}}\le\epsilon_0$ and $\Vert \nabla\eta-I\Vert_{H^{2-\nu}}\le\epsilon_0$ provided $T\le 1/C\epsilon_0 M$. \end{Lemma} \colb \gsdfgdsfgdsfg Since the proofs follow easily from \eqref{EQ06} and \eqref{ZZ15}, we only briefly outline them. \gsdfgdsfgdsfg {\begin{proof}[Proof of Lemma~\ref{L01}] (i) By \eqref{EQ06}, we have   \begin{equation}    \Vert\eta\Vert_{H^{3-\nu}}    \leq    \Vert x\Vert_{H^{3-\nu}}    +    \left\Vert \int_{0}^{t} v\right\Vert_{H^{3-\nu}}    \les     1+ T M    \llabel{0 qq P cC9 mt6 0og cr TiA3 HEjw TKEQ223}   \end{equation} and the rest follows from the choice \eqref{ZZ43}.\\ (ii) From \eqref{ZZ15}, we get $    \Vert a\Vert_{H^{2-\nu}}    \les    1    +    M    \int_{0}^{t} \Vert a\Vert_{H^{2-\nu}}^2 $, and the claim is obtained using the Gronwall lemma.\\ (iii) Follows directly from \eqref{ZZ15}.\\ (iv) To obtain the claim, we  use $  a-I = \int_{0}^{t} \partial_{t} a$ and then apply (ii) to obtain   \begin{align*}    \Vert a - I \Vert_{H^{2-\nu}}      &\les      \int_{0}^{t} \Vert \partial_{t} a \Vert_{H^{2-\nu}}      \les  \int_{0}^{t}  \Vert \nabla v \Vert_{H^{2-\nu}}      \les \epsilon_0    \end{align*} for $t \leq T' = \epsilon/C M$. Similarly,   \begin{align*}    \Vert \nabla\eta - I \Vert_{H^{2-\nu}}      &\les      \int_{0}^{t} \Vert \nabla v \Vert_{H^{2-\nu}}      \les  \int_{0}^{t}  \Vert v \Vert_{H^{3-\nu}}      \les \epsilon_0      \end{align*} under the condition $t\leq \epsilon_0/ C M$. \end{proof} \gsdfgdsfgdsfg \startnewsection{Pressure estimates}{sec04} For reference, we state the trace inequality for the vector fields with the square integrable divergence (cf.~\cite{CF,T}). \gsdfgdsfgdsfg \cole \begin{Lemma} \label{L02} Let $\phi $ be a $3D$ vector field in $L^{2}(\Omega)$,  and $a(x)$ a matrix function with components $a^{\mu \alpha}\in L^{\infty}(\Omega)$. If $ \partial_{\mu} (a^{\mu \alpha} \phi_{\alpha}) \in L^{2}(\Omega)$  and $ a^{\mu \alpha} \phi_{\alpha} \in L^{2}(\Omega)$ for $\mu=1,2,3$, then  $a^{\mu \alpha} \phi_{\alpha} N_{\mu} \in H^{-1/2}(\partial \Omega)$ and   \begin{align}        \Vert a^{\mu \alpha} \phi_{\alpha} N_{\mu}  \Vert_{H^{-1/2}(\partial \Omega)}         \les         \Vert \partial_{\mu} (a^{\mu \alpha} \phi_{\alpha}) \Vert_{L^{2}(\Omega)}           +            \sum_{\mu=1}^{3}             \Vert a^{\mu \alpha} \phi_{\alpha}\Vert_{L^{2}(\Omega)}     .    \llabel{8ymK eu J Mc4 q6d Vz2 00 XnYU tLR9ZZ86}   \end{align} \end{Lemma} \colb \gsdfgdsfgdsfg \gsdfgdsfgdsfg Next, we derive elliptic estimates satisfied by the  Lagrangian pressure $q$ and its time derivative $\partial_{t}q$. \cole \begin{Lemma} \label{L04} (i)   For the Lagrangian pressure $q$, we have   \begin{align}    \begin{split}    \Vert q \Vert_{H^{2.5-\nu/2}}         &\les           \Vert v\Vert_{H^{1.5}}          \Vert v\Vert_{H^{2.5-\nu/2}}             + \Vert \partial_{t}v \Vert_{H^{1.5}}^{1-\nu/3} 	      \Vert \partial_{t}v \Vert_{L^2}^{\nu/3}          +            \Vert q(0)\Vert_{H^{1}}          +1          +  \int_{0}^{t}\Vert \partial_{t}q\Vert_{H^{1}}    .   \end{split}    \label{ZZ22}   \end{align} \noindent (ii)  For the time derivative of the Lagrangian pressure, we have   \begin{align}   \begin{split}    \Vert \partial_{t}q \Vert_{H^{1}}      &\les      \Vert          \partial_{t}^2 v      \Vert_{L^2}      +         \Vert \partial_{t}v\Vert_{L^2}^{(1-\nu)/3}       \Vert \partial_{t}v\Vert_{H^{1.5}}^{(2+\nu)/3}       \Vert v\Vert_{H^{2.5-\nu/2}}      +         \Vert v\Vert_{H^{1.5}}^{(3-3\nu)/(2-\nu)}         \Vert v\Vert_{H^{2.5-\nu/2}}^{3/(2-\nu)} \\&\indeq      +     \Vert v\Vert_{H^{2.5-\nu/2}}      \Vert q\Vert_{H^{1}}^{(2-2\nu-2\delta_0)/(3-\nu)}      \Vert q\Vert_{H^{2.5-\nu/2}}^{(1+\nu+2\delta_0)/(3-\nu)}     \\&\indeq      +        \Vert v\Vert_{H^{3-\nu}}^{\nu/(1-\nu)}        \Vert v\Vert_{H^{2.5-\nu/2}}^{(1-2\nu)/(1-\nu)}      (     \Vert q\Vert_{H^{1}}+1)    .    \label{ZZ23}   \end{split}   \end{align} \end{Lemma} \colb \gsdfgdsfgdsfg Above and in the sequel, $\delta_0>0$ denotes an arbitrarily small constant. In most places it appears when bounding the $L^\infty$ norm of a quantity with a suitable Sobolev norm. \gsdfgdsfgdsfg The exponent $2.5-\nu/2$  in \eqref{ZZ22} is not the highest regularity of the pressure one may obtain (which is $3-\nu$). It is chosen because it is the highest Sobolev exponent for $q$ which can be estimated in terms of $\Vert v\Vert_{H^{2.5-\nu/2}}$, for which in turn we have control based on  Section~\ref{sec07} and  the properties \eqref{EQ80a} and \eqref{EQ38} below. \gsdfgdsfgdsfg \gsdfgdsfgdsfg \gsdfgdsfgdsfg \gsdfgdsfgdsfg Using the notation \eqref{EQ13} and \eqref{EQ13a} and introducing   \begin{equation}    F    = \Vert v\Vert_{H^{3-\nu}}    \llabel{ GYjPXv FO V r6W 1zU K1W bP ToaW JEQ121}   \end{equation} we may rewrite \eqref{ZZ22} and \eqref{ZZ23} in simpler forms as   \begin{align}     \begin{split}    G    &=    \Vert q \Vert_{H^{2.5-\nu/2}}      \les      P_0      +        \left(         P_0         + \int_{0}^{t} P        \right)        (E_0           + E_1^{1-\nu/3}        )          +  \int_{0}^{t}P   \end{split}    \llabel{JuK nxBLnd 0f t DEb Mmj 4lo HY yhZEQ14}   \end{align} and   \begin{align}   \begin{split}    H &=    \Vert \partial_{t}q \Vert_{H^{1}}      \les      E_1      +      \left(P_0+\int_{0}^{t}P\right)      \\&\indeq\indeq\indeq\indeq\times      \Bigl(       E_1^{(2+\nu)/3}       E_0      +       E_0^{3/(2-\nu)}      +      E_0      G^{(1+\nu+2\delta_0)/(3-\nu)}      +        F^{\nu/(1-\nu)}        E_0^{(1-2\nu)/(1-\nu)}      \Bigr)    .    \llabel{y MjM9 1zQS4p 7z 8 eKa 9h0 Jrb ac ZZ23a}   \end{split}   \end{align} Above and in the sequel,  $P_0$ denotes a generic polynomial in  $\Vert v_0\Vert_{H^{3-\nu}}$, $\Vert \partial_{t}v(0)\Vert_{H^{1.5}}$, and  $\Vert \partial_{t}^2v(0)\Vert_{L^2}$, while $P$   denotes a generic polynomial in $\Vert v \Vert_{H^{3-\nu}}$,  $ \Vert \partial_{t}v \Vert_{H^{1.5}}$,  $\Vert q \Vert_{H^{2.5-\nu/2}}$, $ \Vert \partial_{t}q \Vert_{H^{1}}$, and $\Vert \partial_{t}^2v\Vert_{L^2}$. Using the notation \eqref{EQ00} and $\nu<1/2$, we then have   \begin{align}     \begin{split}    G    &     \les        \left(         P_0         + \int_{0}^{t} P        \right)           E^{(3-\nu)/3}   \end{split}    \label{EQ127}   \end{align} and   \begin{align}   \begin{split}    H &     \les      E      +      \left(P_0+\int_{0}^{t}P\right)      \\&\indeq\indeq\indeq\times      \Bigl(      E^{(7+2\nu)/6}      +       E^{3/(4-2\nu)}      +      E^{1/2}      G^{(1+\nu+2\delta_0)/(3-\nu)}      +        F^{\nu/(1-\nu)}        E^{(1-2\nu)/(2-2\nu)}      \Bigr)    .    \llabel{ekci rexG 0z4n3x z0 Q OWS vFj 3jL EQ128}   \end{split}   \end{align} Since $(7+2\nu)/6\ge 3/(4-2\nu)$, we get   \begin{align}   \begin{split}    H &     \les      E      +      \left(P_0+\int_{0}^{t}P\right)      \Bigl(      E^{(7+2\nu)/6}      +      E^{1/2}      G^{(1+\nu+2\delta_0)/(3-\nu)}      +        F^{\nu/(1-\nu)}        E^{(1-2\nu)/(2-2\nu)}      \Bigr)    ,    \label{EQ138}   \end{split}   \end{align} where, as pointed out above, $\delta_0>0$ denotes an arbitrarily small constant. \gsdfgdsfgdsfg \gsdfgdsfgdsfg \gsdfgdsfgdsfg \gsdfgdsfgdsfg \gsdfgdsfgdsfg \gsdfgdsfgdsfg \gsdfgdsfgdsfg Before the proof of the lemma, we recall the Piola identity   \begin{align}    \partial_{\mu}  a^{\mu \alpha} = 0     \label{ZZ00}   \end{align} (cf.~\cite[p.~462]{Ev}). \gsdfgdsfgdsfg \begin{proof}[Proof of Lemma~\ref{L04}] First, we apply $a^{\lambda \alpha}  \partial_{\lambda}$ to the equation~\eqref{ZZ13}  and obtain   \begin{align}    a^{\lambda \alpha} \partial_{\lambda} ( {a^{\mu}}_{\alpha} \partial_{\mu} q )        = - a^{\lambda \alpha} \partial_{\lambda} \partial_{t}v_{\alpha}        = \partial_{t} a^{\lambda \alpha} \partial_{\lambda} v_{\alpha}     ,    \label{ZZ37}   \end{align} where we used the divergence free condition \eqref{ZZ14} in the last step. Isolating $\Delta    q$, we obtain the  Poisson equation    \begin{align}   \begin{split}     \Delta q  &= \partial_{t}a^{\lambda \alpha} \partial_{\lambda} v_{\alpha}          +          (\delta^{\lambda \alpha}- a^{\lambda \alpha })              \partial_{\lambda} (\delta^{\mu}_{\alpha} \partial_{\mu} q)           + a^{\lambda \alpha}              \partial_{\lambda}                       \bigl(                          (\delta^{\mu}_{\alpha} - {a^{\mu}}_{\alpha})                          \partial_{\mu} q                      \bigr)      \\&       =          \partial_{\lambda}( \partial_{t}a^{\lambda \alpha}  v_{\alpha} )         +          \partial_{\lambda}          (             (\delta^{\lambda \alpha}- a^{\lambda \alpha })                \partial_{\alpha} q           )           +               \partial_{\lambda}                       \bigl(                        a^{\lambda \alpha}                          (\delta^{\mu}_{\alpha} - {a^{\mu}}_{\alpha})                          \partial_{\mu} q                      \bigr)      \\&       =        \partial_{\lambda}        \Bigl(          \partial_{t}a^{\lambda \alpha}  v_{\alpha}          +             (\delta^{\lambda \alpha}- a^{\lambda \alpha })                \partial_{\alpha} q           +                         a^{\lambda \alpha}                          (\delta^{\mu}_{\alpha} - {a^{\mu}}_{\alpha})                          \partial_{\mu} q            \Bigr)      =: \partial_{\lambda}f^{\lambda}   \end{split}    \llabel{hW XUIU 21iI AwJtI3 Rb W a90 I7r zZZ50}   \end{align} on $\Omega$, in addition to the boundary conditions   \begin{align}    \partial_{3} q       = (\delta^{\alpha 3} - a^{\alpha 3 }) \partial_{\alpha} q  - \partial_{t} v^{3} =:  h_{1}     \inon{on $\Gamma_1$}    \label{EQ219}   \end{align} and   \begin{align}    \partial_{3} q =  (\delta^{\alpha 3} - a^{\alpha 3 }) \partial_{\alpha} q =:  h_{2}    \inon{on~$\Gamma_{0}$}    ,    \label{EQ220}   \end{align} which result from restricting \eqref{ZZ13}  to  $\Gamma_1$ and $\Gamma_{0}$, respectively. Moreover, from the boundary condition \eqref{ZZ16}, we have   \begin{align}    q =    (1- a^{33}  )q     -      \partial_{i}(\sqrt{g} g^{ij} \partial_{j} \eta^{3} )    \inon{on~$\Gamma_{1} \times [0,T] $}    .    \label{ZZ21}   \end{align} We now invoke the estimate for  $q$ from \cite{IgorMihaelaSurfaceTension, DisconziKukavicaIncompressible} whereby   \begin{align}    \Vert q \Vert_{H^{2.5-\nu/2}}      \les     \Vert \partial_{\lambda}f^{\lambda} \Vert_{H^{0.5-\nu/2}}      + \Vert h_{1} \Vert_{H^{1-\nu/2}(\Gamma_1)}     +\Vert h_{2} \Vert_{H^{1-\nu/2}(\Gamma_0)}      + \Vert q \Vert_{L^{2}(\Gamma_{1})}    .    \label{ZZ87}   \end{align}  Note that   \begin{align}    \begin{split}     \Vert \partial_{\lambda}f^{\lambda} \Vert_{H^{0.5-\nu/2}}      &\les     \Vert             \partial_{t}a^{\lambda \alpha} \partial_{\lambda} v_{\alpha}      \Vert_{H^{0.5-\nu/2}}         +     \Vert             (\delta^{\lambda \alpha}- a^{\lambda \alpha })                \partial_{\alpha}  \partial_{\lambda}q     \Vert_{H^{0.5-\nu/2}}         \\&\indeq     +     \Vert                        a^{\lambda \alpha}                            \partial_{\lambda} {a^{\mu}}_{\alpha}                          \partial_{\mu} q     \Vert_{H^{0.5-\nu/2}}         +     \Vert                        a^{\lambda \alpha}                          (\delta^{\mu}_{\alpha} - {a^{\mu}}_{\alpha})                 \partial_{\lambda}                         \partial_{\mu} q     \Vert_{H^{0.5-\nu/2}}         \\&     \les     \Vert             \partial_{t}a^{\lambda \alpha} \partial_{\lambda} v_{\alpha}      \Vert_{H^{0.5-\nu/2}}         +     \Vert             \delta^{\lambda \alpha}- a^{\lambda \alpha }     \Vert_{H^{1.5+\delta_0}}     \Vert                \partial_{\alpha}  \partial_{\lambda}q     \Vert_{H^{0.5-\nu/2}}         \\&\indeq     +     \Vert                        a^{\lambda \alpha}     \Vert_{H^{1.5+\delta_0}}     \Vert                            \partial_{\lambda} {a^{\mu}}_{\alpha}                          \partial_{\mu} q     \Vert_{H^{0.5-\nu/2}}         +     \Vert                        a^{\lambda \alpha}     \Vert_{H^{1.5+\delta_0}}     \Vert                          \delta^{\mu}_{\alpha} - {a^{\mu}}_{\alpha}     \Vert_{H^{1.5+\delta_0}}     \Vert                 \partial_{\lambda}                         \partial_{\mu} q     \Vert_{H^{0.5-\nu/2}}         \\&     \les     \Vert             \partial_{t}a^{\lambda \alpha} \partial_{\lambda} v_{\alpha}      \Vert_{H^{0.5-\nu/2}}         +     \sum_{\alpha,\lambda}     \epsilon_0     \Vert                \partial_{\alpha}  \partial_{\lambda}q     \Vert_{H^{0.5-\nu/2}}         \\&\indeq     +     \sum_{\alpha,\lambda}     \Vert                            \partial_{\lambda} {a^{\mu}}_{\alpha}                          \partial_{\mu} q     \Vert_{H^{0.5-\nu/2}}         +     \epsilon_0     \sum_{\lambda,\mu}     \Vert                 \partial_{\lambda}  \partial_{\mu} q     \Vert_{H^{0.5-\nu/2}}        ,    \end{split}    \llabel{AI qI 3UEl UJG7 tLtUXz w4 K QNE TvEQ15}   \end{align} where we used  a multiplicative Sobolev inequality $     \Vert f g\Vert_{H^{r}}\les                 \Vert f\Vert_{H^{1.5+\delta_0}}\Vert g\Vert_{H^{r}}$ for $0\le r\le 1.5$. Also, $\epsilon_0>0$ denotes everywhere a constant which can be made arbitrarily small by choosing $T>0$ sufficiently small as in Lemma~\ref{L01}(iv) above. Therefore,   
\begin{align}    \begin{split}     \Vert \partial_{\lambda}f^{\lambda} \Vert_{H^{0.5-\nu/2}}      &\les     \Vert \partial_{t} a \Vert_{H^{1}}      \Vert v \Vert_{H^{2-\nu/2}}     +     \epsilon_0     \Vert q\Vert_{H^{2.5-\nu/2}}     +     \Vert a\Vert_{H^{2-\nu}}     \Vert q\Vert_{H^{2+\nu/2}}    .    \end{split}    \llabel{X zqW au jEMe nYlN IzLGxg B3 A uJ8EQ163}   \end{align} Using \eqref{ZZ87} with \eqref{EQ219}--\eqref{ZZ21}, we get   \begin{align}    \begin{split}      \Vert q \Vert_{H^{2.5-\nu/2}}         &\les      \Vert \partial_{t} a \Vert_{H^{1}}      \Vert v \Vert_{H^{2-\nu/2}}     +     \epsilon_0     \Vert q\Vert_{H^{2.5-\nu/2}}     +     \Vert q\Vert_{H^{2+\nu/2}}        \\&\indeq             + \Vert I-a \Vert_{H^{1 +\delta_0}(\Gamma_{1})}                 \Vert \nabla q\Vert_{H^{1-\nu/2}(\Gamma_{1}) }             + \Vert \partial_{t}v \Vert_{H^{1-\nu/2}(\Gamma_{1}) }             + \Vert I-a \Vert_{H^{1 +\delta_0}(\Gamma_{0})}                 \Vert \nabla q\Vert_{H^{1-\nu/2}(\Gamma_{0}) }        \\&\indeq             + \Vert 1- a^{33} \Vert_{H^{1 +\delta_0}(\Gamma_{1})}                    \Vert q \Vert_{L^{2}(\Gamma_{1})}             +               \Vert                  \partial_{i}(\sqrt{g} g^{ij} \partial_{j} \eta^{3} )                  \Vert_{L^2(\Gamma_1)}    ,    \llabel{ 6VS 6Rc PJ 8OXW w8im tcKZEz Ho p ZZ54}    \end{split}   \end{align} whence, by Lemma~\ref{L01} (in particular $\Vert a-I\Vert_{H^{1.5+\delta_0}}\le \epsilon_0$),   \begin{align}    \begin{split}      \Vert q \Vert_{H^{2.5-\nu/2}}          &\les              \Vert v \Vert_{H^{2}}             \Vert v \Vert_{H^{2-\nu/2}}             + \epsilon_0               \Vert q\Vert_{H^{2.5-\nu/2}}             +             \Vert q\Vert_{H^{2+\nu/2}}             + \Vert \partial_{t}v \Vert_{H^{1.5-\nu/2}}            + Q(\Vert D\eta\Vert_{L^{\infty}(\Gamma_1)})               \Vert \eta\Vert_{H^{2}(\Gamma_1)}         \\& 	\les           \Vert v\Vert_{H^{1.5}}          \Vert v\Vert_{H^{2.5-\nu/2}}             + \epsilon_0               \Vert q\Vert_{H^{2.5-\nu/2}}             +             \Vert q\Vert_{H^{2+\nu/2}}         \\&\indeq             + \Vert \partial_{t}v \Vert_{H^{1.5}}^{1-\nu/3} 	      \Vert \partial_{t}v \Vert_{L^2}^{\nu/3}            + Q(\Vert D\eta\Vert_{L^{\infty}(\Gamma_1)})               \Vert \eta\Vert_{H^{2}(\Gamma_1)}    \llabel{84G 1gS As0 PC owMI 2fLK TdD60y nHZZ55}    \end{split}   \end{align} where $Q$ denotes a rational function in the indicated argument and where we used   \begin{equation}    \Vert v\Vert_{H^{2}}    \les    \Vert v\Vert_{H^{1.5}}^{(1-\nu)/(2-\nu)}    \Vert v\Vert_{H^{2.5-\nu/2}}^{1/(2-\nu)}    \label{EQ118}   \end{equation} and $    \Vert v\Vert_{H^{2-\nu/2}}    \les    \Vert v\Vert_{H^{1.5}}^{1/(2-\nu)}    \Vert v\Vert_{H^{2.5-\nu/2}}^{(1-\nu)/(2-\nu)} $ in the last step. Finally, note that $Q(\Vert D\eta\Vert_{L^{\infty}(\Gamma_1)})               \Vert \eta\Vert_{H^{2}(\Gamma_1)}\les 1$ and   \begin{align}   \begin{split}          \Vert q \Vert_{H^{2+\nu/2}}       \end{split}   \leq    \epsilon_0       \Vert q \Vert_{H^{2.5-\nu/2}}       + C\Vert q\Vert_{H^{1}}    \leq                \epsilon_0       \Vert q \Vert_{H^{2.5-\nu/2}}       + C\Vert q(0)\Vert_{H^{1}}       + C\int_{0}^{t}\Vert \partial_{t}q\Vert_{H^{1}}          .    \llabel{ g 7lk NFj JLq Oo Qvfk fZBN G3o1DgEQ165}   \end{align}   \gsdfgdsfgdsfg   \gsdfgdsfgdsfg \gsdfgdsfgdsfg (ii)  Differentiating $      \partial_{\lambda}        ( a^{\lambda \alpha}  {a^{\mu}}_{\alpha} \partial_{\mu} q )        = \partial_{\lambda}(\partial_{t} a^{\lambda \alpha} v_{\alpha} ) $, we obtain that the time derivative of the Lagrangian pressure satisfies   \begin{align}   \begin{split}      \partial_{\lambda}       (a^{\lambda\alpha} {a^{\mu}}_{\alpha} \partial_{\mu}\partial_{t}q)        &=             \partial_{\lambda} (\partial_{t}^2a^{\lambda \alpha} v_{\alpha} )         +  \partial_{\lambda} (\partial_{t}a^{\lambda \alpha} \partial_{t} v_{\alpha} )        -          \partial_{\lambda} (                                \partial_{t}a^{\lambda \alpha}                                   \partial_{\alpha} q                            )       \\&\indeq        +       \partial_{\lambda} (                              \partial_{t}a^{\lambda \alpha}                                   (\delta^{\mu}_{\alpha}-{a^{\mu}}_{\alpha}) \partial_{\mu} q                            )        -       \partial_{\lambda} (                              a^{\lambda \alpha}                                  \partial_{t} {a^{\mu}}_{\alpha} \partial_{\mu} q                            )     =:       \partial_{\lambda} \tilde f^{\lambda}   \end{split}    \label{ZZ58}   \end{align} in $\Omega$. The boundary conditions, which are deduced from \eqref{ZZ13} and \eqref{ZZ21}, read   \begin{align*}      a^{3\alpha} {a^{\mu}}_{\alpha} \partial_{\mu}\partial_{t}q      =       - a^{3\alpha}\partial_{t}^2v_{\alpha}      - \partial_{t}a^{3\alpha}\partial_{t}v_{\alpha}      - \partial_{t} (                         a^{3\alpha} {a^{\mu}}_{\alpha}                       )        \partial_{\mu}q      =:  \tilde h_{1}     \inon{on $\Gamma_1$}   \end{align*} and   \begin{align*}      a^{3\alpha} {a^{\mu}}_{\alpha} \partial_{\mu}\partial_{t}q      = - a^{3\alpha}\partial_{t}^2v_{\alpha}      - \partial_{t}a^{3\alpha}\partial_{t}v_{\alpha}      - \partial_{t} (                         a^{3\alpha} {a^{\mu}}_{\alpha}                       )        \partial_{\mu}q      =:  \tilde h_{2}     \inon{on $\Gamma_0$}   \end{align*} with   \begin{align*}    \partial_{t}    q =    \partial_{t}(1- a^{33}  )q     +    (1- a^{33}  )\partial_{t}q     -  \partial_{i}\partial_{t}(\sqrt{g} g^{ij} \partial_{j} \eta^{3} )    \inon{on~$\Gamma_{1} \times [0,T] $}    .   \end{align*} Thus we may invoke the estimate     \begin{align}      \Vert \partial_{t}q \Vert_{H^{1}}            \les \Vert \tilde f \Vert_{L^2}                +   \Vert \tilde{h}_{1} \Vert_{H^{-1/2}(\Gamma_{1})}                +   \Vert \tilde h_{2} \Vert_{H^{-1/2}(\Gamma_{0})}                + \Vert \partial_{t}q\Vert_{L^{2}(\Gamma_{1})}      \llabel{ Cn 9 hyU h5V SP5 z6 1qvQ wceU dVJZZ62}   \end{align} from \cite{IgorMihaelaSurfaceTension} and obtain   \begin{align}    \begin{split}     \Vert \partial_{t}q \Vert_{H^{1}}         &\les        \sum_{\lambda}         \Bigl(           \Vert \partial_{t}^2a^{\lambda \alpha} v_{\alpha}\Vert_{L^2}         +  \Vert               \partial_{t}a^{\lambda \alpha} \partial_{t} v_{\alpha}             \Vert_{L^2}        +          \Vert                \partial_{t}a^{\lambda \alpha}                       \partial_{\alpha} q           \Vert_{L^2}        \\&\indeq\indeq\indeq\indeq        +       \Vert                             \partial_{t}a^{\lambda \alpha}                                   (\delta^{\mu}_{\alpha}-{a^{\mu}}_{\alpha}) \partial_{\mu} q        \Vert_{L^2}       +       \Vert                             a^{\lambda \alpha}                                   \partial_{t}{a^{\mu}}_{\alpha} \partial_{\mu} q        \Vert_{L^2}         \Bigr)      \\&\indeq      +       \Vert  a^{3 \alpha}\partial_{t}^2 v_{\alpha} \Vert_{H^{-1/2}(\Gamma_1)}      + \Vert  \partial_{t} a^{3\alpha} \partial_{t} v_{\alpha}\Vert_{H^{-1/2}(\Gamma_1)}      + \Vert  \partial_{t}(a^{3\alpha}{{a^{\mu}}_{\alpha}})\partial_{\mu}q\Vert_{H^{-1/2}(\Gamma_1)}      \\&\indeq      + \Vert  a^{3 \alpha}\partial_{t}^2 v_{\alpha} \Vert_{H^{-1/2}(\Gamma_0)}      + \Vert  \partial_{t} a^{3\alpha} \partial_{t} v_{\alpha}\Vert_{H^{-1/2}(\Gamma_0)}      + \Vert  \partial_{t}(a^{3\alpha}{{a^{\mu}}_{\alpha}})\partial_{\mu}q\Vert_{H^{-1/2}(\Gamma_0)}     \\&\indeq     + \Vert \partial_{t} a^{33}   q  \Vert_{L^2(\Gamma_1)}     +      \Vert \partial_{i}(\partial_{t}(\sqrt{g} g^{ij}  \partial_{j} \eta^{3}) )   \Vert_{L^2(\Gamma_1)}    .    \end{split}    \label{ZZ56}   \end{align} Denote by $S$ the sum in $\lambda$. Then   \begin{align}    \begin{split}       S       &\les           \Vert \partial_{t}^2a\Vert_{H^{\nu/2}}           \Vert v\Vert_{H^{1.5-\nu/2}}         +  \Vert               \partial_{t}a            \Vert_{H^{1.5-\nu/2}}            \Vert              \partial_{t} v            \Vert_{H^{\nu/2}}        +       \Vert \partial_{t} a\Vert_{H^{1.5-\nu/2}}       \Vert \nabla q\Vert_{H^{\nu/2}}    .    \end{split}    \label{ZZ59}   \end{align} It turns out that all three terms on the right side of \eqref{ZZ59} appear in the upper bounds \eqref{EQ129} and \eqref{EQ131} below  thus not leading to  any additional terms compared to \eqref{EQ129} and \eqref{EQ131}. Next, we estimate  $\Vert \tilde h_1\Vert_{H^{-1/2}(\Gamma_1)}$ (the bound for $\tilde h_2$ is the same). We write   \begin{align}    \begin{split}     \Vert \tilde h_1\Vert_{H^{-1/2}(\Gamma_1)}     &\les      \Vert        a^{3 \alpha} \partial_{t}^2v_{ \alpha}           \Vert_{H^{-1/2}(\Gamma_1)}      +      \Vert        \partial_{t} a^{3 \alpha} \partial_{t}v_{ \alpha}           \Vert_{H^{-1/2}(\Gamma_1)}      +      \Vert         \partial_{t} (a^{3\alpha}{a^{\mu}}_\alpha)\partial_{\mu}q      \Vert_{H^{-1/2}(\Gamma_1)}      \\&      = T_1 + T_2 + T_3    .    \end{split}    \llabel{JsB vX D G4E LHQ HIa PT bMTr sLsm ZZ46}   \end{align} For the first term, we have   \begin{align}    \begin{split}     T_1      &      \les      \sum_{\beta}      \Vert           a^{\beta \alpha} \partial_{t}^2v_{ \alpha}           \Vert_{L^2}      +      \Vert         \partial_{\beta}         (            a^{\beta \alpha} \partial_{t}^2v_{ \alpha}              )      \Vert_{L^2}      =      \sum_{\beta}      \Vert           a^{\beta \alpha} \partial_{t}^2v_{ \alpha}           \Vert_{L^2}      +      \Vert         a^{\beta \alpha}  \partial_{\beta}\partial_{t}^2v_{ \alpha}           \Vert_{L^2}      \\&      \les      \Vert           a               \Vert_{L^\infty}      \Vert          \partial_{t}^2 v      \Vert_{L^2}      +      \Vert         a^{\beta \alpha}  \partial_{\beta}\partial_{t}^2v_{ \alpha}           \Vert_{L^2}     \les      \Vert          \partial_{t}^2 v      \Vert_{L^2}      +      \Vert         \partial_{t}^2 a^{\beta \alpha}  \partial_{\beta}v_{ \alpha}           \Vert_{L^2}      +      \Vert         \partial_{t} a^{\beta \alpha}  \partial_{\beta}\partial_{t}v_{ \alpha}           \Vert_{L^2}     \\&     \les      \Vert          \partial_{t}^2 v      \Vert_{L^2}      +      \Vert         \partial_{t}^2 a      \Vert_{H^{\nu/2}}      \Vert         v      \Vert_{H^{2.5-\nu/2}}      +      \Vert         \partial_{t} a      \Vert_{H^{1.5-\nu/2}}      \Vert         \partial_{t}v      \Vert_{H^{1+\nu/2}}     \\&     \les      \Vert          \partial_{t}^2 v      \Vert_{L^2}      +           \Vert \partial_{t}v\Vert_{H^{1+\nu/2}}         \Vert v\Vert_{H^{2.5-\nu/2}}      +        \Vert v\Vert_{H^{2}}        \Vert v\Vert_{H^{1.5+\nu/2}}        \Vert v\Vert_{H^{2.5-\nu/2}}      +      \Vert         v      \Vert_{H^{2.5-\nu/2}}      \Vert         \partial_{t}v      \Vert_{H^{1+\nu/2}}     \\&     \les      \Vert          \partial_{t}^2 v      \Vert_{L^2}      +           \Vert \partial_{t}v\Vert_{L^2}^{(1-\nu)/3}         \Vert \partial_{t}v\Vert_{H^{1.5}}^{(2+\nu)/3}       \Vert v\Vert_{H^{2.5-\nu/2}}      +         \Vert v\Vert_{H^{1.5}}^{(3-3\nu)/(2-\nu)}         \Vert v\Vert_{H^{2.5-\nu/2}}^{3/(2-\nu)}    ,    \end{split}    \label{EQ129}   \end{align} where we used Lemma~\ref{L02} in the first step and the divergence condition \eqref{ZZ14} in the fourth. Also, we used \eqref{EQ118} and $    \Vert v\Vert_{H^{1.5+\nu/2}}    \les    \Vert v\Vert_{H^{1.5}}^{(2-2\nu)/(2-\nu)}    \Vert v\Vert_{H^{2.5-\nu/2}}^{\nu/(2-\nu)} $. For $T_2$, we  apply Lemma~\ref{L02} and estimate   \begin{align}      \begin{split}        T_2         &\les      \Vert         \partial_{\beta}(\partial_{t} a^{\beta \alpha} \partial_{t}v_{ \alpha}     )      \Vert_{L^2}      +      \sum_{\beta}      \Vert         \partial_{t} a^{\beta \alpha} \partial_{t}v_{ \alpha}           \Vert_{L^2}      =      \Vert        \partial_{t} a^{\beta \alpha} \partial_{\beta}\partial_{t}v_{ \alpha}      \Vert_{L^2}      +      \sum_{\beta}      \Vert         \partial_{t} a^{\beta \alpha} \partial_{t}v_{ \alpha}           \Vert_{L^2}      \\&      \les           \Vert         \partial_{t} a      \Vert_{H^{1.5-\nu/2}}      \Vert         \partial_{t}v      \Vert_{H^{1+\nu/2}}    .      \end{split}     \llabel{tXGyOB 7p 2 Os4 3US bq5 ik 4Lin 76EQ130}     \end{align} Observe that this upper bound already appears in \eqref{EQ129}. For $T_3$, we simply use multiplicative Sobolev inequalities to write   \begin{align}      \begin{split}      T_3      &\les      \Vert         \partial_{t} (a^{3\alpha}{a^{\mu}}_\alpha)\partial_{\mu}q      \Vert_{H^{0.5+\delta_0}}      \les       \Vert \partial_{t}a\Vert_{H^{1.5-\nu/2}}      \Vert q\Vert_{H^{1.5+\nu/2+\delta_0}}      \end{split}    \label{EQ131}    \end{align} for an arbitrarily small parameter $\delta_0>0$. Therefore,   \begin{align}    \begin{split}     T_1+T_2+T_3     &\les      \Vert          \partial_{t}^2 v      \Vert_{L^2}      +         \Vert \partial_{t}v\Vert_{L^2}^{(1-\nu)/3}       \Vert \partial_{t}v\Vert_{H^{1.5}}^{(2+\nu)/3}       \Vert v\Vert_{H^{2.5-\nu/2}}      +         \Vert v\Vert_{H^{1.5}}^{(3-3\nu)/(2-\nu)}         \Vert v\Vert_{H^{2.5-\nu/2}}^{3/(2-\nu)}    \\&\indeq    +     \Vert v\Vert_{H^{2.5-\nu/2}}      \Vert q\Vert_{H^{1}}^{(2-2\nu-2\delta_0)/(3-\nu)}      \Vert q\Vert_{H^{2.5-\nu/2}}^{(1+\nu+2\delta_0)/(3-\nu)}    .    \end{split}    \label{ZZ61}   \end{align} Finally, we estimate the last two terms in \eqref{ZZ56}, representing an upper bound for  $ \Vert \partial_{t}q\Vert_{L^{2}(\Gamma_{1})}$. In this case, we have   \begin{align}    \begin{split}     \Vert \partial_{t}q\Vert_{L^{2}(\Gamma_{1})}     &\les     \Vert \partial_{t} a^{33}   q  \Vert_{H^{0.5+\delta_0}}     + \sum_{i=1}^2    \Vert \partial_{t}(\sqrt{g} g^{ij}  \partial_{j} \eta^{3})   \Vert_{H^{1.5}}     \\&     \les      \Vert \partial_{t} a   \Vert_{H^{1.5}}       \Vert q \Vert_{H^{1}}     + \Vert v\Vert_{H^{2.5}}    \les        \Vert v\Vert_{H^{3-\nu}}^{\nu/(1-\nu)}        \Vert v\Vert_{H^{2.5-\nu/2}}^{(1-2\nu)/(1-\nu)}      \colb      (     \Vert q\Vert_{H^{1}} + 1)    \end{split}    \label{ZZ66}   \end{align} where we used   \begin{equation}    \Vert v\Vert_{H^{2.5}}    \les    \Vert v\Vert_{H^{3-\nu}}^{\nu/(1-\nu)}    \Vert v\Vert_{H^{2.5-\nu/2}}^{(1-2\nu)/(1-\nu)}    .    \llabel{9O TkUxmp I8 u GYn fBK bYI 9A QzCFEQ95}   \end{equation} Combining \eqref{ZZ56}, \eqref{ZZ59} (cf.~the comment right after), \eqref{ZZ61}, and \eqref{ZZ66} then leads to \eqref{ZZ23}. \end{proof} \gsdfgdsfgdsfg \begin{Remark} \label{R01} {\rm Here we sketch an argument showing finiteness of the energy $E(0)$ under given conditions on the initial data. First, by  \eqref{ZZ13}, \eqref{ZZ16}, and \eqref{ZZ37},  we have   \begin{align}      \begin{split}         &\Delta q(0) = - \partial^{\alpha}v^{\lambda}\partial_{\lambda}v_{\alpha}(0)         \inon{in $\Omega$}         \\&         q(0) = - \Delta_2 \eta^{3}(0) = 0 \inon{on $\Gamma_{1}$}         \\& 	\partial_3 q = 0 \inon{on $\Gamma_{0}$}    \end{split}    \llabel{ w3h0 geJftZ ZK U 74r Yle ajm km ZEQ171}    \end{align} implying $q(0)\in H^{4-\nu}$ and thus, by \eqref{ZZ13}, $\partial_{t} v(0)\in H^{3-\nu}$. Now, based on \eqref{ZZ58}, evaluated at $t=0$, we have   \begin{align}   \begin{split}          \Delta \partial_{t} q(0)        &=              \partial_{t}^2a^{\lambda \alpha}(0) \partial_{\lambda}v_{\alpha}(0)         +  \partial_{t}a^{\lambda \alpha}(0) \partial_{t} \partial_{\lambda}v_{\alpha}(0)        -                               \partial_{t}a^{\lambda \alpha}(0)                                   \partial_{\alpha} \partial_{\lambda} q (0)       \\&\indeq       -                                  \partial_{\lambda}\partial_{t} {a^{\mu\lambda}}(0) 				 \partial_{\mu} q(0)       -                                  \partial_{t} {a^{\mu \lambda}}(0) 				 \partial_{\mu} \partial_{\lambda}q(0)     \in H^{2-\nu}    ,   \end{split}    \llabel{Jdi TGHO OaSt1N nl B 7Y7 h0y oWJ rEQ172}   \end{align} with the boundary conditions   \begin{equation}    \partial_{t} q(0)      =       \partial_{3} v_3(0) q(0)        - \Delta_{2} v^{3}(0)    \inon{on $\Gamma_1$}    \label{EQ173}   \end{equation} and   \begin{equation}     \partial_{3}\partial_{t} q(0)     = - \partial_{t} a^{\mu3}\partial_{\mu}q(0)     \inon{on $\Gamma_0$}     \llabel{y rVrT zHO8 2S7oub QA W x9d z2X YWEQ174}   \end{equation} which follow from \eqref{ZZ16} and \eqref{ZZ13} respectively. Note that  $\partial_{t} a(0) = -\nabla v(0)\in H^{2-\nu}$ and $\partial_{t}^2 a(0) = - \partial_{t} \nabla v(0) + \nabla v(0)\nabla v(0)\in H^{2-\nu}$, from where, using   \begin{equation}    \Delta_2 v^3(0) |_{\Gamma_1}    \in H^{1/2}(\Gamma_1),    \label{EQ175}   \end{equation} which in turn follows from $   v(0) |_{\Gamma_1}    \in H^{2.5}(\Gamma_1) $, we get $\partial_{t}q(0)\in H^{1}$, from where $\partial_{t}^2 v(0)\in L^2(\Omega)$. \gsdfgdsfgdsfg As pointed out above, the condition \eqref{EQ175} is not only sufficient, but also necessary for $\partial_{t}^2 v(0)\in L^{2}(\Omega)$. To show this, assume that $\partial_{t}^2 v(0)\in L^{2}(\Omega)$. Then $\partial_{t} q(0)\in H^{1}$  implying  $\partial_{t} q(0)|_{\Gamma_1}\in H^{1/2}(\Gamma_1)$. Using \eqref{EQ173}, we get that  \eqref{EQ175} holds. } \end{Remark} \gsdfgdsfgdsfg \gsdfgdsfgdsfg \startnewsection{A cofactor type cancellation}{sec05} In the energy estimate on $\partial_{t}^2v$, the highest order term is the one where all the derivatives fall on $a$. Thus we need to treat the term   \begin{gather}    T = \int_{0}^{t}\int_\Om q \DD a^{\mu\al} \DD \partial_\mu v_\alpha    ,    \label{EQ45}   \end{gather} where $    v=\partial_{t}\eta $. Here $\DD$ represents a differential operator, commuting with spatial and time derivatives. We shall use this with $\DD=\partial_{t}^2$. In this section, we rewrite \eqref{EQ45} using the cofactor form of $a$ and applying cross-integration by parts. \gsdfgdsfgdsfg First, note that we have  \begin{align}    a^{1\al} = \epsilon^{\al \lambda \tau} \partial_2 \eta_\lambda \partial_3 \eta_\tau, \,    a^{2\al} = -\epsilon^{\al \lambda \tau} \partial_1 \eta_\lambda \partial_3 \eta_\tau, \,     a^{3\al} = \epsilon^{\al \lambda \tau} \partial_1 \eta_\lambda \partial_2 \eta_\tau    \label{EQ46} \end{align} \gsdfgdsfgdsfg and thus, expanding in $\mu$ and using \eqref{EQ46},   \begin{align}    \begin{split}    T    &=      \int_{0}^{t}    \int_\Om q 
\epsilon^{\al \lambda \tau}\DD ( \partial_2 \eta_\lambda \partial_3 \eta_\tau) \DD \partial_1 v_\alpha     -     \int_{0}^{t}     \int_\Om q \epsilon^{\al \lambda \tau}\DD ( \partial_1 \eta_\lambda \partial_3 \eta_\tau) \DD \partial_2 v_\alpha     \\&\indeq     +     \int_{0}^{t}     \int_\Om q \epsilon^{\al \lambda \tau}\DD ( \partial_1 \eta_\lambda \partial_2 \eta_\tau) \DD \partial_3 v_\alpha    \end{split}    \llabel{B e5 Kf3A LsUF vqgtM2 O2 I dim rjZEQ58}   \end{align} from where   \begin{align}    \begin{split}    T  &=     \int_0^t \int_\Om q \ep^{\al\lambda \tau} \partial_{2}\DD\eta_\lambda \partial_3 \eta_\tau    \partial_1 \DD v_\al     +     \int_0^t \int_\Om q\ep^{\al\lambda \tau}\partial_2 \eta_\lambda  \partial_3 \DD \eta_\tau    \partial_1 \DD v_\al     \\&\indeq\indeq    -      \int_0^t \int_\Om q \ep^{\al\lambda \tau} \partial_1 \DD \eta_\lambda \partial_3 \eta_\tau    \partial_2 \DD v_\al     -     \int_0^t \int_\Om q \ep^{\al\lambda \tau}\partial_1 \eta_\lambda  \partial_3 \DD \eta_\tau    \partial_2 \DD v_\al     \\&\indeq\indeq    +    \int_0^t  \int_\Om q \ep^{\al\lambda \tau} \partial_1 \DD \eta_{\lambda} \partial_2 \eta_\tau    \partial_3 \DD v_\al     +    \int_0^t \int_\Om q \ep^{\al\lambda \tau}\partial_1 \eta_\lambda  \partial_2 \DD \eta_\tau    \partial_3 \DD v_\al     + L    \\&=     T_1 + \cdots + T_6 + L    \end{split}    \label{EQ02}   \end{align} where   \begin{align}    \begin{split}      L      &=     \int_{0}^{t}     \int_\Om q \epsilon^{\al \lambda \tau}     \Bigl(     \DD ( \partial_2 \eta_\lambda \partial_3 \eta_\tau)      -      \partial_2 \DD\eta_\lambda \partial_3 \eta_\tau     -      \partial_2 \eta_\lambda \partial_3 \DD\eta_\tau     \Bigr)     \DD \partial_1 v_\alpha     \\&\indeq     -     \int_{0}^{t}     \int_\Om q \epsilon^{\al \lambda \tau}     \Bigl(     \DD ( \partial_1 \eta_\lambda \partial_3 \eta_\tau)        -\partial_1 \DD\eta_\lambda \partial_3 \eta_\tau       -\partial_1 \eta_\lambda \partial_3 \DD\eta_\tau     \Bigr)     \DD \partial_2 v_\alpha     \\&\indeq     +     \int_{0}^{t}     \int_\Om q \epsilon^{\al \lambda \tau}     \Bigl(      \DD  (\partial_1 \eta_\lambda \partial_2 \eta_\tau)      -     \partial_1 \DD\eta_\lambda \partial_2 \eta_\tau     -     \partial_1 \eta_\lambda \partial_2 \DD\eta_\tau     \Bigr)     \DD \partial_3 v_\alpha    \end{split}    \label{EQ57}   \end{align} represents the sum of  the lower order terms that appear when we distribute $\DD$ on the product  $\epsilon^{\alpha\lambda\tau}\partial_{\alpha}\eta_{\lambda}\partial_{\beta}\eta_{\tau}$ and all derivatives do not fall on a single $\eta$. \gsdfgdsfgdsfg In order to proceed, we need for  $\DD$ to  contain at least one time derivative. Thus we now restrict our attention to   \begin{equation}    \DD     = \EE \partial_{t}    \label{EQ47}   \end{equation} where $\EE$ is a linear differential operator, for which we assume   \begin{equation}    [\EE,\partial_{t}]=0    \label{EQ48}   \end{equation} and   \begin{equation}    [\EE,\partial_{\alpha}]=0       \comma \alpha=1,2,3    .    \label{EQ56}   \end{equation} Further below we apply the resulting identity to  $\EE=\partial_{t}$. \gsdfgdsfgdsfg \gsdfgdsfgdsfg \gsdfgdsfgdsfg \gsdfgdsfgdsfg We group the leading terms in \eqref{EQ02} as    \begin{align}    \begin{split}    T_1 + T_3 &=     \int_0^t \int_\Om q \ep^{\al\lambda \tau} \partial_3 \eta_\tau\partial_{2}\DD\eta_\lambda     \partial_1 \DD v_\al     -      \int_0^t \int_\Om q \ep^{\al\lambda \tau} \partial_3 \eta_\tau\partial_1 \DD \eta_\lambda     \partial_2 \DD v_\al     \\    T_2 + T_5    &=    \int_0^t \int_\Om q\ep^{\al\lambda \tau}\partial_2 \eta_\lambda  \partial_3 \DD \eta_\tau    \partial_1 \DD v_\al     +    \int_0^t  \int_\Om q \ep^{\al\lambda \tau} \partial_2 \eta_\tau\partial_1 \DD \eta_{\lambda}     \partial_3 \DD v_\al     \\    T_4 + T_6    &=    -     \int_0^t \int_\Om q \ep^{\al\lambda \tau}\partial_1 \eta_\lambda  \partial_3 \DD \eta_\tau    \partial_2 \DD v_\al     +    \int_0^t \int_\Om q \ep^{\al\lambda \tau}\partial_1 \eta_\lambda  \partial_2 \DD \eta_\tau    \partial_3 \DD v_\al     .    \end{split}    \llabel{ 7RN 28 4KGY trVa WW4nTZ XV b RVo EQ115}   \end{align} Here we present the treatment of the sum $T_2+T_5$; the two other pairs are treated similarly (see below). Thus consider   \begin{align}    \begin{split}     T_2+T_5    &=    \int_{0}^{t}\int_{\Omega}     q \epsilon^{\alpha\lambda\tau}     \partial_{2}\eta_{\lambda}     \partial_{3}\EE v_{\tau}     \partial_{1}\EE\partial_{t}v_{\alpha}     +    \int_{0}^{t}\int_{\Omega}     q \epsilon^{\alpha\lambda\tau}     \partial_{2}\eta_{\tau}     \partial_{1}\EE v_{\lambda}     \partial_{3}\EE\partial_{t}v_{\alpha}    \\&    =    \int_{\Omega}     q \epsilon^{\alpha\lambda\tau}     \partial_{2}\eta_{\lambda}     \partial_{3}\EE v_{\tau}     \partial_{1}\EE v_{\alpha} \bigm|_{0}^{t}    -    \int_{0}^{t}\int_{\Omega}     q \epsilon^{\alpha\lambda\tau}     \partial_{2}\eta_{\lambda}     \partial_{t}\partial_{3}\EE v_{\tau}     \partial_{1}\EE v_{\alpha}     \\&\indeq     -    \int_{0}^{t}\int_{\Omega}     \partial_{t}(         q \epsilon^{\alpha\lambda\tau}          \partial_{2}\eta_{\lambda}            )     \partial_{3}\EE v_{\tau}     \partial_{1}\EE v_{\alpha}     +    \int_{0}^{t}\int_{\Omega}     q \epsilon^{\alpha\lambda\tau}     \partial_{2}\eta_{\tau}     \partial_{1}\EE v_{\lambda}     \partial_{3}\EE \partial_{t}v_{\alpha}     \\&     =     I_1     + I_2     + I_3     + I_4    ,    \end{split}    \label{EQ49}   \end{align} where we integrated by parts in $t$ in the first integral. By relabeling the indices, we may rewrite the fourth integral as $     I_4     =    \int_{0}^{t}\int_{\Omega}     q \epsilon^{\tau\alpha\lambda}     \partial_{2}\eta_{\lambda}     \partial_{1}\EE v_{\alpha}     \partial_{3}\EE\partial_{t}v_{\tau}     =    \int_{0}^{t}\int_{\Omega}     q \epsilon^{\alpha\lambda\tau}     \partial_{2}\eta_{\lambda}     \partial_{1}\EE v_{\alpha}     \partial_{3}\EE\partial_{t}v_{\tau} $ and  the last expression cancels with~$I_2$. Next we treat the first term on the far side of \eqref{EQ49} evaluated at $t$ by writing   \begin{align}    \begin{split}    I_1|_{t}    &=     \int_{\Omega}     q \epsilon^{\alpha\lambda\tau}     \partial_{2}\eta_{\lambda}     \partial_{3}\EE v_{\tau}     \partial_{1}\EE v_{\alpha}     =     \int_{\Omega}     q \epsilon^{\alpha2\tau}     \partial_{3}\EE v_{\tau}     \partial_{1}\EE v_{\alpha}     +     \int_{\Omega}     q \epsilon^{\alpha\lambda\tau}     \left(      \int_{0}^{t}        \partial_{2}v_{\lambda}     \right)     \partial_{3}\EE v_{\tau}     \partial_{1}\EE v_{\alpha}    ,    \end{split}    \llabel{Q77 hVL X6 K2kq FWFm aZnsF9 Ch p 8EQ52}   \end{align} where we used $\partial_{2}\eta_{\lambda}=\delta_{2\lambda}+\int_{0}^{t}\partial_{2} v_{\lambda}$ in the last step. \gsdfgdsfgdsfg \gsdfgdsfgdsfg \gsdfgdsfgdsfg \gsdfgdsfgdsfg Note that $T_1+T_3$ is obtained from $T_2+T_5$ by switching $x_2$ and $x_3$ and multiplying by $-1$, while $T_4+T_6$ is  obtained from $T_2+T_5$ by switching $x_1$ and $x_2$ and also multiplying by $-1$. \gsdfgdsfgdsfg We summarize the above derivation in the following statement. \gsdfgdsfgdsfg \cole \begin{Lemma} \label{L12} Consider the integral $       T = \int_\Om q \EE\partial_{t} a^{\mu\al} \EE\partial_{t} \partial_\mu v_\al    $, where $\EE$ is a differential operator which commutes with $\partial_{t}$  and $\partial_\alpha$, i.e., \eqref{EQ48} and \eqref{EQ56} hold. Then we have   \begin{align}    \begin{split}    T    &=     \int_{\Omega}     q \epsilon^{\alpha2\tau}     \partial_{3}\EE v_{\tau}     \partial_{1}\EE v_{\alpha}     \bigm|_{t}     -     \int_{\Omega}     q \epsilon^{\alpha3\tau}     \partial_{2}\EE v_{\tau}     \partial_{1}\EE v_{\alpha}     \bigm|_{t}     -     \int_{\Omega}     q \epsilon^{\alpha1\tau}     \partial_{3}\EE v_{\tau}     \partial_{2}\EE v_{\alpha}     \bigm|_{t}    \\&\indeq     -    \int_{0}^{t}\int_{\Omega}     \partial_{t}(         q \epsilon^{\alpha\lambda\tau}          \partial_{2}\eta_{\lambda}            )     \partial_{3}\EE v_{\tau}     \partial_{1}\EE v_{\alpha}    +    \int_{0}^{t}\int_{\Omega}     \partial_{t}(         q \epsilon^{\alpha\lambda\tau}          \partial_{3}\eta_{\lambda}            )     \partial_{2}\EE v_{\tau}     \partial_{1}\EE v_{\alpha}    \\&\indeq    +    \int_{0}^{t}\int_{\Omega}     \partial_{t}(         q \epsilon^{\alpha\lambda\tau}          \partial_{1}\eta_{\lambda}            )     \partial_{3}\EE v_{\tau}     \partial_{2}\EE v_{\alpha}    \\&\indeq    +     \int_{\Omega}     q \epsilon^{\alpha\lambda\tau}     \left(      \int_{0}^{t}     \partial_{2}v_{\lambda}     \right)     \partial_{3}\EE v_{\tau}     \partial_{1}\EE v_{\alpha}    -     \int_{\Omega}     q \epsilon^{\alpha\lambda\tau}     \left(      \int_{0}^{t}     \partial_{3}v_{\lambda}     \right)     \partial_{2}\EE v_{\tau}     \partial_{1}\EE v_{\alpha}    \\&\indeq    -     \int_{\Omega}     q \epsilon^{\alpha\lambda\tau}     \left(      \int_{0}^{t}     \partial_{1}v_{\lambda}     \right)     \partial_{3}\EE v_{\tau}     \partial_{2}\EE v_{\alpha}    \\&\indeq    -    \int_{\Omega}     q \epsilon^{\alpha2\tau}     \partial_{3}\EE v_{\tau}     \partial_{1}\EE v_{\alpha}      \bigm|_{0}^{}    +    \int_{\Omega}     q \epsilon^{\alpha3\tau}     \partial_{2}\EE v_{\tau}     \partial_{1}\EE v_{\alpha}      \bigm|_{0}^{}     +    \int_{\Omega}     q \epsilon^{\alpha1\tau}     \partial_{3}\EE v_{\tau}     \partial_{2}\EE v_{\alpha}      \bigm|_{0}^{}    + L    ,    \end{split}    \label{EQ55}   \end{align} where $L$ is given in \eqref{EQ57}. \end{Lemma} \colb \gsdfgdsfgdsfg \gsdfgdsfgdsfg It is helpful to expand the commutator term $L$ using \eqref{EQ47}. We thus have   \begin{align}    \begin{split}      L      &=     \int_\Om q \epsilon^{\al \lambda \tau}     \Bigl(     \EE  ( \partial_2 v_\lambda \partial_3 \eta_\tau)      -      \partial_2 \EE v_\lambda \partial_3 \eta_\tau     \Bigr)     \EE\partial_{t} \partial_1 v_\alpha     +     \int_\Om q \epsilon^{\al \lambda \tau}     \Bigl(     \EE ( \partial_2 \eta_\lambda \partial_3 v_\tau)      -      \partial_2 \eta_\lambda \partial_3 \EE v_\tau     \Bigr)     \EE\partial_{t} \partial_1 v_\alpha    \\&\indeq     -     \int_\Om q \epsilon^{\al \lambda \tau}     \Bigl(     \EE ( \partial_1 v_\lambda \partial_3 \eta_\tau)      -\partial_1 \EE v_\lambda \partial_3 \eta_\tau     \Bigr)     \EE\partial_{t} \partial_2 v_\alpha     -     \int_\Om q \epsilon^{\al \lambda \tau}     \Bigl(     \EE ( \partial_1 \eta_\lambda \partial_3 v_\tau)      -\partial_1 \eta_\lambda \partial_3 \EE v_\tau     \Bigr)     \EE\partial_{t} \partial_2 v_\alpha     \\&\indeq     +     \int_\Om q \epsilon^{\al \lambda \tau}     \Bigl(      \EE (\partial_1 v_\lambda \partial_2 \eta_\tau)      -     \partial_1 \EE v_\lambda \partial_2 \eta_\tau)      \Bigr)     \EE\partial_{t}\partial_3 v_\alpha     +     \int_\Om q \epsilon^{\al \lambda \tau}     \Bigl(       \EE (\partial_1 \eta_\lambda \partial_2 v_\tau)      -     \partial_1 \eta_\lambda \partial_2 \EE v_\tau     \Bigr)     \EE\partial_{t}\partial_3 v_\alpha     .    \end{split}    \llabel{Kx rsc SGP iS tVXB J3xZ cD5IP4 Fu EQ90}   \end{align} \gsdfgdsfgdsfg \startnewsection{A boundary integral estimate}{sec06} In Sections~\ref{sec07} and~\ref{sec08}, we obtain two integrals of the form $     \KK    =    - \int_{\Gamma_1}          \EE \partial_{t}( a^{\mu\alpha}q) \EE \partial_{t} v_{\alpha}N_{\mu}    =     \int_{\Gamma_1}          \EE \partial_{t}(\sqrtg \Delta_g\eta^{\alpha}) \EE \partial_{t} v_{\alpha} $ ($I_4$ and $J_4$ in \eqref{EQ91} and \eqref{EQ99} below, respectively), where $\EE$ is as in the previous section, i.e., a differential operator which commutes with spatial and time derivatives. Using the identity   \begin{align}    \begin{split}     {\partial}_t(\sqrt{g} \Delta_g \eta^\al )       &=       \partial_i        \Bigl(         \sqrt{g} g^{ij} (\de^\al_\lambda -g^{kl} \partial_k \eta^\al \partial_l \eta_\lambda)          {\partial}_t \partial_j \eta^\lambda          + \sqrt{g}(g^{ij} g^{kl} - g^{lj}g^{ik} ) \partial_j \eta^\al \partial_k\eta_\lambda         \partial_{t} \partial_l \eta^\lambda        \Bigr)    \end{split}    \llabel{9 Lcd TR2 Vwb cL DlGK 1ro3 EEyqEA EQ60}   \end{align} from \cite{DisconziKukavicaIncompressible} and $v=\partial_{t}\eta$, we get   \begin{align}    \begin{split}    \KK    &=    \int_{\Gamma_1}     \EE \partial_{t}v_{\alpha}     \EE        \partial_{i}        \Bigl(         \sqrt{g} g^{ij} (\de^\al_\lambda -g^{kl} \partial_k \eta^\al \partial_l \eta_\lambda)          \partial_j v^\lambda          + \sqrt{g}(g^{ij} g^{kl} - g^{lj}g^{ik} ) \partial_j \eta^\al \partial_k\eta_\lambda         \partial_l v^\lambda         \Bigr)     \\&     =     -    \int_{\Gamma_1}      \partial_{i} \EE \partial_{t}v_{\alpha}     \EE       \Bigl(         \sqrt{g} g^{ij} (\de^\al_\lambda -g^{kl} \partial_k \eta^\al \partial_l \eta_\lambda)           \partial_j v^\lambda         \Bigr)        \\&\indeq       -       \int_{\Gamma_1}         \partial_{i}   \EE \partial_{t}v_{\alpha}        \EE        \Bigl(         \sqrt{g}(g^{ij} g^{kl} - g^{lj}g^{ik} ) \partial_j \eta^\al \partial_k\eta_\lambda         \partial_l v^\lambda         \Bigr)     =     \KK_{1}     +    \KK_{2}    .   \end{split}    \label{EQ62}   \end{align} We denote by $\Pi$ the projection onto the normal of the moving boundary, given explicitly by   \begin{equation}    \Pi_{\lambda}^{\alpha}    =    \delta_{\lambda}^{\alpha}    - g^{kl}      \partial_{k}\eta^{\alpha}      \partial_{l}\eta_{\lambda}    .    \label{EQ63}   \end{equation} In Section~\ref{sec10}, we show how estimates on $\Pi v$ (and its time derivatives) yield estimates on the normal component of $v$ (and its time derivatives).    Using $\Pi$, we thus have   \begin{align}    \begin{split}    \KK_{1}    &=     -    \int_{\Gamma_1}     \EE       \Bigl(         \sqrt{g} g^{ij} \Pi_{\lambda}^{\alpha}           \partial_j v^\lambda         \Bigr)      \partial_{i} \EE \partial_{t}v_{\alpha}    \\&    =    -    \int_{\Gamma_1}         \sqrt{g} g^{ij} \Pi_{\lambda}^{\alpha}         \EE      \partial_j v^\lambda        \partial_{i} \EE \partial_{t}v_{\alpha}    -    \int_{\Gamma_1}     \Bigl(      \EE         (                \sqrt{g} g^{ij} \Pi_{\lambda}^{\alpha} \partial_j v^\lambda          )       -         \sqrt{g} g^{ij} \Pi_{\lambda}^{\alpha} \EE\partial_j  v^\lambda         \Bigr)      \partial_{i} \EE \partial_{t}v_{\alpha}    =    \KK_{11}    +   \KK_{12}    .    \end{split}    \llabel{zw 6 sKe Eg2 sFf jz MtrZ 9kbd xNw6EQ64}   \end{align} By $\Pi_{\lambda}^{\alpha}=\Pi_{\mu}^{\alpha}\Pi_{\lambda}^{\mu}$ (cf.~\cite{DisconziKukavicaIncompressible}), we may rewrite the first term as   \begin{align}    \begin{split}    \KK_{11}    &=    -    \int_{\Gamma_1}         \sqrt{g} g^{ij}           \Pi_{\lambda}^{\mu}               \partial_j \EE v^\lambda           \Pi_{\mu}^{\alpha}      \partial_{i} \EE \partial_{t}v_{\alpha}    \\&    =    -    \frac12    \frac{d}{dt}    \int_{\Gamma_1}         \sqrt{g} g^{ij}           \Pi_{\lambda}^{\mu}               \partial_j \EE v^\lambda           \Pi_{\mu}^{\alpha}      \partial_{i} \EE v_{\alpha}    +    \frac12    \int_{\Gamma_1}      \partial_{t}      (         \sqrt{g} g^{ij}           \Pi_{\lambda}^{\alpha}      )         \partial_{j}\EE v^\lambda            \partial_{i} \EE v_{\alpha}     \\&    =    -    \frac12    \frac{d}{dt}    \int_{\Gamma_1}         \sqrt{g} g^{ij}           \Pi_{\lambda}^{\mu}               \partial_j \EE v^\lambda           \Pi_{\mu}^{\alpha}      \partial_{i} \EE v_{\alpha}    +    \frac12    \int_{\Gamma_1}      \partial_{t}      (         \sqrt{g} g^{ij}       )      \Pi_{\lambda}^{\alpha}         \partial_{j}\EE v^\lambda           \partial_{i} \EE v_{\alpha}    \\&\indeq       +    \frac12    \int_{\Gamma_1}         \sqrt{g} g^{ij}       \partial_{t}      (          \Pi_{\lambda}^{\alpha}      )               \partial_j \EE v^\lambda           \partial_{i} \EE v_{\alpha}    .    \end{split}    \llabel{6c xf t lzD GZh xQA WQ KkSX jqmm rEQ65}   \end{align} We thus obtain   \begin{align}    \begin{split}    \KK_{11}     &\les        -    \frac12    \frac{d}{dt}    \int_{\Gamma_1}         \sqrt{g} g^{ij}           \Pi_{\lambda}^{\mu}               \partial_j \EE v^\lambda           \Pi_{\mu}^{\alpha}     \partial_{i} \EE v_{\alpha}     +     P(       \Vert\eta\Vert_{H^{2.5+\delta_0}}      )       \Vert v\Vert_{H^{2.5+\delta_0}}       \Vert \Pi \barpar\EE v \Vert_{L^2(\Gamma_1)}^2     \\&\indeq     +     P(       \Vert\eta\Vert_{H^{2.5+\delta_0}}      )      \Vert v\Vert_{H^{2.5+\delta_0}}      \Vert \EE v\Vert_{H^{1.5}}^2    ,    \end{split}    \llabel{EpNuG 6P y loq 8hH lSf Ma LXm5 RzEEQ66}   \end{align} where   \begin{equation}    \barpar=\nabla_2=(\partial_{1},\partial_{2})       .    \label{EQ37}   \end{equation} Next, we consider the second term in \eqref{EQ62}. We have   \begin{align}    \begin{split}    \KK_{2}    &=      -       \int_{\Gamma_1}         \sqrt{g}(g^{ij} g^{kl} - g^{lj}g^{ik} ) \partial_j \eta^\al \partial_k\eta_\lambda           \partial_{l}\EE v^{\lambda}         \partial_{i}\EE\partial_{t} v_{\alpha}     \\&\indeq      -       \int_{\Gamma_1}        \biggl(         \EE         \bigl(          \sqrt{g}(g^{ij} g^{kl} - g^{lj}g^{ik} ) \partial_j \eta^\al \partial_k\eta_\lambda  \partial_l v^\lambda          \bigr)         \\&\indeq\indeq\indeq\indeq         -          \sqrt{g}(g^{ij} g^{kl} - g^{lj}g^{ik} ) \partial_j \eta^\al \partial_k\eta_\lambda          \partial_{l}\EE v^{\lambda}        \biggr)         \partial_{i}\EE\partial_{t} v_{\alpha}     =     \KK_{21}     +    \KK_{22}    .    \end{split}    \llabel{X W4Y1Bq ib 3 UOh Yw9 5h6 f6 o8kw EQ67}   \end{align} As in \cite{CoutandShkollerFreeBoundary} (cf.~also~\cite{DisconziKukavicaIncompressible}),  we may write $     \KK_{21}     = -      \int_{\Gamma_{1}}         \sqrtg^{-1}           \bigl(                  \partial_{t} \det A^{1} + \det A^{2}  + \det A^{3}           \bigr) $, where   \begin{align*}     A^{1} =      \begin{pmatrix}     \partial_{1} \eta_{\mu} \partial_{1} \EE v^{\mu} & \partial_{1} \eta_{\mu} \partial_{2}      \EE v^{\mu}      \\     \partial_{2} \eta_{\mu} \partial_{1} \EE v^{\mu} & \partial_{2} \eta_{\mu}\partial_{2} \EE v^{\mu}     \end{pmatrix},      A^{2} =      \begin{pmatrix}     \partial_{1} v_{\mu}\partial_{1} \EE v^{\mu} & \partial_{1} \eta_{\mu} \partial_{2} \EE v^{\mu} \\     \partial_{2} v_{\mu} \partial_{1} \EE v^{\mu} & \partial_{2} \eta_{\mu} \partial_{2} \EE v^{\mu}     \end{pmatrix},      A^{3} =      \begin{pmatrix}     \partial_{1} \eta_{\mu} \partial_{1} \EE v^{\mu} & \partial_{1}  v_{\mu}  \partial_{2} \EE v^{\mu} \\     \partial_{2} \eta_{\mu} \partial_{1} \EE v^{\mu} & \partial_{2} v_{\mu} \partial_{2} \EE v^{\mu}     \end{pmatrix}     .   \end{align*} Therefore,   \begin{align}   \begin{split}    \KK_{21}    &=           -      \int_{\Gamma_{1}}        
\partial_{t}       \left(        \frac{1}{\sqrt{g}} \det A^{1}       \right)       +      \int_{\Gamma_{1}}        \partial_{t}             \left(        \frac{1}{\sqrt{g}}       \right)       \det A^{1}       -  \int_{\Gamma_{1}} \frac{1}{\sqrt{g}} \det A^{2}       -        \int_{\Gamma_{1}}          \frac{1}{\sqrt{g}} \det A^{3}    \\&    = \KK_{211}       + \KK_{212}      + \KK_{213}      + \KK_{214}              .   \end{split}    \llabel{6frZ wg6fIy XP n ae1 TQJ Mt2 TT fWEQ71}   \end{align} Note that $    \left\Vert     \partial_{t}     \left(     \sqrtg^{-1}     \right)    \right\Vert_{L^\infty(\Gamma_1)}    \les    P(\Vert \eta\Vert_{H^{2.5+\delta_0}})    \Vert v\Vert_{H^{2.5+\delta_0}} $ Since also $|\det A^{1}|\les |\barpar \eta|^2  (\EE\barpar v)^2$, we get   \begin{equation}     \KK_{212}    \les     P(\Vert \eta\Vert_{H^{2.5+\delta_0}})          \Vert v\Vert_{H^{2.5+\delta_0}}     \Vert \EE \barpar v\Vert_{L^2(\Gamma_1)}^2     \les     P(\Vert \eta\Vert_{H^{2.5+\delta_0}})          \Vert v\Vert_{H^{2.5+\delta_0}}     \Vert \EE v\Vert_{H^{1.5}}^2    .    \llabel{Wf jJrX ilpYGr Ul Q 4uM 7Ds p0r VgEQ73}   \end{equation} Similarly,   \begin{equation}    \left|     \int_{\Gamma_1}     \frac{1}{\sqrt{g}}       (\det A^{2}+\det A^{3})    \right|    \les     P(\Vert \eta\Vert_{H^{2.5+\delta_0}})          \Vert \EE \barpar v\Vert_{L^2(\Gamma_1)}^2     \les     P(\Vert \eta\Vert_{H^{2.5+\delta_0}})          \Vert \EE v\Vert_{H^{1.5}}^2    .    \llabel{ 3gIE mQOz TFh9LA KO 8 csQ u6m h25EQ74}   \end{equation} The term $\KK_{211}$ requires more care since if  we bound $\det A^{1}$ as above, we obtain the term $\Vert\EE\barpar v\Vert_{L^2(\Gamma_1)}^2$ which  cannot be absorbed into the left side. Instead we integrate by parts and obtain   \begin{align}    \begin{split}    \int_{\Gamma_1}    \frac{1}{\sqrtg}    \det A^{1}    &=    \int_{\Gamma_1}    \frac{1}{\sqrtg}    \bigl(     \partial_{1}\eta_{\mu}\partial_{2}\eta_{\lambda}     \partial_{1}\EE v^{\mu}     \partial_{2}\EE v^{\lambda}     -     \partial_{1}\eta_{\mu}\partial_{2}\eta_{\lambda}     \partial_{2} \EE v^{\mu}     \partial_{1} \EE v^{\lambda}    \bigr)    \\&    =    \int_{\Gamma_1}    \frac{1}{\sqrtg}    \bigl(     -     \partial_{1}\eta_{\mu}\partial_{2}\eta_{\lambda}     \EE v^{\mu}     \partial_{1}\partial_{2}\EE v^{\lambda}     +     \partial_{1}\eta_{\mu}\partial_{2}\eta_{\lambda}     \EE v^{\mu}     \partial_{2}\partial_{1} \EE v^{\lambda}    \bigr)    \\&\indeq    -    \int_{\Gamma_1}      Q_{\mu\lambda}^{i}(\bar\partial\eta,\bar\partial^2\eta)      \EE v^{\mu}      \partial_{i} \EE v^{\lambda}    \\&    =    -    \int_{\Gamma_1}      Q_{\mu\lambda}^{i}(\bar\partial\eta,\bar\partial^2\eta)      \EE v^{\mu}      \partial_{i} \EE v^{\lambda}     ,    \end{split}    \llabel{ r8 WqRI DZWg SYkWDu lL 8 Gpt ZW1 EQ75}   \end{align} where $Q_{\mu\lambda}^{i}(\bar\partial\eta,\bar\partial^2\eta)$ is a rational function, which is \emph{linear} in $\bar\partial^2\eta$ and can thus be written as $    Q_{\mu\lambda}^{i}(\bar\partial\eta,\bar\partial^2\eta)    =    \tilde Q_{\mu\lambda}^{i}(\bar\partial\eta)\bar\partial^2\eta $ with $\tilde Q$ a rational function. Hence, $    \KK_{211}    =     (   \fractext{d}{dt})    \int_{\Gamma_1}    \tilde Q_{\mu\lambda}^{i}(\bar\partial\eta)\bar\partial^2\eta      \EE v^{\mu}      \partial_{i} \EE v^{\lambda} $, and thus   \begin{equation}    \KK_{21}    \les    \frac{d}{dt}    \int_{\Gamma_1}    \tilde Q_{\mu\lambda}^{i}(\bar\partial\eta)\bar\partial^2\eta      \EE v^{\mu}      \partial_{i} \EE v^{\lambda}     +     P(\Vert \eta\Vert_{H^{2.5+\delta_0}})          (\Vert v\Vert_{H^{2.5+\delta_0}}+1)     \Vert \EE v\Vert_{H^{1.5}}^2    .    \llabel{0Gd SY FUXL zyQZ hVZMn9 am P 9aE WEQ78}   \end{equation} \gsdfgdsfgdsfg We summarize the above derivations in the following statement. \gsdfgdsfgdsfg \gsdfgdsfgdsfg \gsdfgdsfgdsfg \cole \begin{Lemma} \label{L14} Consider the integral $       \KK       =    - \int_{\Gamma_1}          \EE \partial_{t}( a^{\mu\alpha}q) \EE \partial_{t} v_{\alpha}N_{\mu} $, where $\EE$ is a differential operator which commutes with $\partial_{t}$  and $\partial_\alpha$, i.e., \eqref{EQ48} and \eqref{EQ56} hold. Then we have   \begin{align}    \begin{split} &      - \int_{\Gamma_1}          \EE \partial_{t}( a^{\mu\alpha}q) \EE \partial_{t} v_{\alpha}N_{\mu}   \\&\indeq      \les         -    \frac12    \frac{d}{dt}    \int_{\Gamma_1}         \sqrt{g} g^{ij}           \Pi_{\lambda}^{\mu}         \EE      \partial_j v^\lambda           \Pi_{\mu}^{\alpha}      \EE\partial_{i} v_{\alpha}    +    \frac{d}{dt}    \int_{\Gamma_1}    \tilde Q_{\mu\lambda}^{i}(\bar\partial\eta)\bar\partial^2\eta      \EE v^{\mu}      \partial_{i} \EE v^{\lambda}    \\&\indeq\indeq    -    \int_{\Gamma_1}     \Bigl(      \EE         (                \sqrt{g} g^{ij} \Pi_{\lambda}^{\alpha} \partial_j v^\lambda          )       -         \sqrt{g} g^{ij} \Pi_{\lambda}^{\alpha} \EE\partial_j  v^\lambda         \Bigr)      \partial_{i} \EE \partial_{t}v_{\alpha}    \\&\indeq\indeq      -       \int_{\Gamma_1}        \biggl(         \EE         \bigl(          \sqrt{g}(g^{ij} g^{kl} - g^{lj}g^{ik} )          \partial_{j}\eta^{\alpha}                  \partial_{k}\eta_{\lambda}          \partial_{l} v^{\lambda}         \bigr)         \\&\indeq\indeq\indeq\indeq         -          \sqrt{g}(g^{ij} g^{kl} - g^{lj}g^{ik} )          \partial_{j}\eta^{\alpha}                  \partial_{k}\eta_{\lambda}          \partial_{l}\EE v^{\lambda}        \biggr)         \partial_{i}\EE\partial_{t} v_{\alpha}    \\&\indeq\indeq     +     P(       \Vert\eta\Vert_{H^{2.5+\delta_0}}      )      (     \Vert v\Vert_{H^{2.5+\delta_0}}+1)      \Vert \EE v\Vert_{H^{1.5}}^2    .    \end{split}    \label{EQ77}   \end{align} \end{Lemma} \colb \gsdfgdsfgdsfg Note that the third and the fourth terms are of commutator type. Since it is needed in the next two sections, we show here an estimate for  the time integral of the second term on the right side of \eqref{EQ77}. We have   \begin{align}     \begin{split}     &     \int_{\Gamma_1}    \tilde Q_{\mu\lambda}^{i}(\bar\partial\eta)\bar\partial^2\eta      \EE v^{\mu}      \partial_{i} \EE v^{\lambda} \restr_{t}      \les      \Vert    \tilde Q_{\mu\lambda}^{i}(\bar\partial\eta)\bar\partial^2\eta    \Vert_{H^{0.5-\nu}(\Gamma_1)}      \Vert          \EE v^{\mu} \Vert_{H^{0.5+\nu}(\Gamma_1)}      \Vert          \partial_{i} \EE v^{\lambda}    \Vert_{L^2(\Gamma_1)}     \\&\indeq     \les      \Vert        \tilde Q(\bar\partial\eta)\bar\partial^2\eta \Vert_{H^{0.5-\nu}(\Gamma_1)}      \Vert          \EE v \Vert_{H^{0.5+\nu}(\Gamma_1)}      \Vert          \barpar \EE v    \Vert_{L^2(\Gamma_1)}     \\&\indeq     \les      \bigl(      \Vert       \tilde Q(\bar\partial\eta)      \Vert_{L^\infty}      +      \Vert       \tilde Q(\bar\partial\eta)      \Vert_{H^{1}(\Gamma_1)}      \bigr)     \Vert     \barpar^2\eta    \Vert_{H^{0.5-\nu}(\Gamma_1)}      \Vert          \EE v \Vert_{H^{0.5+\nu}(\Gamma_1)}      \Vert          \partial_{i} \EE v    \Vert_{L^2(\Gamma_1)}     ,    \end{split}    \label{EQ42a}   \end{align} where we used   \begin{equation}    \Vert A B\Vert_{H^{0.5-\nu}(\Gamma_1)}    \les    (     \Vert A \Vert_{L^\infty(\Gamma_1)}     +     \Vert A \Vert_{H^1(\Gamma_1)}    )    \Vert B \Vert_{H^{0.5-\nu}(\Gamma_1)}    \label{EQ213}   \end{equation} in the last inequality. Note that \eqref{EQ213} follows by a simple application of the Kato-Ponce fractional chain rule. Using that  $H^{1+\delta_0}(\Gamma_1)$ is an algebra, we obtain from \eqref{EQ42a}   \begin{align}     \begin{split}     &     \int_{\Gamma_1}    \tilde Q_{\mu\lambda}^{i}(\bar\partial\eta)\bar\partial^2\eta      \EE v^{\mu}      \partial_{i} \EE v^{\lambda} \restr_{t}     \les     P(\Vert \barpar\eta\Vert_{H^{1+\delta_0}(\Gamma_1)})     \Vert \barpar^2\eta\Vert_{H^{0.5-\nu}(\Gamma_1)}      \Vert          \EE v \Vert_{H^{1+\nu}}      \Vert          \EE v    \Vert_{H^{1.5}}     \\&\indeq     \les     P(\Vert \eta\Vert_{H^{2.5+\delta_0}})     \Vert \eta\Vert_{H^{3-\nu}}      \Vert          \EE v \Vert_{L^2}^{(1-2\nu)/3}      \Vert          \EE v    \Vert_{H^{1.5}}^{(5+2\nu)/3}    \\&\indeq     \les     P(\Vert \eta\Vert_{H^{3-\nu}})       \left(      \Vert          \EE v(0) \Vert_{L^2}^2       +        \int_{0}^{t}      \Vert          \EE \partial_{t}v    \Vert_{L^2}^2      \right)      +      \epsilon_0      \Vert          \EE v    \Vert_{H^{1.5}}^2     ,    \end{split}    \llabel{zk au0 6d ZghM ym3R jfdePG ln 8 s7EQ42}   \end{align} from where,  by Lemma~\ref{L01}(i),   \begin{equation}     \int_{\Gamma_1}    \tilde Q_{\mu\lambda}^{i}(\bar\partial\eta)\bar\partial^2\eta      \EE v      \partial_{i} \EE v      \les      \Vert          \EE v(0) \Vert_{L^2}^2       +        \int_{0}^{t}      \Vert          \EE \partial_{t} v    \Vert_{L^2}^2      +      \epsilon_0      \Vert          \EE v    \Vert_{H^{1.5}}^2    .    \label{EQ68}   \end{equation} \gsdfgdsfgdsfg   \gsdfgdsfgdsfg \gsdfgdsfgdsfg \gsdfgdsfgdsfg \gsdfgdsfgdsfg \gsdfgdsfgdsfg \gsdfgdsfgdsfg \gsdfgdsfgdsfg \gsdfgdsfgdsfg \gsdfgdsfgdsfg \gsdfgdsfgdsfg \startnewsection{The tangential estimate on $\partial_{t} v$}{sec07} In this and the next sections,  we perform energy estimates on the quantity $    \Vert \EE \partial_{t} v\Vert_{L^2} $ with $\EE=\tpar^{1-\nu/2}$ and $\EE=\partial_{t}$, respectively, where $    \tpar    = (I-\Delta_2)^{1/2} $ with $\Delta_2=\partial_{1}^2+\partial_{2}^2$ denoting the horizontal Laplacian. In both cases, we apply $\EE\partial_{t}$ to \eqref{ZZ13}, multiply the resulting equation with $\EE\partial_{t} v$, and integrate, obtaining   \begin{align}    \begin{split}    \frac12    \frac{d}{dt}    \Vert\EE\partial_{t}v\Vert_{L^2}^2    &=    - \into\EE\partial_{t}(a^{\mu\alpha}\partial_{\mu}q)\EE \partial_{t}v_{\alpha}    =       \into\EE\partial_{t}(a^{\mu\alpha}q)\EE \partial_{t}\partial_{\mu}v_{\alpha}    - \intu\EE\partial_{t}(a^{\mu\alpha}q)\EE \partial_{t}v_{\alpha} N_{\mu}    \end{split}    \label{EQ88}   \end{align} since $   - \intl\EE\partial_{t}(a^{\mu\alpha}q)\EE \partial_{t}v_{\alpha} N_{\mu}=0$ by \eqref{ZZ17}  and $a^{31}=a^{32}=0$ on $\Gamma_0$ due to $a^{31}=\partial_{1}\eta^{2}\partial_{2}\eta^{3}-\partial_{2}\eta^{2}\partial_{1}\eta^{3}$ and    $a^{32}=\partial_{2}\eta^{1}\partial_{1}\eta^{3}-\partial_{1}\eta^{1}\partial_{2}\eta^{3}$. \gsdfgdsfgdsfg \gsdfgdsfgdsfg \gsdfgdsfgdsfg In this section, we set $    \EE    = \tpar^{1-\nu/2} $. The most important assertion in the next statement is that it provides control of $     \Vert v^3\Vert_{H^{2-\nu/2}(\Gamma_1)} $ needed further below. \gsdfgdsfgdsfg \cole \begin{Lemma} \label{L05} The Lagrangian velocity $v$ and its derivative $\partial_{t}v$ satisfy   \begin{align}   \begin{split}     & \Vert \tpar^{1-\nu/2} \partial_{t} v  \Vert_{L^{2}}^{2}     + \Vert \Pi \barpar\tpar^{1-\nu/2}v \Vert_{L^{2}(\Gamma_{1})}^{2}    \les       P_0    +    \epsilon_0 \Vert  v\Vert_{H^{2.5-\nu/2}}^2    +    \int_{0}^{t} P    ,   \end{split}    \label{ZZ24}   \end{align} where $P$ is a polynomial in $\Vert v \Vert_{H^{3-\nu}}$,  $ \Vert \partial_{t}v \Vert_{H^{1.5}}$,  $\Vert q \Vert_{H^{2.5-\nu/2}}$, and $ \Vert \partial_{t}q \Vert_{H^{1}}$, while $P_0$ is a polynomial in $\Vert v_0\Vert_{H^{3-\nu}}$ and $\Vert \partial_{t}v(0)\Vert_{H^{1.5}}$. \end{Lemma} \colb \gsdfgdsfgdsfg Using the notation \eqref{EQ13}--\eqref{EQ00},  the inequality \eqref{ZZ24} implies   \begin{align}   \begin{split}    \Vert \Pi \barpar\tpar^{1-\nu/2}v \Vert_{L^{2}(\Gamma_{1})}^2    &    \les        \epsilon_0 E_0^2     +     P_0     +\int_{0}^{t} P     ,   \end{split}    \label{EQ135}   \end{align} where, as mentioned above, $\epsilon_0>0$ denotes an arbitrarily small constant. \gsdfgdsfgdsfg \gsdfgdsfgdsfg \gsdfgdsfgdsfg \begin{proof}[Proof of Lemma~\ref{L05}] From \eqref{EQ88}, we have the equation   \begin{align}    \begin{split}    \frac12    \frac{d}{dt}    \Vert\tpar^{1-\nu/2}\partial_{t}v\Vert_{L^2}^2    = I_1 + I_2 + I_3 + I_4    ,    \end{split}    \llabel{x HYC IV9 Hw Ka6v EjH5 J8Ipr7 Nk CEQ89b}   \end{align} where   \begin{align}    \begin{split}    I_1    &=    \into\tpar^{1-\nu/2}(\partial_{t}a^{\mu\alpha}q)\tpar^{1-\nu/2} \partial_{t}\partial_{\mu}v_{\alpha}    \comma\ \ \ \ \ \ \ \ \ \ \     I_2    =     \into a^{\mu\alpha}\tpar^{2-\nu}\partial_{t}q\partial_{t}\partial_{\mu}v_{\alpha}    \\    I_3    &=       \into    \Bigl(      \tpar^{2-\nu}(a^{\mu\alpha}\partial_{t}q)      - a^{\mu\alpha}\tpar^{2-\nu}\partial_{t}q    \Bigr)    \partial_{t}\partial_{\mu}v_{\alpha}    \comma       I_4    =    - \int_{\Gamma_1}          \tpar^{1-\nu/2} \partial_{t}( a^{\mu\alpha}q) \tpar^{1-\nu/2} \partial_{t} v_{\alpha}N_{\mu}    .    \end{split}    \label{EQ91}   \end{align} Using multiplicative Sobolev inequalities, we have   \begin{align}    \begin{split}     I_1     &=    \into\tpar^{1.5-\nu}(\partial_{t}a^{\mu\alpha}q)\tpar^{0.5} \partial_{t}\partial_{\mu}v_{\alpha}     \les     \sum_{\mu,\alpha}     \bigl\Vert      \tpar^{1.5-\nu}       (\partial_{t}a^{\mu\alpha} q)     \bigr\Vert_{L^2}     \Vert\partial_{t}v\Vert_{H^{1.5}}     \\&     \les     \sum_{\mu,\alpha}     \bigl\Vert       \partial_{t}a^{\mu\alpha} q     \bigr\Vert_{H^{1.5-\nu}}     \Vert\partial_{t}v\Vert_{H^{1.5}}     \les     \Vert       \partial_{t}a     \Vert_{H^{2-\nu}}     \Vert q\Vert_{H^{1}}     \Vert\partial_{t}v\Vert_{H^{1.5}}     +     \Vert       \partial_{t}a     \Vert_{H^{1}}     \Vert q\Vert_{H^{2-\nu}}     \Vert\partial_{t}v\Vert_{H^{1.5}}     \les     P    \end{split}    \llabel{ xWR 84T Wnq s0 fsiP qGgs Id1fs5 3EQ97}   \end{align} using $L^2$ based Kato-Ponce type estimates (fractional product rule), as in \cite{KukavicaTuffaha-Free2dEuler,KukavicaTuffaha-RegularityFreeEuler}. For the second term in \eqref{EQ91}, we use the divergence-free condition \eqref{ZZ14} to write   \begin{align}    \begin{split}    I_2    &=     - \into \partial_{t}a^{\mu\alpha}\tpar^{2-\nu}\partial_{t}q\partial_{\mu}v_{\alpha}       =     - \into          \partial_{t}a^{\mu\alpha}         \partial_{\mu}v_{\alpha}            \tpar^{1-\nu}\tpar\partial_{t}q    =     -    \int \tpar^{1-\nu}(\partial_{t}a^{\mu\alpha}\partial_{\mu}v_{\alpha})   \tpar\partial_{t}q    \\&    \les    \Vert \tpar^{1-\nu}(\partial_{t}a^{\mu\alpha}\partial_{\mu}v_{\alpha})\Vert_{L^2}    \Vert \partial_{t} q\Vert_{H^{1}}    \\&    \les    \Vert \tpar^{1-\nu}\partial_{t}a^{\mu\alpha}\Vert_{H^{1}}    \Vert \partial_{\mu}v_{\alpha}\Vert_{H^{0.5}}    \Vert\partial_{t}q\Vert_{H^{1}}    +    \Vert \partial_{t}a^{\mu\alpha}\Vert_{H^{1}}    \Vert \tpar^{1-\nu}\partial_{\mu}v_{\alpha}\Vert_{H^{0.5}}    \Vert\partial_{t}q\Vert_{H^{1}}    \les    P    \comma     \end{split}    \llabel{A T 71q RIc zPX 77 Si23 GirL 9MQZ4EQ94}   \end{align} again using the fractional chain rule. The last interior term $I_3$ is estimated as   \begin{align}    \begin{split}    I_3    &\les       \Vert \tpar^{2-\nu} (a^{\mu\alpha}\partial_t q)                         -a^{\mu\alpha}\tpar^{2-\nu}\partial_t q       \Vert_{L^{3/2}}      \Vert \partial_{t}\partial_{\mu}v\Vert_{L^{3}}    \les    \Vert a \Vert_{H^{2-\nu}}    \Vert \partial_{t} q \Vert_{H^{1}}    \Vert \partial_{t} v\Vert_{H^{1.5}}    \les P    \end{split}    \label{EQ116}   \end{align} For completeness, we show the validity of the second inequality above as the Kato-Ponce inequality  can only be applied in the first two variables.  We do so by the successive integration. For any fixed $x_3\in(0,1)$, we employ the Kato-Ponce inequality to obtain   \begin{align}    \begin{split}      &       \Vert \tpar^{2-\nu} (a^{\mu\alpha}\partial_t q)                         -a^{\mu\alpha}\tpar^{2-\nu}\partial_t q       \Vert_{L_{x_1,x_2}^{3/2}}     \les     \Vert \tpar^{2-\nu}a\Vert_{L_{x_1,x_2}^2}     \Vert \partial_{t} q\Vert_{L_{x_1,x_2}^{6}}     +     \Vert \tpar a\Vert_{L_{x_1,x_2}^{6/(1+2\nu)}}     \Vert \tpar^{1-\nu}\partial_{t}q\Vert_{L_{x_1,x_2}^{6/(3-2\nu)}}    \end{split}    \llabel{F pi g dru NYt h1K 4M Zilv rRk6 B4EQ182}   \end{align} (cf.~\cite{KP,GO,KPV91,KWZ}), where $L^{p}_{x_1,x_2}$ denotes the $L^{p}$ norm in $(x_1,x_2)$. Taking the $L_{x_3}^{3/2}$ norm of both sides  and applying the H\"older inequality in the $x_3$ variable gives   \begin{align}    \begin{split}      &       \Vert \tpar^{2-\nu} (a^{\mu\alpha}\partial_t q)                         -a^{\mu\alpha}\tpar^{2-\nu}\partial_t q       \Vert_{L^{3/2}}     \\&\indeq     \les      \bigl\Vert         \Vert \tpar^{2-\nu}a\Vert_{L_{x_1,x_2}^2}      \Vert \partial_{t} q\Vert_{L_{x_1,x_2}^{6}}     \bigr\Vert_{L_{x_3}^{3/2}}     +      \bigl\Vert      \Vert \tpar a\Vert_{L_{x_1,x_2}^{6/(1+2\nu)}}      \Vert \tpar^{1-\nu}\partial_{t}q\Vert_{L_{x_1,x_2}^{6/(3-2\nu)}}     \bigr\Vert_{L_{x_3}^{3/2}}     \\&\indeq     \les     \Vert \tpar^{2-\nu}a\Vert_{L^2}     \Vert \partial_{t} q\Vert_{L^{6}}     +     \Vert \tpar a\Vert_{L^{6/(1+2\nu)}}     \Vert \tpar^{1-\nu}\partial_{t}q\Vert_{L^{6/(3-2\nu)}}     \\&\indeq     \les     \Vert \tpar^{2-\nu}a\Vert_{L^2}     \Vert \partial_{t} q\Vert_{L^{6}}     +     \Vert a\Vert_{H^{2-\nu}}     \Vert \partial_{t}q\Vert_{H^{1}}    \end{split}    \llabel{W5B8 Id 3 Xq9 nhx EN4 P6 ipZl a2UQEQ183}   \end{align} where we used the Sobolev inequality in the last step. \gsdfgdsfgdsfg Finally, we use  Lemma~\ref{L14}  with $    \EE    = \tpar^{1-\nu/2} $ to write   \begin{align}    \begin{split}      I_4      &\les \emb{         -    \frac12    \frac{d}{dt}    \int_{\Gamma_1}         \sqrt{g} g^{ij}           \Pi_{\lambda}^{\mu}         \tpar^{1-\nu/2}      \partial_j v^\lambda           \Pi_{\mu}^{\alpha}      \tpar^{1-\nu/2}\partial_{i} v_{\alpha} }    + \emb{    \frac{d}{dt}    \int_{\Gamma_1}    \tilde Q_{\mu\lambda}^{i}(\bar\partial\eta)\bar\partial^2\eta      \tpar^{1-\nu/2} v^{\mu}      \partial_{i} \tpar^{1-\nu/2} v^{\lambda} }    \\&\indeq\indeq    -    \int_{\Gamma_1}     \Bigl(      \tpar^{1-\nu/2}        (                \sqrt{g} g^{ij} \Pi_{\lambda}^{\alpha} \partial_j v^\lambda          )       -         \sqrt{g} g^{ij} \Pi_{\lambda}^{\alpha} \tpar^{1-\nu/2}\partial_j  v^\lambda         \Bigr)      \partial_{i} \tpar^{1-\nu/2} \partial_{t}v_{\alpha}    \\&\indeq\indeq      -       \int_{\Gamma_1}        \biggl(         \tpar^{1-\nu/2}         \bigl(          \sqrt{g}(g^{ij} g^{kl} - g^{lj}g^{ik} )           \partial_{j}\eta^{\alpha}                  \partial_{k}\eta_{\lambda}          \partial_{l} v^{\lambda}         \bigr)         \\&\indeq\indeq\indeq\indeq         -          \sqrt{g}(g^{ij} g^{kl} - g^{lj}g^{ik} )          \partial_{j}\eta^{\alpha}                  \partial_{k}\eta_{\lambda}          \partial_{l}\tpar^{1-\nu/2} v^{\lambda}        \biggr)         \partial_{i}\tpar^{1-\nu/2}\partial_{t} v_{\alpha}    \\&\indeq\indeq     +     P(       \Vert\eta\Vert_{H^{2.5+\delta_0}}      )      (     \Vert v\Vert_{H^{2.5+\delta_0}}+1)      \Vert v\Vert_{H^{2.5-\nu/2}}^2    \end{split}    \label{EQ19}   \end{align} where, recall,  $\delta_0>0$ is arbitrarily small. The first term in \eqref{EQ19} leads to the second term of \eqref{ZZ24}. Namely, using   \begin{equation}    \sqrtg g^{ij}\xi_{i}\xi_{j}    \geq    \frac{1}{C} |\xi|^2    \comma \xi\in{\mathbb R}^{2}    \label{EQ227}   \end{equation} for $t$ as in Lemma~\ref{L01}(iv), we get   \begin{align}    
\begin{split}    &    \frac12    \int_{\Gamma_1}         \sqrt{g} g^{ij}           \Pi_{\lambda}^{\mu}         \tpar^{1-\nu/2}      \partial_j v^\lambda           \Pi_{\mu}^{\alpha}      \tpar^{1-\nu/2}\partial_{i} v_{\alpha}     \geq     \frac{1}{C}    \int_{\Gamma_1}          \Pi_{\lambda}^{\mu}         \tpar^{1-\nu/2}    \partial_i v^\lambda       \Pi_{\mu}^{\alpha}      \tpar^{1-\nu/2}\partial_{i} v_{\alpha}      =      \frac{1}{C}      \Vert \Pi \barpar (\tpar^{1-\nu/2} v) \Vert_{L^{2}(\Gamma_{1})}^{2}    .    \end{split}    \llabel{ Qx8mda g7 r VD3 zdD rhB vk LDJo tEQ228}   \end{align} In order to establish \eqref{EQ227}, we write $    g^{ij}\xi_{i}\xi_{j}    =    |\xi|^2    +    (\sqrtg g^{ij}-\delta^{ij})\xi_{i}\xi_{j} $ and appeal to Lemma~\ref{L01}(iv). \gsdfgdsfgdsfg \gsdfgdsfgdsfg Note that the last term in \eqref{EQ19} is dominated by $P$. We integrate the inequality \eqref{EQ19} in time on $[0,t]$ and then integrate by parts in time in the third and the fourth terms. Since both integrals are treated the same way, we only estimate the time integral of the third term. Denoting   \begin{equation}    A^{i\alpha}    =         \tpar^{1-\nu/2}        (                \sqrt{g} g^{ij} \Pi_{\lambda}^{\alpha} \partial_j v^\lambda          )       -         \sqrt{g} g^{ij} \Pi_{\lambda}^{\alpha} \tpar^{1-\nu/2}\partial_j  v^\lambda      ,    \llabel{KyV 5IrmyJ R5 e txS 1cv EsY xG zj2EQ185}   \end{equation} we have   \begin{align}      \begin{split}      \Vert A^{i\alpha}\Vert_{L^2(\Gamma_1)}      &\les      \Vert       \tpar^{1-\nu/2}       ( \sqrt{g} g^{ij} \Pi_{\lambda}^{\alpha})      \Vert_{L^{4/(1+\nu)}(\Gamma_1)}      \Vert\partial_j v^\lambda  \Vert_{L^{4/(1-\nu)}(\Gamma_1)}      \\&\indeq      +      \Vert       \tpar       ( \sqrt{g} g^{ij} \Pi_{\lambda}^{\alpha})      \Vert_{L^{4/(1+2\nu)}(\Gamma_1)}      \Vert\tpar^{-\nu/2} \partial_j v^\lambda  \Vert_{L^{4/(1-2\nu)}(\Gamma_1)}      \\&      \les      \Vert       \tpar^{1-\nu/2}       ( \sqrt{g} g^{ij} \Pi_{\lambda}^{\alpha})      \Vert_{H^{(1-\nu)/2}(\Gamma_1)}      \Vert\partial_j v^\lambda  \Vert_{H^{(1+\nu)/2}(\Gamma_1)}      \\&\indeq      +      \Vert       \tpar       ( \sqrt{g} g^{ij} \Pi_{\lambda}^{\alpha})      \Vert_{H^{1/2-\nu}(\Gamma_1)}      \Vert\tpar^{-\nu/2} \partial_j v^\lambda  \Vert_{H^{1/2+\nu}(\Gamma_1)}      \\&      \les      \Vert       \tpar^{1-\nu/2}       ( \sqrt{g} g^{ij} \Pi_{\lambda}^{\alpha})      \Vert_{H^{1-\nu/2}}      \Vert\partial_j v^\lambda  \Vert_{H^{1+\nu/2}}      \\&\indeq      +      \Vert       \tpar       ( \sqrt{g} g^{ij} \Pi_{\lambda}^{\alpha})      \Vert_{H^{1-\nu}}      \Vert\tpar^{-\nu/2} \partial_j v^\lambda  \Vert_{H^{1+\nu}}      \les      P(\Vert\eta\Vert_{H^{3-\nu}})      \Vert v\Vert_{H^{2+\nu/2}}      \end{split}    \llabel{T rfSR myZo4L m5 D mqN iZd acg GQ EQ186}    \end{align} where we used the commutator inequality (2.11)~in~\cite{KPV}. Now, the time integral of the third term on the right side of \eqref{EQ19} may then be estimated using integration by parts in time as   \begin{align}      \begin{split}    &    -    \int_{0}^{t}    \int_{\Gamma_1}     \Bigl(      \tpar^{1-\nu/2}        (                \sqrt{g} g^{ij} \Pi_{\lambda}^{\alpha} \partial_j v^\lambda          )       -         \sqrt{g} g^{ij} \Pi_{\lambda}^{\alpha} \tpar^{1-\nu/2}\partial_j  v^\lambda         \Bigr)      \partial_{i} \tpar^{1-\nu/2} \partial_{t}v_{\alpha}     \\&\indeq      =    -    \int_{0}^{t}    \int_{\Gamma_1}      A^{i\alpha}      \partial_{i} \tpar^{1-\nu/2} \partial_{t}v_{\alpha}     =    \emb{    -    \int_{\Gamma_1}      A^{i\alpha}      \partial_{i} \tpar^{1-\nu/2} v_{\alpha}     \bigm|_{0}^{t}    }    +    \int_{0}^{t}    \int_{\Gamma_1}      \partial_{t}      A^{i\alpha}      \partial_{i} \tpar^{1-\nu/2} v_{\alpha}     \\&\indeq     =    -    \int_{\Gamma_1}      A^{i\alpha}      \partial_{i} \tpar^{1-\nu/2} v_{\alpha}     \bigm|_{0}^{t}    \\&\indeq\indeq    +    \int_{0}^{t}    \int_{\Gamma_1}     \Bigl(         \tpar^{1-\nu/2}        (                \sqrt{g} g^{ij} \Pi_{\lambda}^{\alpha} \partial_j \partial_{t}v^\lambda          )       -         \sqrt{g} g^{ij} \Pi_{\lambda}^{\alpha} \tpar^{1-\nu/2}\partial_j\partial_{t}  v^\lambda        \Bigr)      \partial_{i} \tpar^{1-\nu/2} v_{\alpha}     \\&\indeq\indeq    +    \int_{0}^{t}    \int_{\Gamma_1}     \Bigl(         \tpar^{1-\nu/2}        (                \partial_{t}         (          \sqrt{g} g^{ij} \Pi_{\lambda}^{\alpha}         )        \partial_j v^\lambda          )       -        \partial_{t}        (           \sqrt{g} g^{ij} \Pi_{\lambda}^{\alpha}        )        \tpar^{1-\nu/2}\partial_j  v^\lambda        \Bigr)      \partial_{i} \tpar^{1-\nu/2} v_{\alpha}     \\&\indeq     \les      P_0     +     P(\Vert\eta\Vert_{H^{3-\nu}})     \Vert v\Vert_{H^{2+\nu/2}}     \Vert v\Vert_{H^{2.5-\nu/2}}    +    \int_{0}^{t} P     \\&\indeq     \les      P_0     + \ema{     \epsilon_0 \Vert v\Vert_{H^{2.5-\nu/2}}^2 }    +    \int_{0}^{t} P    .     \llabel{0KRw QKGX g9o8v8 wm B fUu tCO cKc EQ187}      \end{split}    \end{align} We estimate the fourth term in \eqref{EQ19} the same way. For the second term on the right side of \eqref{EQ19}, we use \eqref{EQ68} and obtain   \begin{equation}     \int_{\Gamma_1}    \tilde Q_{\mu\lambda}^{i}(\bar\partial\eta)\bar\partial^2\eta      \tpar^{1-\nu/2} v^{\mu}      \partial_{i} \tpar^{1-\nu/2} v^{\lambda}      \les      P_0      + \int_{0}^{t}P      + \ema{      \epsilon_0 \Vert  v\Vert_{H^{2.5-\nu/2}}^2 }    .    \label{EQ70}   \end{equation} Collecting all the estimates and using the bound \eqref{EQ70}, we obtain   $    \int_{0}^{t}      I_4      \les      P_0    +    \epsilon_0 \Vert  v\Vert_{H^{2.5-\nu/2}}^2    +    \int_{0}^{t}      P $, and \eqref{ZZ24} follows. \end{proof} \gsdfgdsfgdsfg \gsdfgdsfgdsfg \gsdfgdsfgdsfg \gsdfgdsfgdsfg \gsdfgdsfgdsfg \gsdfgdsfgdsfg \gsdfgdsfgdsfg \gsdfgdsfgdsfg \gsdfgdsfgdsfg \gsdfgdsfgdsfg \gsdfgdsfgdsfg \gsdfgdsfgdsfg \gsdfgdsfgdsfg \gsdfgdsfgdsfg \gsdfgdsfgdsfg \gsdfgdsfgdsfg \gsdfgdsfgdsfg \gsdfgdsfgdsfg \gsdfgdsfgdsfg \startnewsection{The $L^2$ estimate on $\partial_{t}^2 v$}{sec08} \gsdfgdsfgdsfg We have \eqref{EQ88} with $    \EE=\partial_{t} $, i.e.,   \begin{align}    \begin{split}    \frac12    \frac{d}{dt}    \Vert\partial_{t}^2v\Vert_{L^2}^2    &=      \into\partial_{t}^2(a^{\mu\alpha}q) \partial_{t}^2\partial_{\mu}v_{\alpha}    - \intu\partial_{t}^2(a^{\mu\alpha}q)\partial_{t}^2v_{\alpha} N_{\mu}    .    \end{split}    \llabel{zz kx4U fhuA a8pYzW Vq 9 Sp6 CmA cEQ05}   \end{align} We rewrite this as   \begin{align}    \begin{split}    \frac12    \frac{d}{dt}    \Vert\partial_{t}^2v\Vert_{L^2}^2    &    = J_1 + J_2 + J_3 + J_4    \end{split}    \llabel{ZL Mx ceBX Dwug sjWuii Gl v JDb 08EQ89a}   \end{align} where   \begin{align}    \begin{split}    J_1    &=    \into\partial_t^2a^{\mu\alpha} q\partial_t^2\partial_{\mu}v_{\alpha}       \comma    J_2    =    2 \into\partial_{t}a^{\mu\alpha}\partial_t q\partial_t^2\partial_{\mu}v_{\alpha}       \\    J_3    &=    \into a^{\mu\alpha}\partial_{t}^2 q\partial_t^2\partial_{\mu}v_{\alpha}       \comma    J_4    =    - \int_{\Gamma_1}          \partial_{t}^2( a^{\mu\alpha}q) \partial_{t}^2 v_{\alpha}N_{\mu}    .    \end{split}    \label{EQ99}   \end{align} \gsdfgdsfgdsfg \gsdfgdsfgdsfg \cole \begin{Lemma} \label{L10} The  time derivative of the Lagrangian  velocity  $\partial_{t}v$ and its second derivative $\partial_{t}^2v$ satisfy   \begin{align}   \begin{split}    &    \Vert \partial_{t}^2 v \Vert_{L^{2}}^{2}     + \Vert \Pi \barpar \partial_{t}v \Vert_{L^{2}(\Gamma_{1})}^{2}    \\&\indeq    \les        \Vert q\Vert_{H^{1}}     \Vert v\Vert_{H^{1.5}}^{(3-2\nu)/(2-\nu)}     \Vert v\Vert_{H^{2.5-\nu/2}}^{1/(2-\nu)}     \Vert \partial_{t}v \Vert_{H^{1.5}}    \\&\indeq\indeq    +    \Vert q\Vert_{H^{1}}^{(2-\nu-2\delta_0)/(3-\nu)}    \Vert q\Vert_{H^{2.5-\nu/2}}^{(1+2\delta_0)/(3-\nu)}    \Vert \partial_{t}v\Vert_{L^2}^{2/3}    \Vert \partial_{t}v\Vert_{H^{1.5}}^{4/3}    \left(1+\int_{0}^{t}P\right)    \\&\indeq\indeq    +    \Vert v\Vert_{H^{1.5}}^{(1-\nu)/(2-\nu)}    \Vert v\Vert_{H^{2.5-\nu/2}}^{1/(2-\nu)}    \Vert \partial_{t} q\Vert_{H^{1}}    \Vert \partial_{t} v\Vert_{L^2}^{(2-2\delta_0)/3}    \Vert \partial_{t} v\Vert_{H^{1.5}}^{(1+2\delta_0)/3}      \\&\indeq\indeq      +       \Vert \partial_{t}v \Vert_{L^2}^{1-\nu/3}       \Vert \partial_{t}v \Vert_{H^{1.5}}^{\nu/3}       \Vert \partial_{t}q \Vert_{H^{1}}       \Vert v \Vert_{H^{2.5-\nu/2}}      \\&\indeq\indeq      +         \Vert v\Vert_{H^{1.5}}^{(1-\nu)/(2-\nu)}    \Vert v\Vert_{H^{2.5-\nu/2}}^{1/(2-\nu)}    \Vert v\Vert_{H^{1.5}}^2    \Vert \partial_{t}q\Vert_{H^{1}}      \\&\indeq\indeq      +    \Vert v\Vert_{H^{1.5}}^{2(1-\nu)/(2-\nu)}    \Vert v\Vert_{H^{2.5-\nu/2}}^{2/(2-\nu)}      \Vert \partial_{t}v\Vert_{H^{1.5}}    + \epsilon_0 \Vert \partial_{t} v\Vert_{H^{1.5}}^2    +    P_0    +    \int_{0}^{t} P    ,   \end{split}    \label{EQ23}   \end{align} where $P$ is a polynomial in $\Vert v \Vert_{H^{3-\nu}}$,  $ \Vert \partial_{t}v \Vert_{H^{1.5}}$, $ \Vert \partial_{t}^2v \Vert_{L^2}$,  $\Vert q \Vert_{H^{2.5-\nu/2}}$, and $ \Vert \partial_{t}q \Vert_{H^{1}}$ and $P_0$ is a polynomial in $\Vert v_0\Vert_{H^{3-\nu}}$, $\Vert \partial_{t}v(0)\Vert_{H^{1.5}}$, and $\Vert \partial_{t}^2v(0)\Vert_{L^2}$. \end{Lemma} \colb \gsdfgdsfgdsfg We recall that $\barpar$ is given by \eqref{EQ37}. With the notation $    G    = \Vert q\Vert_{H^{2.5-\nu/2}} $ and $    H    = \Vert \partial_{t}q\Vert_{H^{1}} $, the equation \eqref{EQ23} may be rewritten as   \begin{align}   \begin{split}    &    \Vert \partial_{t}^2 v \Vert_{L^{2}}^{2}     + \Vert \Pi \barpar  \partial_{t}v \Vert_{L^{2}(\Gamma_{1})}^{2}    \\&\indeq    \les        \epsilon_0 E_1^2     +    \left(P_0+\int_{0}^{t}P\right)    \biggl(     E_0^{1/(2-\nu)}     E_1     +     G^{(1+2\delta_0)/(3-\nu)}     E_1^{4/3}    +    E_0^{1/(2-\nu)}    H    E_1^{(1+2\delta_0)/3}      \\&\indeq\indeq\indeq\indeq\indeq\indeq\indeq\indeq    + E_1^{\nu/3}      H      E_0    +    E_0^{1/(2-\nu)}    H    +    E_0^{2/(2-\nu)}    E_1    + 1    \biggr)    ,   \end{split}    \llabel{h BOV C1 pni6 4TTq Opzezq ZB J y5oEQ124}   \end{align} from where, taking the square root   \begin{align}   \begin{split}    &    \Vert \partial_{t}^2 v \Vert_{L^{2}}    + \Vert \Pi \barpar  \partial_{t}v \Vert_{L^{2}(\Gamma_{1})}    \\&\indeq    \les        \epsilon_0 E_1     +    \left(P_0+\int_{0}^{t}P\right)    \biggl(     E_0^{1/2(2-\nu)}     E_1^{1/2}     +     G^{(1+2\delta_0)/2(3-\nu)}     E_1^{2/3}    +    E_0^{1/2(2-\nu)}    H^{1/2}    E_1^{(1+2\delta_0)/6}      \\&\indeq\indeq\indeq\indeq\indeq\indeq\indeq\indeq    + E_1^{\nu/6}      H^{1/2}      E_0^{1/2}    +    E_0^{1/2(2-\nu)}    H^{1/2}    +    E_0^{1/(2-\nu)}    E_1^{1/2}    + 1    \biggr)    ,   \end{split}    \llabel{ KS8 BhH sd nKkH gnZl UCm7j0 Iv Y EQ189}   \end{align} and then using Young's inequality   \begin{align}   \begin{split}    &    \Vert \partial_{t}^2 v \Vert_{L^{2}}    + \Vert \Pi  \barpar \partial_{t}v \Vert_{L^{2}(\Gamma_{1})}    \\&\indeq    \les       \epsilon_0 E_1    +    \left(P_0+\int_{0}^{t}P\right)    \biggl(     E_0^{1/(2-\nu)}     +     G^{3(1+2\delta_0)/2(3-\nu)}    +    E_0^{3/(5-2\delta_0)(2-\nu)}    H^{3/(5-2\delta_0)}      \\&\indeq\indeq\indeq\indeq\indeq\indeq\indeq\indeq    +      H^{3/(6-\nu)}      E_0^{3/(6-\nu)}    +    E_0^{1/2(2-\nu)}    H^{1/2}    +    E_0^{2/(2-\nu)}    + 1    \biggr)    .   \end{split}    \llabel{jQE 7JN 9fd ED ddys 3y1x 52pbiG LcEQ125}   \end{align} Using the notation \eqref{EQ00},  i.e., $E=E_0^2+E_1+1$, this may be rewritten as   \begin{align}   \begin{split}    &    \Vert \partial_{t}^2 v  \Vert_{L^{2}}    + \Vert \Pi \barpar  \partial_{t}v \Vert_{L^{2}(\Gamma_{1})}    \\&\indeq    \les    \epsilon_0 E    +    \left(P_0+\int_{0}^{t}P\right)    \biggl(     E^{1/2(2-\nu)}     +     G^{3(1+2\delta_0)/2(3-\nu)}    +    E^{3/2(5-2\delta_0)(2-\nu)}    H^{3/(5-2\delta_0)}      \\&\indeq\indeq\indeq\indeq\indeq\indeq\indeq\indeq    +      H^{3/(6-\nu)}      E^{3/2(6-\nu)}    +    E^{1/4(2-\nu)}    H^{1/2}    +    E^{1/(2-\nu)}    + 1    \biggr)    ,   \end{split}    \llabel{ a 71j G3e uli Ce uzv2 R40Q 50JZUBEQ22}   \end{align} where $\delta_0>0$ is arbitrarily small. Using Young's inequality on the terms involving $E^{\gamma}$, where $\gamma\in[0,1)$, we get   \begin{align}   \begin{split}    &    \Vert \partial_{t}^2 v  \Vert_{L^{2}}    + \Vert \Pi \barpar  \partial_{t}v \Vert_{L^{2}(\Gamma_{1})}    \\&\indeq    \les    \epsilon_0 E    +    \left(P_0+\int_{0}^{t}P\right)    \biggl(     G^{3(1+2\delta_0)/2(3-\nu)}    +    H^{6(2-\nu)/(2(5-2\delta_0)(2-\nu)-3)}      \\&\indeq\indeq\indeq\indeq\indeq\indeq\indeq\indeq       +      H^{6/(9- 2\nu)}      +    H^{2(2-\nu)/(7-4\nu)}    + 1    \biggr)    .   \end{split}    \llabel{ uK d U3m May 0uo S7 ulWD h7qG 2FKEQ136}   \end{align} It is easy to check that the exponents of $H$ are all less than $3/4$ for $\delta_0>0$ sufficiently small. (In order to verify $6(2-\nu)/(2(5-2\delta_0)(2-\nu)-3)\le 3/4$, for $\delta_0$ sufficiently small, first set $\delta_0=0$ and check that $6(2-\nu)/(10(2-\nu)-3)< 3/4$ for $\nu\in[0,1/2)$.) Therefore,   \begin{align}   \begin{split}    &    \Vert \partial_{t}^2 v  \Vert_{L^{2}}    + \Vert \Pi \barpar  \partial_{t}v \Vert_{L^{2}(\Gamma_{1})}    \les    \epsilon_0 E    +    \left(P_0+\int_{0}^{t}P\right)    \biggl(     G^{3(1+2\delta_0)/2(3-\nu)}    +      H^{3/4}    + 1    \biggr)    .   \end{split}    \label{EQ103}   \end{align} \gsdfgdsfgdsfg \begin{proof}[Proof of Lemma~\ref{L10}] Let $J_1, J_2, J_3, J_4$ be as in \eqref{EQ99}. \gsdfgdsfgdsfg {\bf Treatment of $J_1$:} For $J_1$, we apply Lemma~\ref{L12} with  $\EE=\partial_{t}$ (that is $\DD=\partial_{t}^2$). We start with the term $L$ in \eqref{EQ57}, which, with $\DD=\partial_{t}^2$, reads   \begin{align}    \begin{split}    L    &=    2    \int_{0}^{t}    \into      q \epsilon^{\al \lambda \tau}      \partial_2 v_\lambda     \partial_3 v_\tau          \partial_{t}^2 \partial_1 v_\alpha     -    2       \int_{0}^{t}    \into      q \epsilon^{\al \lambda \tau}      \partial_1 v_\lambda     \partial_3 v_\tau          \partial_{t}^2 \partial_2 v_\alpha     +       2    \int_{0}^{t}    \into      q \epsilon^{\al \lambda \tau}      \partial_1 v_\lambda     \partial_2 v_\tau          \partial_{t}^2 \partial_3 v_\alpha     \\&     = L_1 + L_2 + L_3    .    \end{split}    \llabel{w2T JX z BES 2Jk Q4U Dy 4aJ2 IXs4 EQ85}   \end{align} We only treat the first term as the other two are handled similarly. Integrating by parts in time, we have   \begin{align}    \begin{split}    L_1    &= \emb{    2    \into      q \epsilon^{\al \lambda \tau}      \partial_2 v_\lambda     \partial_3 v_\tau          \partial_{t} \partial_1 v_\alpha     \bigm|_{0}^{t} }     -2    \int_{0}^{t}    \into      \partial_{t}q      \epsilon^{\al \lambda \tau}      \partial_2 v_\lambda     \partial_3 v_\tau          \partial_{t} \partial_1 v_\alpha     \\&\indeq     -2    \int_{0}^{t}    \into      q      \epsilon^{\al \lambda \tau}      \partial_2 \partial_{t}v_\lambda     \partial_3 v_\tau          \partial_{t} \partial_1 v_\alpha     -2    \int_{0}^{t}    \into      q      \epsilon^{\al \lambda \tau}      \partial_2 v_\lambda     \partial_3 \partial_{t}v_\tau          \partial_{t} \partial_1 v_\alpha     \\&     \les     \Vert q\Vert_{H^{1}}     \Vert v\Vert_{H^{1.75}}^2     \Vert \partial_{t}v \Vert_{H^{1.5}}\bigm|_{t}    +     \Vert q\Vert_{H^{1}}     \Vert v\Vert_{H^{1.75}}^2     \Vert \partial_{t}v \Vert_{H^{1.5}}\bigm|_{0}    \\&\indeq    +    \int_{0}^{t}     \Vert\partial_{t}q\Vert_{H^{1}}     \Vert v\Vert_{H^{2}}     \Vert v\Vert_{H^{1.5}}     \Vert \partial_{t}v\Vert_{H^{1.5}}    +    \int_{0}^{t}     \Vert q\Vert_{H^{1}}     \Vert \partial_{t}v\Vert_{H^{1.5}}     \Vert v\Vert_{H^{2}}     \Vert \partial_{t}v\Vert_{H^{1.5}}    \\&    \les    P_0    + \ema{     \Vert q\Vert_{H^{1}}     \Vert v\Vert_{H^{1.5}}^{(3-2\nu)/(2-\nu)}     \Vert v\Vert_{H^{2.5-\nu/2}}^{1/(2-\nu)}     \Vert \partial_{t}v \Vert_{H^{1.5}}\bigm|_{t} }    +    \int_{0}^{t}     P    ,    \end{split}    \llabel{RNH41s py T GNh hk0 w5Z C8 B3nU BpEQ109}   \end{align} where we used $    \Vert v\Vert_{H^{1.75}}    \les    \Vert v\Vert_{H^{1.5}}^{(3-2\nu)/(4-2\nu)}    \Vert v\Vert_{H^{2.5-\nu/2}}^{1/(4-2\nu)} $ in the last step. On the other hand, the right side of \eqref{EQ55} without $L$ is bounded by   \begin{align}    \begin{split}    &    \Vert q\Vert_{H^{1.5+\delta_0}}    \Vert \partial_{t}v\Vert_{H^{1}}^2    +    \Vert q(0)\Vert_{H^{1.5+\delta_0}}    \Vert \partial_{t}v(0)\Vert_{H^{1}}^2    \\&\indeq\indeq    +    \int_{0}^{t}    \Bigl(      \Vert \partial_{t}q \Vert_{H^{1}}      \Vert \eta\Vert_{H^{2}}      \Vert \partial_{t}v\Vert_{H^{1.5}}^2      +      \Vert q\Vert_{H^{1}}      \Vert v\Vert_{H^{2}}      \Vert \partial_{t}v\Vert_{H^{1.5}}^2     \Bigr)     +    \Vert q\Vert_{H^{1.5+\delta_0}}    \Vert \partial_{t}v\Vert_{H^{1}}^2    \int_{0}^{t} \Vert v\Vert_{H^{2.5+\delta_0}}     \\&\indeq     \les    P_0    + \ema{    \Vert q\Vert_{H^{1}}^{(2-\nu-2\delta_0)/(3-\nu)}    \Vert q\Vert_{H^{2.5-\nu/2}}^{(1+2\delta_0)/(3-\nu)}    \Vert \partial_{t}v\Vert_{L^2}^{2/3}    \Vert \partial_{t}v\Vert_{H^{1.5}}^{4/3}    \left(1+\int_{0}^{t}P\right) }    +    \int_{0}^{t}    P    \comma     \end{split}    \label{EQ105}   \end{align} where $\delta_0>0$ is arbitrarily small. Note that the second term on the right side of \eqref{EQ105} is an upper bound for both the first and the fourth terms on the left. Therefore, we conclude   \begin{align}    \begin{split}    \int_{0}^{t}     J_1     &\les     P_0     +    \Vert q\Vert_{H^{1}}^{(2-\nu-2\delta_0)/(3-\nu)}    \Vert q\Vert_{H^{2.5-\nu/2}}^{(1+2\delta_0)/(3-\nu)}    \Vert \partial_{t}v\Vert_{L^2}^{2/3}    \Vert \partial_{t}v\Vert_{H^{1.5}}^{4/3}    \left(1+\int_{0}^{t}P\right)    \\&\indeq    +     \Vert q\Vert_{H^{1}}     \Vert v\Vert_{H^{1.5}}^{(3-2\nu)/(2-\nu)}     \Vert v\Vert_{H^{2.5-\nu/2}}^{1/(2-\nu)}     \Vert \partial_{t}v \Vert_{H^{1.5}}\bigm|_{t}    +    \int_{0}^{t}     P    .    \end{split}    \llabel{9p 8eLKh8 UO 4 fMq Y6w lcA GM xCHtEQ110}   \end{align} \gsdfgdsfgdsfg \gsdfgdsfgdsfg \nnewpage {\bf Treatment of $J_2$:} Now we bound $    \int_{0}^{t} J_{2}       =            2\int_{0}^{t} \int_{\Omega}  \partial_{t}  a^{\mu \alpha} \partial_{t}q \partial_{\mu} \partial_{t}^2v_{\alpha} $. Using integration by parts in $x_\mu$ and the Piola identity \eqref{ZZ00}, we get   \begin{align}      \begin{split}       \int_{0}^{t} J_{2}       &=          -          2\int_{0}^{t} \int_{\Omega}  \partial_{t}  a^{\mu \alpha} \partial_{t}\partial_{\mu}q \partial_{t}^2v_{\alpha}       +         2\int_{0}^{t} 
\int_{\Gamma_1}  \partial_{t}  a^{\mu \alpha} \partial_{t}q \partial_{t}^2v_{\alpha} N_{\mu}      = \int_{0}^{t}J_{21} + \int_{0}^{t}J_{22}      .     \end{split}    \label{EQ158}    \end{align} The first term is estimated using H\"older inequality as   \begin{equation}    \int_{0}^{t}    J_{21}    \les    \int_{0}^{t}    \Vert \partial_{t} a\Vert_{H^{1.5+\delta_0}}    \Vert \partial_{t}q\Vert_{H^{1}}    \Vert \partial_{t}^2 v\Vert_{L^2}    \les    \int_{0}^{t}    P    .    \llabel{ vlOx MqAJoQ QU 1 e8a 2aX 9Y6 2r lEQ159}   \end{equation} For the second term in \eqref{EQ158}, we integrate by parts in $t$, leading to   \begin{align}      \begin{split}      \int_{0}^{t}      J_{22}      &= \emb{         2 \int_{\Gamma_1}  \partial_{t}  a^{\mu \alpha} \partial_{t}q \partial_{t}v_{\alpha} N_{\mu}\restr_{0}^{t} }       -         2\int_{0}^{t} \int_{\Gamma_1}  \partial_{t}^2  a^{\mu \alpha} \partial_{t}q \partial_{t}v_{\alpha} N_{\mu}       -         2\int_{0}^{t} \int_{\Gamma_1}  \partial_{t}  a^{\mu \alpha} \partial_{t}^2q \partial_{t}v_{\alpha} N_{\mu}      \\&       = J_{221} + \int_{0}^{t} J_{222} + \int_{0}^{t}J_{223}    .      \end{split}    \llabel{IS6 dejK Y3KCUm 25 7 oCl VeE e8p 1EQ160}    \end{align} For the pointwise in time term, we have   \begin{align}     \begin{split}    J_{221}|_{t}     &\les    \Vert \partial_{t}a\Vert_{H^{1/2}(\Gamma_1)}    \Vert \partial_{t} q\Vert_{H^{1/2}(\Gamma_1)}    \Vert \partial_{t} v\Vert_{L^2(\Gamma_1)}     \\&     \les    \Vert \partial_{t}a\Vert_{H^{1}}    \Vert \partial_{t} q\Vert_{H^{1}}    \Vert \partial_{t} v\Vert_{H^{1/2+\delta_0}}    \les    \Vert v\Vert_{H^{2}}    \Vert \partial_{t} q\Vert_{H^{1}}    \Vert \partial_{t} v\Vert_{H^{1/2+\delta_0}}    \\&    \les \ema{    \Vert v\Vert_{H^{1.5}}^{(1-\nu)/(2-\nu)}    \Vert v\Vert_{H^{2.5-\nu/2}}^{1/(2-\nu)}    \Vert \partial_{t} q\Vert_{H^{1}}    \Vert \partial_{t} v\Vert_{L^2}^{(2-2\delta_0)/3}    \Vert \partial_{t} v\Vert_{H^{1.5}}^{(1+2\delta_0)/3} }    .   \end{split}    \llabel{z UJSv bmLd Fy7ObQ FN l J6F RdF kEEQ161}   \end{align} We emphasize that \eqref{ZZ15} and \eqref{ZZ16} should not be used to treat $J_{222}$. Instead, we write   \begin{align}    \begin{split}      \int_{0}^{t}      J_{222}      &=       -         2\int_{0}^{t} \int_{\Gamma_1}  \partial_{t}^2  a^{\mu \alpha} \partial_{t}q \partial_{t}v_{\alpha} N_{\mu}     =       -         2\int_{0}^{t} \int_{\Gamma_1}  \partial_{t}^2  a^{3 \alpha} \partial_{t}q \partial_{t}v_{\alpha}    .    \end{split}    \llabel{m qM N0Fd NZJ0 8DYuq2 pL X JNz 4rOEQ162}    \end{align} From \cite{DisconziKukavicaIncompressible}, recall the formula for the third row of the matrix $a$, which reads   \begin{align}    \begin{split}    a^{3 \cdot} =     \begin{bmatrix}    \partial_1 \eta^2 \partial_2 \eta^3 - \partial_2 \eta^2 \partial_1 \eta^3,    &    \partial_2 \eta^1 \partial_1 \eta^3 - \partial_1 \eta^1 \partial_2 \eta ^3,     &    \partial_1 \eta^1 \partial_2 \eta^2 - \partial_2 \eta^1 \partial_1 \eta^2    \end{bmatrix}    .    \end{split}    \llabel{ ZkZ X2 IjTD 1fVt z4BmFI Pi 0 GKD EQ170}   \end{align} It is essential that \emph{only tangential derivatives appear in each entry}. Therefore, for all $\alpha=1,2,3$,   \begin{align}    \begin{split}      \Vert \partial_{t}^2         a^{3\alpha}\Vert_{L^2(\Gamma_1)}      &\les       \Vert \barpar \eta\Vert_{H^{1+\delta_0}(\Gamma_1)}       \Vert \barpar \partial_{t}v\Vert_{L^2(\Gamma_1)}       +       \Vert \barpar v\Vert_{H^{0.5}(\Gamma_1)}^2      \\&      \les       \Vert \eta\Vert_{H^{2+\delta_0}(\Gamma_1)}       \Vert \partial_{t}v\Vert_{H^{1}(\Gamma_1)}       +       \Vert v\Vert_{H^{1.5}(\Gamma_1)}^2     \\&       \les       \Vert \eta\Vert_{H^{2.5+\delta_0}}       \Vert \partial_{t}v\Vert_{H^{1.5}}       +       \Vert v\Vert_{H^{2}}^2     \les       \Vert \partial_{t}v\Vert_{H^{1.5}}       +       \Vert v\Vert_{H^{2}}^2     \les P     .    \end{split}    \label{EQ164}    \end{align} Thus we have   \begin{align}    \begin{split}     \int_{0}^{t}     J_{222}     & \les     \int_{0}^{t}     \Vert \partial_{t}^2 a\Vert_{L^2(\Gamma_1)}     \Vert \partial_{t} q\Vert_{H^{1/2}(\Gamma_1)}     \Vert \partial_{t} v\Vert_{H^{1/2}(\Gamma_1)}     \les     \int_{0}^{t}     \Vert \partial_{t}^2 a\Vert_{L^2(\Gamma_1)}     \Vert \partial_{t} q\Vert_{H^{1}}     \Vert \partial_{t} v\Vert_{H^{1}}     \leq     \int_{0}^{t} P    \end{split}    \llabel{R2W PhO zH zTLP lbAE OT9XW0 gb T LEQ166}    \end{align} using \eqref{EQ164} and the trace inequality. Lastly, we consider $J_{223}$, for which we use \eqref{ZZ15} and \eqref{ZZ16}:   \begin{align}      \begin{split}       \int_{0}^{t}       J_{223}       &=         2\int_{0}^{t} \int_{\Gamma_1}            a^{\mu \beta} \partial_{\lambda} v_{\beta} a^{\lambda \alpha} 	   \partial_{t}^2q 	   \partial_{t}v_{\alpha} N_{\mu}      \\&         =         2\int_{0}^{t} \int_{\Gamma_1}            \partial_{t}^2( 	   N_{\mu}            a^{\mu \beta}  	   q) 	   \partial_{\lambda} v_{\beta} a^{\lambda \alpha} 	   \partial_{t}v_{\alpha}        -         2\int_{0}^{t} \int_{\Gamma_1}            \partial_{t}^2 a^{\mu \beta}  	   q 	   \partial_{\lambda} v_{\beta} a^{\lambda \alpha} 	   \partial_{t}v_{\alpha} N_{\mu}             \\&\indeq        -          4\int_{0}^{t} \int_{\Gamma_1}            \partial_{t} a^{\mu \beta}             \partial_{t} q 	   \partial_{\lambda} v_{\beta} a^{\lambda \alpha} 	   \partial_{t}v_{\alpha} N_{\mu}            \\&         =         2\int_{0}^{t} \int_{\Gamma_1}            \partial_{t}^2( 	   \partial_{i} (\sqrt{g} g^{ij} \partial_{j}\eta^{\beta})) 	   \partial_{\lambda} v_{\beta} a^{\lambda \alpha} 	   \partial_{t}v_{\alpha}        -         2\int_{0}^{t} \int_{\Gamma_1}            \partial_{t}^2 a^{3 \beta}  	   q 	   \partial_{\lambda} v_{\beta} a^{\lambda \alpha} 	   \partial_{t}v_{\alpha}        \\&\indeq        -          4\int_{0}^{t} \int_{\Gamma_1}            \partial_{t} a^{3 \beta}             \partial_{t} q 	   \partial_{\lambda} v_{\beta} a^{\lambda \alpha} 	   \partial_{t}v_{\alpha}       \\&         =         -         2\int_{0}^{t} \int_{\Gamma_1}            \partial_{t}^2 	    (\sqrt{g} g^{ij} \partial_{j}\eta^{\beta})            \partial_{i}(\partial_{\lambda} v_{\beta} a^{\lambda \alpha} 	   \partial_{t}v_{\alpha})        \\&\indeq        -         2\int_{0}^{t} \int_{\Gamma_1}            \partial_{t}^2 a^{3 \beta}  	   q 	   \partial_{\lambda} v_{\beta} a^{\lambda \alpha} 	   \partial_{t}v_{\alpha}        -          4\int_{0}^{t} \int_{\Gamma_1}            \partial_{t} a^{3 \beta}             \partial_{t} q 	   \partial_{\lambda} v_{\beta} a^{\lambda \alpha} 	   \partial_{t}v_{\alpha}       \\&       = \int_{0}^{t}J_{2231} + \int_{0}^{t}J_{2232} + \int_{0}^{t}J_{2233}    .          \end{split}    \label{EQ167}      \end{align} The term $\int_{0}^{t}J_{2231}$ may now be estimated  with $\int_{0}^{t}P$ by simply expanding. For the second term, we use \eqref{EQ164}, after which it is also bounded by $\int_{0}^{t}P$. The third term is also bounded directly, and thus $    \int_{0}^{t}    J_{223}    \les \int_{0}^{t}P $. Collecting all the inequalities, we get   \begin{align}    \begin{split}    \int_{0}^{t}J_2     &\les     P_0     +     \Vert v\Vert_{H^{1.5}}^{(1-\nu)/(2-\nu)}    \Vert v\Vert_{H^{2.5-\nu/2}}^{1/(2-\nu)}    \Vert \partial_{t} q\Vert_{H^{1}}    \Vert \partial_{t} v\Vert_{L^2}^{(2-2\delta_0)/3}    \Vert \partial_{t} v\Vert_{H^{1.5}}^{(1+2\delta_0)/3}     +     \int_{0}^{t}P    .    \end{split}    \llabel{b3 XRQ qGG 8o 4TPE 6WRc uMqMXh s6 EQ09}   \end{align} \gsdfgdsfgdsfg \nnewpage {\bf Treatment of $J_3$:} Here we estimate $    J_3    =    \into a^{\mu\alpha}\partial_{t}^2 q\partial_t^2\partial_{\mu}v_{\alpha}    $. Using  \eqref{ZZ14} and \eqref{ZZ15}, the term $J_{3}$ can be expressed as   \begin{align*}     J_{3}      &=        -        \int_{\Omega}            \partial_{t}^{2}a^{\mu \alpha} \partial_t^2q \partial_{\mu} v_{\alpha}        - 2 \int_{\Omega}                  \partial_{t}a^{\mu \alpha}                 \partial_t^2q \partial_{\mu} \partial_{t}v_{\alpha}    \\&     =  \int_{\Omega}                a^{\mu \beta}               \partial_{\lambda} \partial_{t}v_{\beta} a^{\lambda \alpha}               \partial_t^2q \partial_{\mu} v_{\alpha}        + \int_{\Omega}                  \partial_{t}(a^{\mu \beta}a^{\lambda \alpha} )                 \partial_{\lambda} v_{\beta}                \partial_t^2q \partial_{\mu} v_{\alpha}        \\&\indeq           - 2 \int_{\Omega}                  \partial_{t}a^{\mu \alpha}                  \partial_t^2q \partial_{\mu} \partial_{t}v_{\alpha}      = J_{31} + J_{32} +J_{33}    .   \end{align*} To treat the term $J_{31}$, we integrate by parts in $x_\lambda$ obtaining   \begin{align}    \begin{split}      J_{31}       &=       -      \int_{\Omega}           \partial_{\lambda} a^{\mu \beta} \partial_{t}v_{\beta}          a^{\lambda \alpha} \partial_t^2q \partial_{\mu} v_{\alpha}       - \int_{\Omega}             a^{\mu \beta} \partial_{t}v_{\beta} a^{\lambda \alpha}            \partial_{\lambda} \partial_t^2q \partial_{\mu} v_{\alpha}       \\&\indeq       -  \int_{\Omega}              a^{\mu \beta} \partial_{t}v_{\beta} a^{\lambda \alpha}             \partial_t^2q \partial_{\lambda}  \partial_{\mu} v_{\alpha}        +  \int_{\Gamma_{1}}             a^{\mu \beta} \partial_{t}v_{\beta} a^{\lambda \alpha}            \partial_t^2q \partial_{\mu} v_{\alpha} N_{\lambda}     \\&     =  J_{311}      +J_{312}      +J_{313}      +J_{314}    \end{split}    \label{ZZ82}   \end{align} where we used \eqref{ZZ00}. Integrating in time the first term and then treating it by integration by parts in time, we get   \begin{align*}    \begin{split}     \int_{0}^{t} J_{311}       &=    \emb{     -      \int_{\Omega}          \partial_{\lambda} a^{\mu \beta} \partial_{t}v_{\beta} a^{\lambda \alpha}         \partial_{t}q \partial_{\mu} v_{\alpha}       \bigm|_{0}^{t} }     + \int_{0}^{t}  \int_{\Omega}          \partial_{t}( \partial_{\lambda} a^{\mu \beta}          \partial_{t}v_{\beta} a^{\lambda \alpha}) \partial_{t}q \partial_{\mu} v_{\alpha}    \\&\indeq     + \int_{0}^{t}  \int_{\Omega}           \partial_{\lambda} a^{\mu \beta} \partial_{t}v_{\beta}          a^{\lambda \alpha} \partial_{t}q \partial_{\mu} \partial_{t}v_{ \alpha}     .    \end{split}   \end{align*} The pointwise in time term in the above sum may be bounded as   \begin{align}    \begin{split}     &     -      \int_{\Omega}          \partial_{\lambda} a^{\mu \beta} \partial_{t}v_{\beta} a^{\lambda \alpha}         \partial_{t}q \partial_{\mu} v_{\alpha}       \bigm|_t       \les       \Vert a \Vert_{H^{1.5}}         \Vert \partial_{t}v \Vert_{H^{\nu/2}}         \Vert \partial_{t}q \Vert_{H^{1}}       \Vert v \Vert_{H^{2.5-\nu/2}}        \les       \Vert \partial_{t}v \Vert_{H^{\nu/2}}         \Vert \partial_{t}q \Vert_{H^{1}}       \Vert v \Vert_{H^{2.5-\nu/2}}        ,    \end{split}    \llabel{x Ofv 8st jDi u8 rtJt TKSK jlGkGw ZZ67}   \end{align} by Lemma~\ref{L01}(ii), and we obtain   \begin{align}    \begin{split}    \int_{0}^{t} J_{311}     \les    P_0    + \ema{       \Vert \partial_{t}v \Vert_{L^2}^{1-\nu/3}       \Vert \partial_{t}v \Vert_{H^{1.5}}^{\nu/3}       \Vert \partial_{t}q \Vert_{H^{1}}       \Vert v \Vert_{H^{2.5-\nu/2}} }    +\int_{0}^{t} P    \end{split}    \label{EQ18}   \end{align} where we used Lemma~\ref{L01}(ii). Similarly, using the divergence condition \eqref{ZZ14},   we have $a^{\lambda\alpha}\partial_{\lambda}\partial_{\mu}v_{\alpha} =- \partial_{\mu}a^{\lambda\alpha}\partial_{\lambda}v_{\alpha}$, and the third term in~\eqref{ZZ82} can be rewritten as $    J_{313}    =    \int_{\Omega}              a^{\mu \beta} \partial_{t}v_{\beta}             \partial_{\mu}a^{\lambda \alpha}             \partial_t^2q \partial_{\lambda}  v_{\alpha}  $. Note that  it has the same structure as $J_{311}$ and it thus satisfies the same estimate.  \gsdfgdsfgdsfg In the term $J_{312}$, we integrate by parts in time, obtaining   \begin{align}    \begin{split}     \int_{0}^{t}     J_{312}      &=  \emb{      -      \int_{\Omega}           a^{\mu \beta} \partial_{t}v_{\beta} a^{\lambda \alpha}            \partial_{\lambda} \partial_t q \partial_{\mu} v_{\alpha}       \restr_{0}^{t} }      +      \int_{0}^{t}      \int_{\Omega}          \partial_{t}( a^{\mu \beta} \partial_{t}v_{\beta} a^{\lambda \alpha} )           \partial_{\lambda} \partial_t q \partial_{\mu} v_{\alpha}      \\&\indeq      +      \int_{0}^{t}      \int_{\Omega}           a^{\mu \beta} \partial_{t}v_{\beta} a^{\lambda \alpha}            \partial_{\lambda} \partial_t q            \partial_{\mu} \partial_{t}v_{\alpha}     .    \end{split}    \llabel{t8 n FDx jA9 fCm iu FqMW jeox 5AkwZZ75}   \end{align} The pointwise in term satisfies   \begin{align}    \begin{split}     &      -      \int_{\Omega}           a^{\mu \beta} \partial_{t}v_{\beta} a^{\lambda \alpha}            \partial_{\lambda} \partial_t q \partial_{\mu} v_{\alpha}      \bigm|_{0}^{t}     \les     \Vert \partial_{t} v\Vert_{H^{\nu/2}}     \Vert \partial_{t}q\Vert_{H^{1}}     \Vert v\Vert_{H^{2.5-\nu/2}}     +     \Vert \partial_{t} v\Vert_{H^{\nu/2}}     \Vert \partial_{t} q\Vert_{H^{1}}     \Vert v\Vert_{H^{2.5-\nu/2}}      \restr_{0}    \\&\indeq    \les    P_0    +     \Vert \partial_{t} v\Vert_{L^2}^{1-\nu/3}     \Vert \partial_{t} v\Vert_{H^{1.5}}^{\nu/3}     \Vert \partial_{t}q\Vert_{H^{1}}     \Vert v\Vert_{H^{2.5-\nu/2}}    \end{split}    \label{ZZ76}   \end{align} where, in particular, we used Lemma~\ref{L01}(ii). (Note that this has the same upper bound as in \eqref{EQ18}.) Therefore,   \begin{align}    \begin{split}     \int_{0}^{t}      J_{312}     &\les     P_0     +     \Vert \partial_{t} v\Vert_{L^2}^{1-\nu/3}     \Vert \partial_{t} v\Vert_{H^{1.5}}^{\nu/3}     \Vert \partial_{t}q\Vert_{H^{1}}     \Vert v\Vert_{H^{2.5-\nu/2}}       + \int_{0}^{t}             P    .    \end{split}    \llabel{3w Sd 8 1vK 8c4 C0O dj CHIs eHUO hZZ77}   \end{align} The boundary term  $J_{314}$ can be expressed as   \begin{align}   \begin{split}    \int_{0}^{t}        J_{314}        &=       \int_{0}^{t}       \int_{\Gamma_{1}}             a^{\mu \beta} \partial_{t}v_{\beta} a^{\lambda \alpha}            \partial_t^2q \partial_{\mu} v_{\alpha} N_{\lambda}        \\&       =       \int_{0}^{t}        \int_{\Gamma_{1}}          \partial_t^2(N_{\lambda}a^{\lambda \alpha}  q )          a^{\mu \beta} \partial_{t}v_{\beta} 	 \partial_{\mu} v_{\alpha}        -        \int_{0}^{t}        \int_{\Gamma_{1}}          \partial_t^2 a^{\lambda \alpha}  q           a^{\mu \beta} \partial_{t}v_{\beta} 	 \partial_{\mu} v_{\alpha} N_{\lambda}         \\&\indeq         -         2 	\int_{0}^{t}               \int_{\Gamma_{1}}          \partial_t a^{\lambda \alpha}  \partial_{t}q           a^{\mu \beta} \partial_{t}v_{\beta} 	 \partial_{\mu} v_{\alpha} N_{\lambda}     .   \end{split}    \llabel{yqGx3 Kw O lDq l1Y 4NY 4I vI7X DE4EQ01}     \end{align} Note that all three terms have the same structure as the three terms in \eqref{EQ167} and are treated analogously, leading to the same upper bounds. \gsdfgdsfgdsfg \gsdfgdsfgdsfg The term $J_{32}$ is treated by using integration by parts in time (and no integration by parts in space). Since all the terms are treated in a straight-forward way, we only estimate the pointwise in time term, which equals   \begin{align}    \begin{split}         &  \emb{        \int_{\Omega}                  \partial_{t}(a^{\mu \beta}a^{\lambda \alpha} )                 \partial_{\lambda} v_{\beta}                \partial_t q \partial_{\mu} v_{\alpha}  }       \les     \Vert \partial_{t}a\Vert_{H^{1}}     \Vert v\Vert_{H^{1.5}}^2     \Vert \partial_{t}q\Vert_{H^{1}}     \\&\indeq     \les     \Vert v\Vert_{H^{2}}     \Vert v\Vert_{H^{1.5}}^2     \Vert \partial_{t}q\Vert_{H^{1}}    \les \ema{    \Vert v\Vert_{H^{1.5}}^{(1-\nu)/(2-\nu)}    \Vert v\Vert_{H^{2.5-\nu/2}}^{1/(2-\nu)}    \Vert v\Vert_{H^{1.5}}^2    \Vert \partial_{t}q\Vert_{H^{1}} }    \end{split}    \llabel{c FeXdFV bC F HaJ sb4 OC0 hu Mj65 EQ113}   \end{align} by \eqref{EQ118}. \gsdfgdsfgdsfg   \gsdfgdsfgdsfg It remains to consider $J_{33}$. We first integrate by parts in $x_{\mu}$ leading to   \begin{align}      \begin{split}        \int_{0}^{t}     J_{33}       &=          - 2 \int_{0}^{t} \int_{\Omega}               \partial_{t}a^{\mu \alpha}               \partial_t^2q \partial_{\mu} \partial_{t}v_{\alpha}         =           2 \int_{0}^{t} \int_{\Omega}               \partial_{t}a^{\mu \alpha}               \partial_t^2 \partial_{\mu}q \partial_{t}v_{\alpha}            - 2 \int_{0}^{t} \int_{\Gamma_1}               \partial_{t}a^{\mu \alpha}               \partial_t^2q  \partial_{t}v_{\alpha} N_{\mu}       \\&       = \int_{0}^{t}J_{331} + \int_{0}^{t}J_{332}          .    \end{split}    \llabel{J4fa vgGo7q Y5 X tLy izY DvH TR zdEQ101}      \end{align} The second term is identical to $\int_{0}^{t}J_{223}$. For $J_{331}$, we integrate by parts in time. The integrated in time terms are controlled by $\int_{0}^{t}P$, while the pointwise in time term evaluated at $t$ reads \begin{equation*} \emb{        2        \into              \partial_{t}a^{\mu \alpha}               \partial_t \partial_{\mu}q \partial_{t}v_{\alpha} \restr_{t} }       \les \ema{       \Vert \partial_{t}a\Vert_{H^{1.5-\nu/2}}       \Vert \partial_{t}\nabla q\Vert_{L^2}       \Vert \partial_{t}v\Vert_{H^{\nu/2}} } . \end{equation*} Note that this is the same upper bound as in \eqref{ZZ76}. \gsdfgdsfgdsfg \gsdfgdsfgdsfg \gsdfgdsfgdsfg \gsdfgdsfgdsfg \gsdfgdsfgdsfg \gsdfgdsfgdsfg \nnewpage {\bf Treatment of $J_{4}$:} It only remains to consider the boundary term $J_{4}$, in which case we use \eqref{EQ77} with $\EE=\partial_{t}$. Thus   \begin{align}    \begin{split}    J_{4} &      \les         -    \frac12    \frac{d}{dt}    \int_{\Gamma_1}         \sqrt{g} g^{ij}           \Pi_{\lambda}^{\mu}         \partial_t      \partial_j v^\lambda           \Pi_{\mu}^{\alpha}      \partial_t\partial_{i} v_{\alpha}    +    \frac{d}{dt}    \int_{\Gamma_1}    \tilde Q_{\mu\lambda}^{i}(\bar\partial\eta)\bar\partial^2\eta      \partial_t v^{\mu}      \partial_{i} \partial_t v^{\lambda}    \\&\indeq    -    \int_{\Gamma_1}     \Bigl(      \partial_t         (                \sqrt{g} g^{ij} \Pi_{\lambda}^{\alpha} \partial_j v^\lambda          )       -         \sqrt{g} g^{ij} \Pi_{\lambda}^{\alpha} \partial_t\partial_j  v^\lambda         \Bigr)      \partial_{i} \partial_t^2v_{\alpha}    \\&\indeq      -       \int_{\Gamma_1}        \biggl(         \partial_t         \Bigl(          
\sqrt{g}(g^{ij} g^{kl} - g^{lj}g^{ik} )          \partial_{j}\eta^{\alpha}                  \partial_{k}\eta_{\lambda}          \partial_{l} v^{\lambda}         \Bigr)         \\&\indeq\indeq\indeq         -          \sqrt{g}(g^{ij} g^{kl} - g^{lj}g^{ik} )           \partial_{j}\eta^{\alpha}                  \partial_{k}\eta_{\lambda}          \partial_{l}\partial_t v^{\lambda}        \biggr)         \partial_{i}\partial_t^2 v_{\alpha}    \\&\indeq      +     P(\Vert \eta\Vert_{H^{2.5+\delta_0}})           (        \Vert v\Vert_{H^{2.5+\delta_0}}+1)      \Vert \partial_{t} v\Vert_{H^{1.5}}^2    \end{split}    \llabel{9x SRVg 0Pl6Z8 9X z fLh GlH IYB x9EQ03}    \end{align} Note that the last term is dominated by $P$. Therefore,   \begin{align}    \begin{split}    \int_{0}^{t}    J_{4}      &\les \emb{         -    \frac12    \int_{\Gamma_1}         \sqrt{g} g^{ij}           \Pi_{\lambda}^{\mu}         \partial_t      \partial_j v^\lambda           \Pi_{\mu}^{\alpha}      \partial_t\partial_{i} v_{\alpha}      \bigm|_{0}^{t} }    + \emb{    \int_{\Gamma_1}    \tilde Q_{\mu\lambda}^{i}(\bar\partial\eta)\bar\partial^2\eta      \partial_t v^{\mu}      \partial_{i} \partial_t v^{\lambda}      \bigm|_{0}^{t} }    \\&\indeq    -    \int_{0}^{t}   \int_{\Gamma_1}      \partial_t         (                \sqrt{g} g^{ij} \Pi_{\lambda}^{\alpha}        ) \partial_j v^\lambda        \partial_{i} \partial_t^2v_{\alpha}    \\&\indeq        -        \int_{0}^{t}      \int_{\Gamma_1}         \partial_t         \Bigl(          \sqrt{g}(g^{ij} g^{kl} - g^{lj}g^{ik} )          \partial_{j}\eta^{\alpha}                  \partial_{k}\eta_{\lambda}         \Bigr)          \partial_{l} v^{\lambda}         \partial_{i}\partial_t^2 v_{\alpha}      +      \int_{0}^{t}P    \\&    = J_{41}+J_{42}+\int_{0}^{t}J_{43}+\int_{0}^{t}J_{44}+\int_{0}^{t}P    .    \end{split}    \llabel{ OELo 5loZ x4wag4 cn F aCE KfA 0uzEQ04}    \end{align} As for $I_{41}$ in the previous section,  the first term $J_{41}$ is the coercive term leading to the second term on the left side of \eqref{EQ23} by simply using \eqref{EQ227}. The second term $J_{42}$ is bounded  in \eqref{EQ68} as $      J_{42}      \les      P_0      + \int_{0}^{t}P \ema{    + \epsilon_0 \Vert \partial_{t} v\Vert_{H^{1.5}}^2 } $. For $J_{43}$ and $J_{44}$, we integrate by parts in time, yielding   \begin{align}    \begin{split}    \int_{0}^{t}J_{43}+\int_{0}^{t}J_{44}     &= \emb{     -    \int_{\Gamma_1}      \partial_t         (                \sqrt{g} g^{ij} \Pi_{\lambda}^{\alpha}        ) \partial_j v^\lambda        \partial_{i} \partial_{t}v_{\alpha}      \restr_{0}^{t} }       +    \int_{0}^{t}   \int_{\Gamma_1}      \partial_t^2        (                \sqrt{g} g^{ij} \Pi_{\lambda}^{\alpha}        ) \partial_j v^\lambda        \partial_{i}  \partial_{t}v_{\alpha}      \\&\indeq       +    \int_{0}^{t}   \int_{\Gamma_1}      \partial_t         (                \sqrt{g} g^{ij} \Pi_{\lambda}^{\alpha}        ) \partial_j \partial_{t}v^\lambda        \partial_{i} \partial_t v_{\alpha}    \\&\indeq \emb{        -        \int_{\Gamma_1}         \partial_t         \Bigl(          \sqrt{g}(g^{ij} g^{kl} - g^{lj}g^{ik} )          \partial_{j}\eta^{\alpha}                  \partial_{k}\eta_{\lambda}         \Bigr)          \partial_{l} v^{\lambda}         \partial_{i}\partial_{t} v_{\alpha}         \restr_{0}^{t} }        \\&\indeq        +        \int_{0}^{t}      \int_{\Gamma_1}         \partial_t^2         \Bigl(          \sqrt{g}(g^{ij} g^{kl} - g^{lj}g^{ik} )          \partial_{j}\eta^{\alpha}                  \partial_{k}\eta_{\lambda}         \Bigr)          \partial_{l} v^{\lambda}         \partial_{i}\partial_t v_{\alpha}        \\&\indeq        +        \int_{0}^{t}      \int_{\Gamma_1}         \partial_t         \Bigl(          \sqrt{g}(g^{ij} g^{kl} - g^{lj}g^{ik} )          \partial_{j}\eta^{\alpha}                  \partial_{k}\eta_{\lambda}         \Bigr)          \partial_{l}\partial_{t} v^{\lambda}         \partial_{i}\partial_{t} v_{\alpha}    \end{split}    \label{EQ107}   \end{align} whence   \begin{align}    \begin{split}    \int_{0}^{t}J_{43}+\int_{0}^{t}J_{44}     &\les     P_0     +     \Vert v\Vert_{H^{2}}^2      \Vert \partial_{t}v\Vert_{H^{1.5}}        +        \int_{0}^{t} P   \\&   \les     P_0     + \ema{    \Vert v\Vert_{H^{1.5}}^{2(1-\nu)/(2-\nu)}    \Vert v\Vert_{H^{2.5-\nu/2}}^{2/(2-\nu)}      \Vert \partial_{t}v\Vert_{H^{1.5}} }        +        \int_{0}^{t} P    \end{split}    \label{EQ192}   \end{align} after a short calculation. Note that both pointwise in time terms in \eqref{EQ107} are estimated by the second term on the far right side of \eqref{EQ192}. The proof of the lemma is thus complete. \end{proof} \gsdfgdsfgdsfg \gsdfgdsfgdsfg \gsdfgdsfgdsfg \startnewsection{Div-Curl estimates and the Cauchy invariance}{sec09} In this section, we use the Cauchy invariance property and the div-curl estimates to the norms $   \Vert v\Vert_{H^{2.5-\nu/2}}$, $   \Vert v\Vert_{H^{3-\nu}}$, and $   \Vert \partial_{t}v\Vert_{H^{1.5}}$ in terms of $\Vert v^3\Vert_{H^{2-\nu/2}(\Gamma_1)}$, $   \Vert v^3\Vert_{H^{2.5-\nu}(\Gamma_1)}$, and $   \Vert \partial_{t}v^3\Vert_{H^{1}(\Gamma_1)}$, respectively. \gsdfgdsfgdsfg \gsdfgdsfgdsfg \gsdfgdsfgdsfg We summarize the resulting inequalities in the following statement. \gsdfgdsfgdsfg \cole \begin{Lemma} \label{L09} For the velocity $v$, we have   \begin{align}    \begin{split}    \Vert v\Vert_{H^{2.5-\nu/2}}    &\les    P_0    +    \Vert v^3\Vert_{H^{2-\nu/2}(\Gamma_1)}    \end{split}    \label{EQ80a}   \end{align} and   \begin{align}    \begin{split}    \Vert v\Vert_{H^{3-\nu}}    &\les    P_0    +    \Vert v^3\Vert_{H^{2.5-\nu}(\Gamma_1)}    .    \end{split}    \label{EQ80b}   \end{align} while for the derivative $\partial_{t}v$, we have   \begin{align}    \begin{split}    \Vert \partial_{t}v\Vert_{H^{1.5}}    &\les    \Vert v\Vert_{H^{2}}^2    +    \Vert \partial_{t}v^3\Vert_{H^{1}(\Gamma_1)}    .    \end{split}    \label{EQ83}   \end{align} \end{Lemma} \colb \gsdfgdsfgdsfg \gsdfgdsfgdsfg \gsdfgdsfgdsfg \gsdfgdsfgdsfg \gsdfgdsfgdsfg {\begin{proof}[Proof of Lemma~\ref{L09}] First, let $     \qq\in\{1.5-\nu/2,2-\nu\} $. By the Cauchy invariance property   \begin{gather}    \epsilon^{\al \be \ga} \partial_\be v^\mu \partial_\ga \eta_\mu     = (\curl v_0)^\al    \label{EQ79}     \end{gather} (cf.~\cite{KukavicaTuffahaVicol-3dFreeEuler}), we obtain   \begin{align}      \begin{split}         (\curl v)^{\alpha}          = \epsilon^{\alpha\beta\gamma}               \partial_{\beta} v_{\gamma} 	 =           \epsilon^{\al \be \ga} \partial_\be v^\mu             ( \delta_{\gamma\mu}-\partial_\ga \eta_\mu)           + (\curl v_0)^{\alpha}    .    \end{split}    \label{EQ177}      \end{align} Similarly, from  the divergence condition \eqref{ZZ14}, we get   \begin{align}      \begin{split}       \dive v       = (\delta^{\alpha\beta}-a^{\alpha\beta})\partial_{\alpha}v_{\beta}    .    \end{split}    \label{EQ178}   \end{align} Using the elliptic estimate   \begin{align}    \begin{split}    \Vert X\Vert_{H^{s}}    \les     \Vert \curl X\Vert_{H^{s-1}}    +    \Vert \dive X\Vert_{H^{s-1}}    +    \Vert X \cdot N\Vert_{H^{s-0.5}(\Gamma_1\cup \Gamma_0)}    \comma s\ge1   \end{split}    \label{EQ30}   \end{align} (\cite{BourguignonBrezis,ShkollerElliptic,CoutandShkollerFreeBoundary}) along with \eqref{ZZ17}, \eqref{EQ177}, and \eqref{EQ178}, we arrive at   \begin{align}    \begin{split}    \Vert v\Vert_{H^{\qq+1}}    &\les    \Vert \nabla \eta-I\Vert_{H^{2-\nu}}    \Vert \nabla v\Vert_{H^{\qq}}    +    \Vert       a - I     \Vert_{H^{2-\nu}}    \Vert \nabla v \Vert_{H^{\qq}}    +    \Vert \curl v_0\Vert_{H^{\qq}}    +    \Vert v^3\Vert_{H^{\qq+0.5}(\Gamma_1)}    .    \end{split}    \label{EQ21}   \end{align} Assuming that the time $T>0$ is sufficiently small as in Lemma~\ref{L01}(iv), the first two terms on the right side of \eqref{EQ21} may be dominated by the left, and we obtain   \begin{align}    \begin{split}    \Vert v\Vert_{H^{\qq+1}}    &\les    \Vert \curl v_0\Vert_{H^{\qq}}    +    \Vert v^3\Vert_{H^{\qq+0.5}(\Gamma_1)}    \comma     \qq\in\left\{1.5-\nu/2,2-\nu\right\}    .    \end{split}    \llabel{ fw HMUV M9Qy eARFe3 Py 6 kQG GFx EQ80}   \end{align} Therefore, we obtain \eqref{EQ80a} and \eqref{EQ80b}. \gsdfgdsfgdsfg Next, we apply the Cauchy invariance to $\partial_{t}v$, i.e.,   \begin{align}    \epsilon^{\al \be \ga}          \partial_\be \partial_{t} v^\mu \partial_\ga \eta_\mu     =    -     \epsilon^{\al \be \ga}          \partial_\be v^\mu \partial_\ga v_\mu    ,    \label{EQ198}   \end{align} obtained by differentiating \eqref{EQ79} in $t$,  which we may rewrite as   \begin{align}    \begin{split}         (\curl \partial_{t}v)^{\alpha}          &= \epsilon^{\alpha\beta\gamma}               \partial_{\beta} \partial_{t}v_{\gamma} 	 =\epsilon^{\alpha\beta\gamma}               \partial_{\beta} \partial_{t}v^{\mu}            (\delta_{\gamma\mu}-\partial_{\gamma}\eta_{\mu})            +           \epsilon^{\alpha\beta\gamma}               \partial_{\beta} \partial_{t}v^{\mu}            \partial_{\gamma}\eta_{\mu}          \\& 	 =           \epsilon^{\al \be \ga} \partial_\be \partial_{t}v^\mu             ( \delta_{\gamma\mu}-\partial_\ga \eta_\mu)          -           \epsilon^{\al \be \ga}              \partial_\be v^\mu \partial_\ga v_\mu    \end{split}    \llabel{rPf 6T ZBQR la1a 6Aeker Xg k blz nEQ111}   \end{align} using \eqref{EQ198} in the last step. On the other hand, the divergence condition for $\partial_{t}v$ may be rewritten as $    \partial^{\beta}    \partial_{t} v_{\beta}    =      (\delta^{\alpha\beta}-a^{\alpha\beta})         \partial_{\alpha}\partial_{t}v_{\beta}    - \partial_{t}a^{\alpha\beta}\partial_{\alpha}v_{\beta} $. Using the div-curl elliptic estimate \eqref{EQ30} with $X=\partial_{t}v$ and $s=1.5$, we get   \begin{align}    \begin{split}    \Vert \partial_{t}v\Vert_{H^{1.5}}    &\les    \Vert \nabla \eta-I\Vert_{H^{2-\nu}}    \Vert \nabla \partial_{t} v\Vert_{H^{0.5}}       +    \Vert       a - I     \Vert_{H^{2-\nu}}    \Vert \nabla \partial_{t} v \Vert_{H^{0.5}}    \\&\indeq    +    \Vert v\Vert_{H^{2}}^2    +    \Vert \partial_{t}a\Vert_{H^{1}}    \Vert v\Vert_{H^{2}}    +    \Vert \partial_{t}v^3\Vert_{H^{1}(\Gamma_1)}    \\&    \les    \Vert \nabla \eta-I\Vert_{H^{2-\nu}}    \Vert \nabla \partial_{t} v\Vert_{H^{0.5}}       +    \Vert       a - I     \Vert_{H^{2-\nu}}    \Vert \nabla \partial_{t} v \Vert_{H^{0.5}}    +    \Vert v\Vert_{H^{2}}^2    +    \Vert \partial_{t}v^3\Vert_{H^{1}(\Gamma_1)}    \end{split}    \llabel{Sm mhY jc z3io WYjz h33sxR JM k DoEQ82}   \end{align} and thus, if $T$ is sufficiently small as in Lemma~\ref{L01}(iv), we obtain \eqref{EQ83}. \end{proof} \gsdfgdsfgdsfg \gsdfgdsfgdsfg \gsdfgdsfgdsfg \gsdfgdsfgdsfg \gsdfgdsfgdsfg \gsdfgdsfgdsfg \gsdfgdsfgdsfg \gsdfgdsfgdsfg \gsdfgdsfgdsfg \gsdfgdsfgdsfg \gsdfgdsfgdsfg \gsdfgdsfgdsfg \gsdfgdsfgdsfg \gsdfgdsfgdsfg \gsdfgdsfgdsfg \startnewsection{Relation between the projection and the normal component of $v$ and $\partial_{t}v$}{sec10} In order to close the estimates, we  need to connect the projections and the normal components of the vector fields $v^3|_{\Gamma_1}$ and $\partial_{t} v^3|_{\Gamma_1}$. We first address the comparison between $\Pi X$ and $X\cdot N$, where $X$ shall be chosen as certain derivative operators of $v$ and $\partial_{t}v$. \gsdfgdsfgdsfg From \eqref{EQ63}, recall that $    \Pi_{\alpha}^{\beta}    =    \delta_{\alpha}^{\beta}    - g^{kl}      \partial_{k}\eta^{\beta}      \partial_{l}\eta_{\alpha} $. Therefore, $     (   \Pi X)^{3}     = \Pi_{\alpha}^{3}X^{\alpha}     = \delta_{\alpha}^{3} X^{\alpha}              - g^{kl}\partial_{l}\eta_{\alpha}\partial_{k}\eta^{3} X^{\alpha} $, from where $    X^{3}    = (\Pi X)^{3} + g^{kl}\partial_{k}\eta_{\alpha}\partial_{l}\eta^{3} X^{\alpha} $. Using $\eta^{3}=\eta^{3}(0)+\int_{0}^{t} v^{3}=x^{3}+\int_{0}^{t}v^{3}$, and thus $\partial_{l}\eta^{3}=\int_{0}^{t}\partial_l v^{3}$, we get   \begin{equation}    X^{3}      = (\Pi X)^{3} + g^{kl}\partial_{k}\eta_{\alpha} X^{\alpha} \int_{0}^{t} \partial_{l} v^{3}      .    \label{EQ148}   \end{equation} \gsdfgdsfgdsfg Applying the formula \eqref{EQ148} with $X=\barpar\partial_{t}v$. From $    \barpar\partial_t v^{3}      = (\Pi \barpar\partial_t v)^{3} + g^{kl}\partial_{k}\eta_{\alpha} \barpar\partial_t v^{\alpha} \int_{0}^{t} \partial_{l} v^{3} $ we obtain   \begin{equation}       \Vert   \barpar\partial_t v^{3}\Vert_{L^2(\Gamma_1)}      \les     \Vert \Pi \barpar\partial_t v\Vert_{L^2(\Gamma_1)}        + \left\Vert g^{kl}\partial_{k}\eta_{\alpha} \barpar\partial_t v^{\alpha} \int_{0}^{t} \partial_{l} v^{3}\right\Vert_{L^2(\Gamma_1)}    .    \llabel{s EAA hUO Oz aQfK Z0cn 5kqYPn W7 1EQ150}   \end{equation} The first term on the right side is estimated in Section~\ref{sec08}. For the second term, we have   \begin{align}      \begin{split}       &\left\Vert g^{kl}\partial_{k}\eta_{\alpha} \barpar\partial_t v^{\alpha} \int_{0}^{t} \partial_{l} v^{3}\right\Vert_{L^2(\Gamma_1)}       \les        \Vert g^{-1}\barpar\eta\Vert_{L^\infty(\Gamma_1)}        \Vert \barpar\partial_{t} v\Vert_{L^2(\Gamma_1)}        \int_{0}^{t}\Vert \barpar v\Vert_{L^\infty(\Gamma_1)}       \\&\indeq             \les        Q(\Vert \eta\Vert_{H^{2.5+\delta_0}})        \Vert \partial_{t} v\Vert_{H^{1.5}}        \int_{0}^{t} \Vert v\Vert_{H^{2.5+\delta_0}}      \les        \Vert \partial_{t} v\Vert_{H^{1.5}}        \int_{0}^{t} \Vert v\Vert_{H^{3-\nu}}      \end{split}    \llabel{ vCT 69a EC9 LD EQ5S BK4J fVFLAo QEQ151}   \end{align} where we used $\Vert \eta\Vert_{H^{2.5+\delta_0}}\les 1$ from Lemma~\ref{L01}(i). Therefore, $    \Vert \barpar\partial_{t} v^{3}\Vert_{L^2(\Gamma_1)}    \les        \Vert \Pi \barpar\partial_t v\Vert_{L^2(\Gamma_1)}        +        \Vert \partial_{t} v\Vert_{H^{1.5}}        \int_{0}^{t} P $. Adding $\Vert \partial_{t}v^{3}\Vert_{L^2(\Gamma_1)}$ to both sides then gives   \begin{align}     \begin{split}      \Vert \partial_{t} v^{3}\Vert_{H^{1}(\Gamma_1)}     &\les        \Vert \Pi \barpar\partial_t v\Vert_{L^2(\Gamma_1)}        +        \Vert \partial_{t}v\Vert_{H^{0.5+\delta_0}}        +        \Vert \partial_{t} v\Vert_{H^{1.5}}        \int_{0}^{t} P     \\&     \les        \Vert \Pi \barpar\partial_t v\Vert_{L^2(\Gamma_1)}        +        \Vert \partial_{t}v\Vert_{L^2}        +        \epsilon_0 \Vert \partial_{t}v\Vert_{H^{1.5}}        +        \Vert \partial_{t} v\Vert_{H^{1.5}}        \int_{0}^{t} P     \\&     \les        \Vert \Pi \barpar\partial_t v\Vert_{L^2(\Gamma_1)}        +        P_0        +        \int_{0}^{t}P        +        \epsilon_0 \Vert \partial_{t}v\Vert_{H^{1.5}}        +        \Vert \partial_{t} v\Vert_{H^{1.5}}        \int_{0}^{t} P        ,   \end{split}    \llabel{p N dzZ HAl JaL Mn vRqH 7pBB qOr7fEQ36}   \end{align} where we used interpolation and Young's inequalities in the second step. We may rewrite the resulting inequality as   \begin{align}     \begin{split}      \Vert \partial_{t} v^{3}\Vert_{H^{1}(\Gamma_1)}     &\les        \Vert \Pi \barpar\partial_t v\Vert_{L^2(\Gamma_1)}        +        \epsilon_0 \Vert \partial_{t}v\Vert_{H^{1.5}}        +        P_0        +         (1+      \Vert \partial_{t} v\Vert_{H^{1.5}})        \int_{0}^{t} P    .   \end{split}    \label{EQ155}   \end{align} Next, we apply \eqref{EQ148} with $X=\barpar\tpar^{1-\nu/2}v$, leading to $    \barpar\tpar^{1-\nu/2}v^{3}    = (\Pi\barpar\tpar^{1-\nu/2}v)^{3}    + g^{kl}\partial_{k}\eta_{\alpha}\barpar\tpar^{1-\nu/2}v^{\alpha}                  \int_{0}^{t}\partial_{l}v^{3} $. Then $     \Vert \barpar\tpar^{1-\nu/2}v^{3}\Vert_{L^2(\Gamma_1)}     \les     \Vert \Pi\barpar\tpar^{1-\nu/2}v\Vert_{L^2(\Gamma_1)}     +     \Vert v\Vert_{H^{2.5-\nu/2}}        \int_{0}^{t} \Vert v\Vert_{H^{2.5+\delta_0}} $. Note that the first term on the right side is estimated in Section~\ref{sec07}. Adding $\Vert v^{3}\Vert_{L^2(\Gamma_1)}$ to both sides gives   \begin{align}     \begin{split}     \Vert v^{3}\Vert_{H^{2-\nu/2}(\Gamma_1)}     &\les     \Vert \Pi\barpar\tpar^{1-\nu/2}v\Vert_{L^2(\Gamma_1)}     +     \Vert v\Vert_{H^{1}}     +     \Vert v\Vert_{H^{2.5-\nu/2}}        \int_{0}^{t} P     \\&     \les     \Vert \Pi\barpar\tpar^{1-\nu/2}v\Vert_{L^2(\Gamma_1)}     +     \epsilon_0         \Vert v\Vert_{H^{2.5-\nu/2}}     +     \Vert v\Vert_{L^2}     +     \Vert v\Vert_{H^{2.5-\nu/2}}        \int_{0}^{t} P    .     \end{split}    \llabel{v oa e BSA 8TE btx y3 jwK3 v244 dlEQ156}    \end{align} \gsdfgdsfgdsfg We rewrite this as   \begin{align}     \begin{split}     \Vert v^{3}\Vert_{H^{2-\nu/2}(\Gamma_1)}     &\les     \Vert \Pi\barpar\tpar^{1-\nu/2}v\Vert_{L^2(\Gamma_1)}     +     \epsilon_0         \Vert v\Vert_{H^{2.5-\nu/2}}        +        P_0        +     (    \Vert v\Vert_{H^{2.5-\nu/2}} + 1)        \int_{0}^{t} P    .     \end{split}    \label{EQ157}    \end{align} \gsdfgdsfgdsfg \gsdfgdsfgdsfg Combining \eqref{EQ80a} with \eqref{EQ157} and choosing $\epsilon_0$ sufficiently small, we get   \begin{align}     \begin{split}    \Vert v\Vert_{H^{2.5-\nu/2}}     &\les     \Vert \Pi\barpar\tpar^{1-\nu/2}v\Vert_{L^2(\Gamma_1)}        +        P_0        +     (    \Vert v\Vert_{H^{2.5-\nu/2}} + 1)        \int_{0}^{t} P    .     \end{split}    \label{EQ38}   \end{align} \colb Combining  \eqref{EQ83} and \eqref{EQ155} with $\epsilon_0>0$ sufficiently small, we obtain   \begin{align}    \begin{split}    \Vert \partial_{t}v\Vert_{H^{1.5}}    &\les        \Vert \Pi \barpar\partial_t v\Vert_{L^2(\Gamma_1)}        +        \Vert v\Vert_{H^{2}}^2        +        P_0        +         (1+      \Vert \partial_{t} v\Vert_{H^{1.5}})        \int_{0}^{t} P     \\&      \les        \Vert \Pi \barpar\partial_t v\Vert_{L^2(\Gamma_1)}        +        \Vert v\Vert_{H^{1.5}}^{2(1-\nu)/(2-\nu)}        \Vert v\Vert_{H^{2.5-\nu/2}}^{2/(2-\nu)}        \\&\indeq        +        P_0        +         (1+      \Vert \partial_{t} v\Vert_{H^{1.5}})        \int_{0}^{t} P    .    \end{split}    \label{EQ83a}   \end{align} \gsdfgdsfgdsfg \gsdfgdsfgdsfg \gsdfgdsfgdsfg \startnewsection{The concluding estimates}{sec11} \gsdfgdsfgdsfg Now, we are ready to combine all the available inequalities to prove Theorem~\ref{T01}. \gsdfgdsfgdsfg {\begin{proof}[Proof of Theorem~\ref{T01}] Squaring \eqref{EQ38} and using \eqref{EQ135}, we get   \begin{align}    \begin{split}    E_0^2    &\les     P_0     +(E_0^2+1)\int_{0}^{t} P    .    \end{split}    \label{EQ133}   \end{align} Also, combining \eqref{EQ103} and \eqref{EQ83a}, we get   \begin{align}      \begin{split}    E_1    &    \les    \epsilon_0 E    +    E\int_{0}^{t}P        +    \left(P_0+\int_{0}^{t}P\right)    \biggl(     G^{3(1+2\delta_0)/2(3-\nu)}    +      H^{3/4}    +      E^{1/(2-\nu)}    + 1    \biggr)      \end{split}    \llabel{fwRL Dc g X14 vTp Wd8 zy YWjw eQmFEQ126a}    \end{align} from where, using Young's inequality,   \begin{align}      \begin{split}    E_1    &    \les    
\epsilon_0 E    +    E\int_{0}^{t}P        +    \left(P_0+\int_{0}^{t}P\right)    \biggl(     G^{3(1+2\delta_0)/2(3-\nu)}    +      H^{3/4}    + 1    \biggr)      \end{split}    \label{EQ126}    \end{align} Finally, we add \eqref{EQ133} and \eqref{EQ126} and choose $\epsilon_0$ sufficiently small so we can absorb $2\epsilon_0 E$, obtaining   \begin{align}    \begin{split}    E    &\les    E\int_{0}^{t}P    +       \left(P_0+\int_{0}^{t}P\right)    \biggl(      G^{3(1+2\delta_0)/2(3-\nu)}      +       H^{3/4}     +     1    \biggr)    .    \end{split}    \label{EQ140}   \end{align} \gsdfgdsfgdsfg \gsdfgdsfgdsfg \gsdfgdsfgdsfg \gsdfgdsfgdsfg \gsdfgdsfgdsfg \gsdfgdsfgdsfg \gsdfgdsfgdsfg \gsdfgdsfgdsfg \gsdfgdsfgdsfg \gsdfgdsfgdsfg \gsdfgdsfgdsfg Now, we turn to establishing the control of  $\Vert v\Vert_{H^{3-\nu}}$. From  \cite{DisconziKukavicaIncompressible}, recall the identity   \begin{align}    \begin{split}      &   \sqrt{g} g^{ij} \partial^2_{ij} v^3 - \sqrt{g} g^{ij} \Gamma_{ij}^k \partial_k v^3     \\&\indeq    =        - \partial_t(\sqrt{g} g^{ij} ) \partial^2_{ij} \eta^3       - \partial_t( \sqrt{g} g^{ij} \Gamma_{ij}^k ) \partial_k \eta^3       -       \partial_t a^{\mu 3} N_\mu q       -     a^{\mu 3} N_\mu \partial_t q \, \text{ on } \, \Gamma_1    ,   \end{split}    \llabel{ yD5y5l DN l ZbA Jac cld kx Yn3V QEQ144}   \end{align} which follows from differentiating \eqref{ZZ16} in $t$ and setting $\alpha=3$. We rewrite the equation above as   \begin{align}    \begin{split}         \Delta v^3 &=    (\delta^{ij}-\sqrt{g} g^{ij}) \partial^2_{ij} v^3    +    \sqrt{g} g^{ij} \Gamma_{ij}^k \partial_k v^3     - \partial_t(\sqrt{g} g^{ij} ) \partial^2_{ij} \eta^3     \\&\indeq    - \partial_t( \sqrt{g} g^{ij} \Gamma_{ij}^k ) \partial_k \eta^3     -       \partial_t a^{\mu 3} N_\mu q       -     a^{\mu 3} N_\mu \partial_t q \, \text{ on } \, \Gamma_1   \end{split}    \llabel{YIV v6fwmH z1 9 w3y D4Y ezR M9 BduEQ211}   \end{align} from where, by ellipticity,   \begin{align}    \begin{split}    \Vert v^{3}\Vert_{H^{2.5-\nu}(\Gamma_1)}    &\les    \Vert          (\delta^{ij}-\sqrt{g} g^{ij}) \partial^2_{ij} v^3    \Vert_{H^{0.5-\nu}(\Gamma_1)}    +    \Vert         \sqrt{g} g^{ij} \Gamma_{ij}^k \partial_k v^3     \Vert_{H^{0.5-\nu}(\Gamma_1)}    \\&\indeq    +    \Vert \partial_t(\sqrt{g} g^{ij} ) \partial^2_{ij} \eta^3\Vert_{H^{0.5-\nu}(\Gamma_1)}    +     \Vert      \partial_t( \sqrt{g} g^{ij} \Gamma_{ij}^k ) \partial_k \eta^3    \Vert_{H^{0.5-\nu}(\Gamma_1)}    \\&\indeq    +\Vert \partial_t a^{\mu 3} N_\mu q\Vert_{H^{0.5-\nu}(\Gamma_1)}    +    \Vert a^{\mu 3} N_\mu \partial_t q\Vert_{H^{0.5-\nu}(\Gamma_1)}    .    \llabel{E L7D9 2wTHHc Do g ZxZ WRW Jxi pv EQ193}    \end{split}    \end{align} Using  \eqref{EQ213}, we get   \begin{align}    \begin{split}    \Vert v^{3}\Vert_{H^{2.5-\nu}(\Gamma_1)}    &\les    \sum_{i,j}     \bigl(     \Vert \delta^{ij}-\sqrt{g} g^{ij} \Vert_{L^\infty}     +     \Vert \delta^{ij}-\sqrt{g} g^{ij} \Vert_{H^{1}(\Gamma_1)}     \bigr)     \Vert v^{3}\Vert_{H^{2.5-\nu}(\Gamma_1)}     \\&\indeq     +    \Vert       \sqrt{g} g^{ij} \Gamma_{ij}^k     \Vert_{H^{0.5}(\Gamma_1)}    \Vert       \partial_k v^3     \Vert_{H^{1-\nu}(\Gamma_1)}    \\&\indeq    +    \bigl(      \Vert \partial_t(\sqrt{g} g^{ij} ) \Vert_{L^\infty}      +      \Vert \partial_t(\sqrt{g} g^{ij} ) \Vert_{H^{1}(\Gamma_1)}    \bigr)    \Vert \partial^2_{ij} \eta^3\Vert_{H^{0.5-\nu}(\Gamma_1)}    \\&\indeq    +     \Vert \partial_t( \sqrt{g} g^{ij} \Gamma_{ij}^k ) \Vert_{H^{0.5-\nu}(\Gamma_1)}     \bigl(    \Vert \partial_k \eta^3 \Vert_{L^{\infty}}    +    \Vert \partial_k \eta^3 \Vert_{H^{1}(\Gamma_1)}    \bigr)    \\&\indeq    +\Vert \partial_t a^{\mu 3} N_\mu \Vert_{H^{1-\nu}(\Gamma_1)}     \Vert q\Vert_{H^{0.5}(\Gamma_1)}    +    \Vert a^{\mu 3} N_\mu \Vert_{H^{1-\nu}(\Gamma_1)}    \Vert \partial_t q\Vert_{H^{0.5}(\Gamma_1)}    .    \end{split}    \llabel{fz48 ZVB7 FZtgK0 Y1 w oCo hLA i70 EQ212}       \end{align} Since $     \Vert \delta^{ij}-\sqrt{g} g^{ij} \Vert_{L^\infty}     +     \Vert \delta^{ij}-\sqrt{g} g^{ij} \Vert_{H^{1}}     \leq     \epsilon_0     $ by Lemma~\ref{L01}(iv) (ensuring that $T\leq 1/C M \epsilon_0$), and using {(recall that $\eta(0)$ is the identity)} $    \Vert \partial_k \eta^3 \Vert_{L^{\infty}}    +    \Vert \partial_k \eta^3 \Vert_{H^{1}(\Gamma_1)}    \leq    \epsilon_0 $, also by Lemma~\ref{L01}(iv), we get   \begin{align}    \begin{split}    \Vert v^{3}\Vert_{H^{2.5-\nu}(\Gamma_1)}    &\les      \epsilon_0      \Vert v^{3}\Vert_{H^{3-\nu}}    +   \Vert q\Vert_{H^{1}}    +   \Vert \partial_{t} q\Vert_{H^{1}}    +\Vert v\Vert_{H^{2.5+\delta_0}}    + P_0 + \int_{0}^{t} P    \\&    \les    P_0    +    \Vert \partial_{t} q\Vert_{H^{1}}    +    \epsilon_0      \Vert v\Vert_{H^{3-\nu}}    +    \int_{0}^{t}P    .    \end{split}    \label{EQ216}   \end{align} (Note that $\Vert q\Vert_{H^{1}}\les P_0+\int_{0}^{t}P$ and $\Vert v\Vert_{H^{2.5+\delta_0}}\les \epsilon_0 \Vert v\Vert_{H^{3-\nu}} +P_0+\int_{0}^{t}P$.) \gsdfgdsfgdsfg \gsdfgdsfgdsfg \gsdfgdsfgdsfg \gsdfgdsfgdsfg \gsdfgdsfgdsfg Combining \eqref{EQ216} with \eqref{EQ80b} and setting $\epsilon_0>0$ sufficiently small, we obtain   \begin{align}    \begin{split}    F    &=    \Vert v\Vert_{H^{3-\nu}}    \les    P_0    +    \Vert \partial_{t} q\Vert_{H^{1}}    +    \int_{0}^{t} P    \les    H    +    P_0    +    \int_{0}^{t} P    ,    \end{split}    \label{EQ35}   \end{align} while by \eqref{EQ140} we have   \begin{align}    \begin{split}    E    &\les    E\int_{0}^{t}P    +       \left(P_0+\int_{0}^{t}P\right)    \biggl(      G^{3(1+2\delta_0)/2(3-\nu)}      +       H^{3/4}     +     1    \biggr)    .    \end{split}    \label{EQ142}   \end{align} Also, \eqref{EQ127} reads   \begin{align}     \begin{split}    G    &     \les        \left(         P_0         + \int_{0}^{t} P        \right)           E^{(3-\nu)/3}    .   \end{split}    \label{EQ139}   \end{align} Substituting \eqref{EQ139} in \eqref{EQ142} and using Young's inequality then yields   \begin{align}    \begin{split}    E    &\les    E\int_{0}^{t}P    +       \left(P_0+\int_{0}^{t}P\right)    \biggl(      H^{3/4}     +     1    \biggr)    .    \end{split}    \label{EQ194}   \end{align} Next, we have \eqref{EQ138}, which is   \begin{align}   \begin{split}    H &     \les      E      +      \left(P_0+\int_{0}^{t}P\right)      \Bigl(      E^{(7+2\nu)/6}      +      E^{1/2}      G^{(1+\nu+2\delta_0)/(3-\nu)}      +        F^{\nu/(1-\nu)}        E^{(1-2\nu)/(2-2\nu)}      \Bigr)    .    \llabel{NO Ta06 u2sY GlmspV l2 x y0X B37 xEQ141}   \end{split}   \end{align} The inequality  \eqref{EQ139} then gives   \begin{align}   \begin{split}    H &     \les      E      +      \left(P_0+\int_{0}^{t}P\right)      \Bigl(      E^{(7+2\nu)/6}      +      E^{(5+2\nu+4\delta_0)/6}      +        F^{\nu/(1-\nu)}        E^{(1-2\nu)/(2-2\nu)}      \Bigr)     \\&     \les      E      +      \left(P_0+\int_{0}^{t}P\right)      \Bigl(      E^{(7+2\nu)/6}      +        F^{\nu/(1-\nu)}        E^{(1-2\nu)/(2-2\nu)}      \Bigr)    \label{EQ195}   \end{split}   \end{align} where we used Young's inequality  and $(5+2\nu+4\delta_0)/6<1$ in the last step. Replacing \eqref{EQ35} into \eqref{EQ195}, we get   \begin{align}     \begin{split}      H      &\les      E      +      \left(P_0+\int_{0}^{t}P\right)      \Bigl(      E^{(7+2\nu)/6}      +        H^{\nu/(1-\nu)}        E^{(1-2\nu)/(2-2\nu)}      \Bigr)     \end{split}    \llabel{43 k5 kaoZ deyE sDglRF Xi 9 6b6 w9EQ25}    \end{align} from where, using Young's inequality to absorb $H^{\nu/(1-\nu)}$ into the left side (note that $\nu/(1-\nu)<1$ by the restriction on $\nu$), we get   \begin{align}      \begin{split}       H       &\les      E      +      \left(P_0+\int_{0}^{t}P\right)      \Bigl(      E^{(7+2\nu)/6}      +      E^{1/2}      \Bigr)      \les      E      +      \left(P_0+\int_{0}^{t}P\right)       E^{(7+2\nu)/6}     .      \end{split}      \label{EQ24}     \end{align} We need to combine this inequality with \eqref{EQ194}. Observe that   \begin{equation}    \frac{7+2\nu}{6} \frac{3}{4} < 1,    \label{EQ197}   \end{equation} which follows from $0\leq \nu<1/2$. Thus, we may choose $   \tilde\epsilon>0$ such that   \begin{equation}     \frac{7+2\nu}{6}     \frac{3}{4}     (1+\tilde\epsilon)     <1    \label{EQ196}   \end{equation} Then replacing  \eqref{EQ24}  in  \eqref{EQ194}, we get   \begin{align}      \begin{split}      E       &\les      E\int_{0}^{t}P      +      \left(P_0+\int_{0}^{t}P\right)      \Bigl(      E^{3/4}      +      E^{3(7+2\nu)/24}      \Bigr)      \les      E\int_{0}^{t}P      +      \left(P_0+\int_{0}^{t}P\right)      \Bigl(        E^{3/4}        + E^{1/(1+\tilde\epsilon)}      \Bigr)      \end{split}    \label{EQ26}     \end{align} where we used \eqref{EQ196} in the last step. Using Young's inequality on \eqref{EQ26}, we get   \begin{align}      \begin{split}      E       &\les      P_0      +      E\int_{0}^{t}P      \end{split}    \label{EQ33}     \end{align} Note that $P$ here and below depends on $E$, $F$, $G$, and $H$, i.e., $P=P(E,F,G,H)$. The inequality  \eqref{EQ33} is combined with \eqref{EQ24}, i.e.,     \begin{align}      \begin{split}       H      \les      \left(P_0+\int_{0}^{t}P\right)       E^{(7+2\nu)/6}      .      \end{split}    \label{EQ27}     \end{align} In addition, we have an inequality for $F$, which is \eqref{EQ35} with \eqref{EQ27} applied to it,    \begin{align}      \begin{split}       F       \les       \left(P_0+\int_{0}^{t}P\right)       E^{(7+2\nu)/6}    .      \end{split}    \label{EQ29}    \end{align} Finally, by \eqref{EQ139}, we have   \begin{align}     \begin{split}    G    &     \les        \left(         P_0         + \int_{0}^{t} P        \right)           E^{(3-\nu)/3}    .   \end{split}    \label{EQ31}   \end{align} A barrier technique applied to \eqref{EQ33}--\eqref{EQ31} then  leads to the boundedness of $E$, $F$, $G$, and $H$ for a sufficiently small $T>0$ and the proof is concluded. \end{proof} \gsdfgdsfgdsfg \gsdfgdsfgdsfg \gsdfgdsfgdsfg \gsdfgdsfgdsfg \gsdfgdsfgdsfg \gsdfgdsfgdsfg \gsdfgdsfgdsfg \startnewsection{The case of a general domain}{sec12} \def\OP{\tilde{\Omega}} \def\intt{\int_{\OP}} In this section, we show how to adapt the ideas used to prove Theorem \ref{T01},  where the initial surface was flat,  to the case of a general bounded domain. The physical situation which we have  in mind is that of a water droplet with surface tension. In this case the fluid domain does not have a rigid bottom, and thus only equations \eqref{ZZ13}--\eqref{ZZ16} are considered. Note however that the presence of a rigid bottom can also be handled with  minor modifications. If $U$ is a domain in $\mathbb{R}^3$, $\Vert \partial  U \Vert_{s}$ is the $H^s$ norm of the boundary of the domain, defined in the usual way via local representations as graphs. \gsdfgdsfgdsfg \cole \begin{Theorem} \label{T02} Let $\sigma >0$ and $\epsilon\in[0,1/2)$. Assume that $v_0$ is a smooth divergence-free vector field on  a bounded domain $\Omega \subset \mathbb{R}^3$ with smooth boundary $\Gamma$, and denote by $N$ the unit outer normal to $\Gamma$. Then there exist $C_*>0$  and $T_{*}>0$, depending only on $\Vert v_0\Vert_{H^{2.5+\epsilon}}$, $\Vert v_0 \cdot N \Vert_{H^{2.5}(\Gamma)}$, $\sigma>0$, and $\Vert\Gamma\Vert_{H^{3.75+\epsilon/2}}$, such that any smooth solution $(v,q)$ to  \eqref{ZZ13}--\eqref{ZZ16} with the initial condition $v_0$ and defined on the time interval $[0,T_*]$ satisfies   \begin{align}    \begin{split}     \Vert v\Vert_{H^{2.5+\epsilon}}     + \Vert \partial_{t}v\Vert_{H^{1.5}}     + \Vert \partial_{t}^2v\Vert_{L^2}     + \Vert q\Vert_{H^{2.25+\epsilon/2}}     + \Vert \partial_{t}q\Vert_{H^{1}}     \leq C_*.    \end{split}    \label{EQ12}   \end{align}   Moreover,    $\Vert \Gamma(t)\Vert_{H^{3 +\epsilon}} \leq C_*$   for $t \in [0,T_*]$, where $\Gamma(t) = \eta(t)(\Gamma)$. \end{Theorem} \colb \gsdfgdsfgdsfg As in Theorem \ref{T01}, the dependence of $C_*$ and $T_*$ on $\Vert v_0 \cdot N \Vert_{H^{2.5}(\Gamma)}$ occurs to guarantee that  $\partial_t^2 v$ belongs to $L^2$ at time zero. More precisely, solving for $\partial_t^2 v(0)$ in terms of $v(0)$ and $q(0)$ as in Remark \ref{R01},  we can bound $\partial_t^2 v(0)$ in $L^2$ in terms of the initial data if $v_0 \cdot N \in H^{2.5}(\Gamma)$. However, instead of solving for time-differentiated quantities in terms of the initial data to determine regularity conditions on the latter, many times it is preferable to directly state the a priori estimate upon the assumption that the energy we seek to bound is finite at time zero, as done for example in \cite{CoutandShkollerFreeBoundary}. Therefore, introducing  \begin{gather} N(t) =     \Vert v(t)\Vert_{H^{2.5+\epsilon}}     + \Vert \partial_{t}v(t)\Vert_{H^{1.5}}     + \Vert \partial_{t}^2v(t)\Vert_{L^2}     + \Vert q(t)\Vert_{H^{2.25+\epsilon/2}}     + \Vert \partial_{t}q(t)\Vert_{H^{1}}, \nonumber \end{gather} we have the following. \gsdfgdsfgdsfg \cole \begin{Theorem} \label{T03} Let $\sigma >0$ and $\epsilon\in[0,1/2)$. Assume that $v_0$ is a smooth divergence-free vector field on  a bounded domain $\Omega \subset \mathbb{R}^3$ with smooth boundary $\Gamma$. Then there exist $C_*>0$  and $T_{*}>0$ depending only on $N(0)$,  $\sigma>0$, such that  any smooth solution $(v,q)$ to  \eqref{ZZ13}--\eqref{ZZ16} with the initial condition $v_0$ and defined on the time interval $[0,T_*]$ satisfies $N(t) \leq C_*$. \end{Theorem} \colb \gsdfgdsfgdsfg We remark that Theorem \ref{T02} entails some derivative loss for the boundary, i.e., a $H^{3.75+\epsilon/2}$ initial boundary  $\Gamma$ yields only a $H^{3+\epsilon}$ moving boundary $\Gamma(t)$. This loss of regularity is known to be prevented in $H^s$ for $s \geq 4$ \cite{CoutandShkollerFreeBoundary, ShatahZengGeometry}. It seems challenging, however, to avoid some loss of derivatives  for the boundary evolution when working in such low regularity spaces as presented here. It should be stressed, however, that some regularity of the boundary is propagated, namely, $\Gamma(t)$ is in $H^{3+\epsilon}$, thus more regular than the flow $\eta|_{\Gamma}$ which is guaranteed to be only in $H^{2 + \epsilon}(\Gamma)$. \gsdfgdsfgdsfg We now turn to the proof of Theorems \ref{T02} and \ref{T03}.  The crucial observation is that in appropriate coordinates that flatten the boundary near a point, the equations take exactly the same form as \eqref{ZZ13}--\eqref{ZZ16}, with $\partial_i$, for $i=1,2$, being tangent to the boundary, as in the case of the domain \eqref{EQ181}. \gsdfgdsfgdsfg More precisely,  given $y_0 \in \partial \Omega$, we take coordinates that flatten the boundary near $y_0$. This means that there exist $r,R>0$  and a diffeomorphism $\Psi \colon B_R(0,0,1) \cap \{ x^3 \leq 1 \} \to B_r(y_0) \cap \Omega$ such that (after a rigid motion and relabeling the coordinates if necessary) we have $\Psi(x^1,x^2,x^3) = (x^1,x^2, x^3 + \psi(x^1,x^2))$, where $\psi\colon B_R(0) \cap \{ x^3 = 1 \} \to \mathbb{R}$ is a smooth function. Note that $\det D\Psi = \det D \Psi^{-1} = 1$. Consider the Lagrangian map $\eta\colon \Omega  \to \Omega(t)$, and set $\tilde{\eta} = \eta \circ \Psi$, which is defined in the domain of $\Psi$. Then $\partial_t \tilde{\eta} = \partial_t \eta \circ \Psi = u \circ \eta \circ \Psi = u \circ \tilde{\eta}$, where $u$ is the Eulerian velocity, i.e., the velocity in the moving domain $\Omega(t)$. It follows that if we introduce $\tilde{v} = u \circ \tilde{\eta}$ and  $\tilde{q}= p \circ \tilde{\eta}$, where $p$ is the Eulerian pressure, then $\tilde{v}$ and $\tilde{q}$ satisfy equations \eqref{ZZ13}--\eqref{ZZ16} with all variables replaced by their respective $\,\tilde{\,}\,$ counter-parts -- except that these equations are now defined  only locally, i.e., in $B_R(0) \cap \{ x^3 \leq 1 \}$. We thus  use suitably chosen cut-off functions to produce local estimates, passing to a global estimate by a simple addition procedure.  In order to simplify the exposition, we will omit tildes from all quantities and continue to label $\eta$, $v$, and $q$, which are  only locally defined, the Lagrangian map, velocity, and pressure, respectively. \gsdfgdsfgdsfg We need expressions for  $\eta(0)$, $a(0)$, and $g_{ij}(0)$, which now are slightly more complicated than in the case of the domain \eqref{EQ181}. We have \begin{align} \begin{split} &\eta(0,x) = (x^1, x^2, x^3 + \psi(x^1,x^2)), \, \, \partial_i \eta^\mu(0) = \delta_i^\mu + \delta^{\mu 3} \partial_i \psi, \, \, g_{ij}(0) = \delta_{ij} + \partial_i \psi \partial_j \psi,\\ &\text{and } \, g(0) = 1 + (\partial_1 \psi)^2 + (\partial_2 \psi)^2, \end{split} \nonumber \end{align} where we recall that $g$ is the determinant of $(g_{ij})$. Also,  \begin{align} \begin{split} g^{-1}(0) = \frac{1}{ 1 + (\partial_1 \psi)^2 + (\partial_2 \psi)^2} \begin{bmatrix} 1 + (\partial_2 \psi)^2 & -\partial_1 \psi \partial_2 \psi \\ -\partial_1 \psi \partial_2 \psi & 1 + (\partial_1 \psi)^2 \end{bmatrix}, a(0)=  \begin{bmatrix} 1 & 0 & 0 \\ 0 & 1 & 0 \\ -\partial_1 \psi & -\partial_2 \psi & 1 \end{bmatrix}. \end{split} \nonumber \end{align} \gsdfgdsfgdsfg In the proof of Theorem \ref{T01}, for which $\psi \equiv 0$,  we used the above quantities at time zero to produce some small parameters in the energy estimates. In order to apply the same argument here, we need $\nabla \psi$ to be small. This can be achieved as follows. Without loss of generality we may assume that $\nabla \psi(0,0,1) = 0$. Reducing $R$ and invoking the mean value theorem, we may make $\Vert \nabla \psi \Vert_{L^\infty(\Gamma)}$ as small as we wish provided that  $\psi$ is bounded in $H^{2+\delta}$, where $\delta > 0$, which is consistent with Theorem \ref{T02}. Note that the compactness of $\Gamma$ assures that we may take  $R\geq R_0$ for some fixed $R_0$. \gsdfgdsfgdsfg \gsdfgdsfgdsfg \gsdfgdsfgdsfg \gsdfgdsfgdsfg \gsdfgdsfgdsfg \gsdfgdsfgdsfg \gsdfgdsfgdsfg We shall derive estimates near the point $(0,0,1)$, with the variables defined in the ball of radius $R/2$, where $R>0$ is as introduced above in the construction of the local parameterization of $\Omega$. Let $\theta$ be a smooth cut-off function such that $0\leq\theta\leq1$ with $\theta\equiv 1$ on $\bar B_{R/5}(0,0,1)$ and  $\supp\theta\subseteq B_{R/4}(0,0,1)$. In what follows, all integrands carry a cut-off function of this type. Therefore, extending  all quantities to be identically zero outside $ B_{R/4}(0,0,1)$, we may consider the equations and variables defined on the domain $\OP =\mathbb{T}^{2} \times [0,1]$. This will make it easier to adapt the estimates from Section \ref{sec07}. Also, as in that section, we shall denote the upper boundary of $\OP$ by $\Gamma_1$ and the lower 
boundary by $\Gamma_0$. However, unlike Section \ref{sec07}, no integral over $\Gamma_0$ is present since all variables vanish there in view of the way they have been extended. \gsdfgdsfgdsfg We now apply the energy estimates of Section \ref{sec07} with ,   \begin{equation}    \EE    = \tpar^{1-\nu/2}(\theta \,\cdot\,).    \label{EQ199}   \end{equation}    obtaining    \begin{align}    \begin{split}    \frac12    \frac{d}{dt}    \Vert\EE\partial_{t}v\Vert_{L^2}^2    &=    - \intt\EE\partial_{t}(a^{\mu\alpha}\partial_{\mu}q)\EE \partial_{t}v_{\alpha}    =    -    \intt    \tpar^{1-\nu/2} (\theta \partial_{t}(a^{\mu\alpha}\partial_{\mu}q))    \tpar^{1-\nu/2}(\theta\partial_{t} v_{\alpha})    \\&    =    -    \intt    \theta \partial_{t}(a^{\mu\alpha}\partial_{\mu}q)    \tpar^{2-\nu}(\theta\partial_{t} v_{\alpha})    \\&    =    \intt    \theta \partial_{t}(a^{\mu\alpha}q)    \tpar^{2-\nu}(\theta\partial_{t} \partial_{\mu}v_{\alpha})    -    \intu    \theta \partial_{t}( N_{\mu} a^{\mu\alpha}q)    \tpar^{2-\nu}(\theta\partial_{t}v_{\alpha})    \\&\indeq    +       \intt    \partial_{\mu}\theta \partial_{t}(a^{\mu\alpha}q)    \tpar^{2-\nu}(\theta\partial_{t} v_{\alpha})    +       \intt    \theta \partial_{t}(a^{\mu\alpha}q)    \tpar^{2-\nu}( \partial_{\mu}\theta\partial_{t} v_{\alpha})    ,    \end{split}    \llabel{B dId Ko gSUM NLLb CRzeQL UZ m i9OEW88a}   \end{align} from where   \begin{align}    \begin{split}    \frac12    \frac{d}{dt}    \Vert\EE\partial_{t}v\Vert_{L^2}^2    &=      \intt\EE\partial_{t}(a^{\mu\alpha}q)\EE \partial_{t}\partial_{\mu}v_{\alpha}    - \intu\EE\partial_{t}(N_{\mu}a^{\mu\alpha}q)\EE \partial_{t}v_{\alpha}    \\&\indeq    +       \intt    \tpar^{0.5-\nu}(    \partial_{\mu}\theta \partial_{t}(a^{\mu\alpha}q))    \tpar^{1.5}(\theta\partial_{t} v_{\alpha})    +       \intt    \tpar^{0.5-\nu}(    \theta \partial_{t}(a^{\mu\alpha}q))    \tpar^{1.5}(\partial_{\mu}\theta\partial_{t} v_{\alpha})    .    \end{split}    \label{EW88}   \end{align} By \eqref{EW88}, we have $    \frac12    \frac{d}{dt}    \Vert\tpar^{1-\nu/2}(\theta\partial_{t}v)\Vert_{L^2}^2    = I_1 + I_2 + I_3 + I_4 + I_5 $, where   \begin{align}    \begin{split}    I_1    &=    \intt\tpar^{1-\nu/2}(\theta\partial_{t}a^{\mu\alpha}q)     \tpar^{1-\nu/2}(\theta \partial_{t}\partial_{\mu}v_{\alpha})    \\    I_2    &=     \intt      a^{\mu\alpha}\tpar^{2-\nu}(\theta\partial_{t}q)\theta\partial_{t}\partial_{\mu}v_{\alpha}    \\    I_3    &=    \intt    \Bigl(      \tpar^{2-\nu}(\theta a^{\mu\alpha}\partial_{t}q)      - a^{\mu\alpha}\tpar^{2-\nu}(\theta\partial_{t}q)    \Bigr)    \theta   \partial_{t}\partial_{\mu}v_{\alpha}    \\    I_4    &=    - \int_{\Gamma_1}          \tpar^{1-\nu/2}(\theta \partial_{t}( N_{\mu}a^{\mu\alpha}q))          \tpar^{1-\nu/2} (\theta\partial_{t} v_{\alpha})    \\    I_5    &=    \intt    \tpar^{0.5-\nu}(    \partial_{\mu}\theta \partial_{t}(a^{\mu\alpha}q))    \tpar^{1.5}(\theta\partial_{t} v_{\alpha})    +       \intt    \tpar^{0.5-\nu}(    \theta \partial_{t}(a^{\mu\alpha}q))    \tpar^{1.5}(\partial_{\mu}\theta\partial_{t} v_{\alpha})    .    \end{split}    \llabel{ 2qv VzD hz v1r6 spSl jwNhG6 s6 i EW91}   \end{align} The first term is rewritten as   \begin{align}    \begin{split}     I_1     &=    \intt\tpar^{1.5-\nu}(\theta\partial_{t}a^{\mu\alpha}q)    \tpar^{0.5} (\theta\partial_{t}\partial_{\mu}v_{\alpha})     \les     \bigl\Vert      \tpar^{1.5-\nu}       (\theta\partial_{t}a q)     \bigr\Vert_{L^2}       \Vert\theta\partial_{t}\nabla v\Vert_{H^{0.5}}    .    \end{split}    \llabel{SdX hob hbp 2u sEdl 95LP AtrBBi bPEQ200}   \end{align} Now, let $\bar\theta$ be a smooth cut-off function such that $0\leq\bar\theta\leq1$ with $    \supp \bar\theta\subseteq B_{R/3}(0,0,1) $ and $    \bar\theta\equiv 1    \inon{on~$\supp \theta$} $. We need this cut-off function for an application of the fractional product rule below, as each separate term needs to be properly cut-off. Having $\bar{\theta} \equiv 1$ on $\supp\theta$ assures that we may introduce $\bar\theta$ without altering  given expressions. We have   \begin{align}    \begin{split}     &     \bigl\Vert      \tpar^{1.5-\nu}       (\theta\partial_{t}a q)     \bigr\Vert_{L^2}     =     \bigl\Vert      \tpar^{1.5-\nu}       (\theta\partial_{t}a  \bar\theta q)     \bigr\Vert_{L^2}     \les     \Vert       \theta       \partial_{t}a     \Vert_{H^{2-\nu}}     \Vert \bar\theta q\Vert_{H^{1}}     +     \Vert         \theta  \partial_{t}a     \Vert_{H^{1}}     \Vert \bar\theta q\Vert_{H^{2-\nu}}     ,    \end{split}    \llabel{ C wSh pFC CUa yz xYS5 78ro f3UwDPEQ204}   \end{align} where we used the fractional product rule. Also,   \begin{align}    \begin{split}     &\Vert\theta\partial_{t}\nabla v\Vert_{H^{0.5}}     \les     \Vert\nabla(\theta\partial_{t}v)\Vert_{H^{0.5}}         +     \Vert\nabla\theta\partial_{t}v\Vert_{H^{0.5}}         \les     \Vert\bar\theta\theta\partial_{t}v\Vert_{H^{1.5}}         +     \Vert\bar \theta\nabla\theta\partial_{t}v\Vert_{H^{0.5}}        \\&\indeq     \les     \Vert \theta\Vert_{H^{1.5+\delta_0}}     \Vert\bar\theta\partial_{t}v\Vert_{H^{1.5}}         +     \Vert \nabla \theta\Vert_{H^{1.5+\delta_0}}     \Vert\bar \theta\partial_{t}v\Vert_{H^{0.5}}         \les     \Vert\bar \theta\partial_{t}v\Vert_{H^{1.5}}        \end{split}    \llabel{ sC I pES HB1 qFP SW 5tt0 I7oz jXuEQ28}   \end{align} Therefore, we get $    I_1    \les     \Vert       \theta       \partial_{t}a     \Vert_{H^{2-\nu}}     \Vert \bar\theta q\Vert_{H^{1}}     \Vert\bar\theta\partial_{t}v\Vert_{H^{1.5}}     +     \Vert         \theta  \partial_{t}a     \Vert_{H^{1}}     \Vert \bar\theta q\Vert_{H^{2-\nu}}     \Vert\bar\theta\partial_{t}v\Vert_{H^{1.5}} $. Next, by  the divergence-free condition \eqref{ZZ14} we have   \begin{align}    \begin{split}    I_2    &=     - \intt\partial_{t}a^{\mu\alpha}\tpar^{2-\nu}(\theta\partial_{t}q)\theta\partial_{\mu}v_{\alpha}       =     -    \intt \tpar^{1-\nu}(\partial_{t}a^{\mu\alpha}\theta\partial_{\mu}v_{\alpha})     \tpar(\theta\partial_{t}q)    \\&    =     -    \intt \tpar^{1-\nu}(\bar\theta\partial_{t}a^{\mu\alpha}\theta\partial_{\mu}v_{\alpha})     \tpar(\theta\partial_{t}q)    \les    \Vert \tpar^{1-\nu}(\bar\theta\partial_{t}a^{\mu\alpha}\theta\partial_{\mu}v_{\alpha})\Vert_{L^2}    \Vert \theta\partial_{t} q\Vert_{H^{1}}    \end{split}    \llabel{n6c z4 c QLB J4M NmI 6F 08S2 Il8C WQ94}   \end{align} and thus, using the fractional chain rule, $    I_2    \les    \Vert \theta\partial_{t}a^{\mu\alpha}\Vert_{H^{1.5}}    \Vert   \theta\partial_{\mu}v_{\alpha}\Vert_{H^{1-\nu}}    \Vert \theta\partial_{t} q\Vert_{H^{1}} $. For $I_3$, we have, as in \eqref{EQ116},   \begin{align}    \begin{split}    I_3    &    =    \intt    \Bigl(      \tpar^{2-\nu}(\bar\theta a^{\mu\alpha}\theta\partial_{t}q)      - \bar\theta a^{\mu\alpha}\tpar^{2-\nu}(\theta\partial_{t}q)    \Bigr)    \theta   \partial_{t}\partial_{\mu}v_{\alpha}    \\&    \les       \Vert \tpar^{2-\nu} (\bar\theta a^{\mu\alpha}\theta\partial_t q)                         -\bar\theta a^{\mu\alpha}\tpar^{2-\nu}(\theta\partial_t q)       \Vert_{L^{3/2}}      \Vert \theta\partial_{t}\partial_{\mu}v\Vert_{L^{3}}    \\&    \les    \Vert \bar\theta a \Vert_{H^{2-\nu}}    \Vert \theta\partial_{t} q \Vert_{H^{1}}    \Vert \theta \partial_{t} \partial_{\mu}v\Vert_{L^{3}}    \les    \Vert \bar\theta a \Vert_{H^{2-\nu}}    \Vert \theta\partial_{t} q \Vert_{H^{1}}    \Vert \bar\theta \partial_{t} v\Vert_{H^{1.5}}    \end{split}    \llabel{0JQYiU lI 1 YkK oiu bVt fG uOeg SlEQ206}   \end{align} where we used   \begin{align}   \begin{split}     \Vert \theta\partial_{t}\partial_{\mu}v\Vert_{L^{3}}         & \leq     \Vert \partial_{\mu}(\theta\partial_{t}v)\Vert_{L^{3}}        +     \Vert \partial_{\mu}\theta\partial_{t}v\Vert_{L^{3}}     =     \Vert \partial_{\mu}(\theta\partial_{t}v)\Vert_{L^{3}}        +     \Vert \bar\theta \partial_{\mu}\theta \partial_{t}v\Vert_{L^{3}}     \\&     \les     \Vert \partial_{\mu}(\theta\partial_{t}v)\Vert_{H^{0.5}}        +     \Vert \bar\theta\partial_{t}v\Vert_{L^{3}}     \les     \Vert \theta\partial_{t}v\Vert_{H^{1.5}}        +     \Vert \bar\theta\partial_{t}v\Vert_{H^{0.5}}     \les     \Vert \bar\theta\partial_{t}v\Vert_{H^{1.5}}      \end{split}    \llabel{lv b4HGn3 bS Z LlX efa eN6 v1 B6m3EQ207}   \end{align} in the last step. \gsdfgdsfgdsfg Before treating the most difficult term $I_4$, we bound the lower order term $I_5$ which is the sum of two terms, denoted by $I_{51}$ and $I_{52}$. For the first one, we write   \begin{align}    \begin{split}    I_{51}    &=      \intt     \tpar^{0.5-\nu}(\partial_{\mu}\theta\partial_{t}(a^{\mu\alpha}q))     \tpar^{1.5}(\theta \partial_{t}v_{\alpha})    \les     \Vert     \tpar^{0.5-\nu}(\partial_{\mu}\theta\partial_{t}(\bar\theta a^{\mu\alpha}\bar\theta q))     \Vert_{L^2}     \Vert     \tpar^{1.5}(\theta \partial_{t}v_{\alpha})     \Vert_{L^2}     \\&     \les     \Vert     \theta     \Vert_{H^{2.5+\delta_0}}     \Vert \bar\theta \partial_{t} a\Vert_{H^{1.5-\nu/2}}         \Vert \bar\theta q\Vert_{H^{0.5-\nu/2}}     \Vert     \bar\theta \partial_{t}v_{\alpha}     \Vert_{H^{1.5}}     \\&\indeq     +     \Vert     \theta     \Vert_{H^{2.5+\delta_0}}     \Vert \bar\theta a\Vert_{H^{1.5-\nu/2}}         \Vert \bar\theta \partial_{t}q\Vert_{H^{0.5-\nu/2}}     \Vert     \bar\theta \partial_{t}v_{\alpha}     \Vert_{H^{1.5}}    \end{split}    \llabel{ Ek3J SXUIjX 8P d NKI UFN JvP Ha VEQ208}   \end{align} while for the second one we have similarly   \begin{align}    \begin{split}    I_{52}    &=      \intt     \tpar^{0.5-\nu}(\theta\partial_{t}(a^{\mu\alpha}q))     \tpar^{1.5}(\partial_{\mu}\theta \partial_{t}v_{\alpha})    \les     \Vert     \tpar^{0.5-\nu}(\theta\partial_{t}(\bar\theta a^{\mu\alpha}\bar\theta q))     \Vert_{L^2}     \Vert     \tpar^{1.5}(\partial_{\mu}\theta \partial_{t}v_{\alpha})     \Vert_{L^2}     \\&     \les     \Vert     \theta     \Vert_{H^{1.5+\delta_0}}     \Vert \bar\theta \partial_{t} a\Vert_{H^{1.5-\nu/2}}         \Vert \bar\theta q\Vert_{H^{0.5-\nu/2}}     \Vert     \bar\theta \partial_{t}v_{\alpha}     \Vert_{H^{1.5}}     +     \Vert     \theta     \Vert_{H^{1.5+\delta_0}}     \Vert \bar\theta a\Vert_{H^{1.5-\nu/2}}         \Vert \bar\theta \partial_{t}q\Vert_{H^{0.5-\nu/2}}     \Vert     \bar\theta \partial_{t}v_{\alpha}     \Vert_{H^{1.5}}    .    \end{split}    \llabel{r4T eARP dXEV7B xM 0 A7w 7je p8M 4EQ209}   \end{align} Now, we turn to the term $I_4$, for which we modify  the considerations in Section~\ref{sec06}. With $\EE$ defined in \eqref{EQ199}, we first obtain the first equality in \eqref{EQ62}, i.e.,   \begin{align}    \begin{split}    I_4    &=    \int_{\Gamma_1}     \EE \partial_{t}v_{\alpha}     \EE        \partial_{i}        \Bigl(         \sqrt{g} g^{ij} (\de^\al_\lambda -g^{kl} \partial_k \eta^\al \partial_l \eta_\lambda)          \partial_j v^\lambda          + \sqrt{g}(g^{ij} g^{kl} - g^{lj}g^{ik} ) \partial_j \eta^\al \partial_k\eta_\lambda         \partial_l v^\lambda         \Bigr)     \\&     =     -    \int_{\Gamma_1}      \partial_{i} \EE \partial_{t}v_{\alpha}     \EE       \Bigl(         \sqrt{g} g^{ij} (\de^\al_\lambda -g^{kl} \partial_k \eta^\al \partial_l \eta_\lambda)           \partial_j v^\lambda         \Bigr)        \\&\indeq       -       \int_{\Gamma_1}         \partial_{i}   \EE \partial_{t}v_{\alpha}        \EE        \Bigl(         \sqrt{g}(g^{ij} g^{kl} - g^{lj}g^{ik} ) \partial_j \eta^\al \partial_k\eta_\lambda         \partial_l v^\lambda         \Bigr)     \\&\indeq     -    \int_{\Gamma_1}       \EE \partial_{t}v_{\alpha}     \tpar^{1-\nu/2}       \Bigl(         \partial_{i}\theta         \sqrt{g} g^{ij} (\de^\al_\lambda -g^{kl} \partial_k \eta^\al \partial_l \eta_\lambda)           \partial_j v^\lambda         \Bigr)     \\&\indeq       -       \int_{\Gamma_1}          \EE \partial_{t}v_{\alpha}        \tpar^{1-\nu/2}        \Bigl(         \partial_{i}\theta         (          \sqrt{g}(g^{ij} g^{kl} - g^{lj}g^{ik} ) \partial_j \eta^\al \partial_k\eta_\lambda           \partial_l v^\lambda          )        \Bigr)     =     I_{41}     + I_{42}     + I_{43}     + I_{44}    .   \end{split}    \llabel{Q ahOi hEVo Pxbi1V uG e tOt HbP tsEQ210}   \end{align} Using \eqref{EQ63}, we rewrite   \begin{align}    \begin{split}    I_{41}    &=     -    \intu     \tpar^{1-\nu/2}       \Bigl(         \theta         \sqrt{g} g^{ij} \Pi_{\lambda}^{\alpha}           \partial_j v^\lambda         \Bigr)      \partial_{i} \tpar^{1-\nu/2}(\theta \partial_{t}v_{\alpha})    \\&    =     -    \intu     \tpar^{1-\nu/2}       \Bigl(         \bar\theta         \sqrt{g} g^{ij} \Pi_{\lambda}^{\alpha}           \theta\partial_j v^\lambda         \Bigr)      \tpar^{1-\nu/2}(\theta \partial_{i}\partial_{t}v_{\alpha})     \\&\indeq     -    \intu     \tpar^{1-\nu/2}       \Bigl(         \theta         \sqrt{g} g^{ij} \Pi_{\lambda}^{\alpha}           \partial_j v^\lambda         \Bigr)       \tpar^{1-\nu/2}(\partial_{i}\theta \partial_{t}v_{\alpha})    \\&    =     -    \intu       \bar\theta       \sqrt{g} g^{ij} \Pi_{\lambda}^{\alpha}     \tpar^{1-\nu/2}       \Bigl(           \theta\partial_j v^\lambda         \Bigr)      \tpar^{1-\nu/2}(\theta \partial_{i}\partial_{t}v_{\alpha})     \\&\indeq     -    \intu     \biggl(     \tpar^{1-\nu/2}       \Bigl(         \bar\theta         \sqrt{g} g^{ij} \Pi_{\lambda}^{\alpha}           \theta\partial_j v^\lambda         \Bigr)       -        \bar\theta        \sqrt{g} g^{ij} \Pi_{\lambda}^{\alpha}       \tpar^{1-\nu/2}       (           \theta\partial_j v^\lambda         )       \biggr)      \tpar^{1-\nu/2}(\theta \partial_{i}\partial_{t}v_{\alpha})     \\&\indeq     -    \intu     \tpar^{1-\nu/2}       \Bigl(         \theta         \sqrt{g} g^{ij} \Pi_{\lambda}^{\alpha}           \partial_j v^\lambda         \Bigr)       \tpar^{1-\nu/2}(\partial_{i}\theta \partial_{t}v_{\alpha})    =    I_{411}    +  I_{412}    + I_{413}    .    \end{split}    \llabel{O 5r 363R ez9n A5EJ55 pc L lQQ Hg6EW64}   \end{align} Using $\Pi_{\lambda}^{\alpha}=\Pi_{\mu}^{\alpha}\Pi_{\lambda}^{\mu}$, the first term equals   \begin{align}    \begin{split}     I_{411}    &=    -    \intu      \bar\theta         \sqrt{g} g^{ij}           \Pi_{\lambda}^{\mu}         \tpar^{1-\nu/2}(\theta    \partial_j v^\lambda  )         \Pi_{\mu}^{\alpha}      \tpar^{1-\nu/2}( \theta\partial_{t}\partial_{i}v_{\alpha})    \\&    =    -    \intu      \bar\theta         \sqrt{g} g^{ij}           \Pi_{\lambda}^{\mu}         \tpar^{1-\nu/2}\partial_{j}(\theta    v^\lambda  )         \Pi_{\mu}^{\alpha}      \tpar^{1-\nu/2}\partial_{i}( \theta\partial_{t}v_{\alpha})     +\lot    \\&    =    -    \frac12    \frac{d}{dt}    \int_{\Gamma_1}        \bar\theta         \sqrt{g} g^{ij}           \Pi_{\lambda}^{\mu}          \partial_{j} \tpar^{1-\nu/2}(\theta v^{\lambda})         \Pi_{\mu}^{\alpha}      \partial_{i}     \tpar^{1-\nu/2}(\theta v_{\alpha})    +    \lot    \end{split}    \llabel{ X1J EW K8Cf 9kZm 14A5li rN 7 kKZ WQ65}   \end{align} It is easy to check that $I_{412}$ and $I_{413}$ constitute lower order terms. We thus obtain   \begin{align}    \begin{split}     I_{41}     &      =    -    \frac12    \frac{d}{dt}    \int_{\Gamma_1}        \bar\theta         \sqrt{g} g^{ij}           \Pi_{\lambda}^{\mu}          \partial_{j} \tpar^{1-\nu/2}(\theta v^{\lambda})         \Pi_{\mu}^{\alpha}      \partial_{i}     \tpar^{1-\nu/2}(\theta v_{\alpha})     +     \lot     \\&     =    -    \frac12    \frac{d}{dt}    \int_{\Gamma_1}        \sqrt{g(0)} g^{ij}        \bar\theta        \,          \Pi_{\lambda}^{\mu}          \partial_{j} \tpar^{1-\nu/2}(\theta v^{\lambda})        \,         \Pi_{\mu}^{\alpha}      \partial_{i}     \tpar^{1-\nu/2}(\theta v_{\alpha})     \\&\indeq    -    \frac12    \frac{d}{dt}    \int_{\Gamma_1}        \bar\theta         \Bigl(            \sqrt{g} g^{ij} - \sqrt{g(0)}g^{ij}(0)         \Bigr)          \Pi_{\lambda}^{\mu}          \partial_{j} \tpar^{1-\nu/2}(\theta v^{\lambda})         \Pi_{\mu}^{\alpha}      \partial_{i}     \tpar^{1-\nu/2}(\theta v_{\alpha})     +     \lot    \end{split}    \llabel{rY0 K10 It eJd3 kMGw opVnfY EG 2 oWQ66}   \end{align} The first term on the right hand side leads to the needed coercive term, providing the control of the $H^{2-\nu/2}(\Gamma)$ norm of $\Pi v$. The second term is, after the time integration, dominated by the coercive term  by Lemma~\ref{L01}(iv). As in Section~\ref{sec06}, we have   \begin{align}    \begin{split}    I_{42}    &=      -       \int_{\Gamma_1}         \bar\theta         \sqrt{g}(g^{ij} g^{kl} - g^{lj}g^{ik} ) \partial_j \eta^\al \partial_k\eta_\lambda           \partial_{l}\EE v^{\lambda}         \partial_{i}\EE\partial_{t} v_{\alpha}     \\&\indeq      -       \int_{\Gamma_1}        \biggl(         \EE         \bigl(         \bar\theta          \sqrt{g}(g^{ij} g^{kl} - g^{lj}g^{ik} ) \partial_j \eta^\al \partial_k\eta_\lambda  \partial_l v^\lambda          \bigr)         -         \bar\theta          \sqrt{g}(g^{ij} g^{kl} - g^{lj}g^{ik} ) \partial_j \eta^\al \partial_k\eta_\lambda          \partial_{l}\EE v^{\lambda}        \biggr)         \partial_{i}\EE\partial_{t} v_{\alpha}     \\&     =     I_{421}     +        I_{422}    .    \end{split}    \llabel{rG fj0 TTA Xt ecJK eTM0 x1N9f0 lR WQ67}   \end{align} Also, as in Section~\ref{sec06}, we have   \begin{align*}     I_{421}     = -      \int_{\Gamma_{1}} \frac{\bar\theta}{\sqrt{g}}            \bigl(                  \partial_{t} \det A^{1} + \det A^{2}  + \det A^{3}           \bigr)   \end{align*} where   \begin{align*}     A^{1} =      \begin{pmatrix}     \partial_{1} \eta_{\mu} \partial_{1} \EE v^{\mu} & \partial_{1} \eta_{\mu} \partial_{2}      \\     \partial_{2} \eta_{\mu} \partial_{1} \EE v^{\mu} & \partial_{2} \eta_{\mu}\partial_{2} \EE v^{\mu}     \end{pmatrix},     A^{2} =      \begin{pmatrix}     \partial_{1} v_{\mu}\partial_{1} \EE v^{\mu} & \partial_{1} \eta_{\mu} \partial_{2} \EE v^{\mu} \\     \partial_{2} v_{\mu} \partial_{1} \EE v^{\mu} & \partial_{2} \eta_{\mu} \partial_{2} \EE v^{\mu}     \end{pmatrix},     A^{3} =      \begin{pmatrix}     \partial_{1} \eta_{\mu} \partial_{1} \EE v^{\mu} & \partial_{1}  v_{\mu}  \partial_{2} \EE v^{\mu} \\     \partial_{2} \eta_{\mu} \partial_{1} \EE v^{\mu} & \partial_{2} v_{\mu} \partial_{2} \EE v^{\mu}     \end{pmatrix},   \end{align*} and we obtain    \begin{align}   \begin{split}    I_{421}    &=           -      \int_{\Gamma_{1}}        \partial_{t}       \left(        \frac{\bar\theta}{\sqrt{g}} \det A^{1}       \right)       +      \int_{\Gamma_{1}}        \partial_{t}             \left(        \frac{\bar\theta}{\sqrt{g}}       \right)       \det A^{1}       
-  \int_{\Gamma_{1}} \frac{\bar\theta}{\sqrt{g}} \det A^{2}       -        \int_{\Gamma_{1}}          \frac{\bar\theta}{\sqrt{g}} \det A^{3}    \\&    = I_{4211}       + I_{4212}      + I_{4213}      + I_{4214}              .   \end{split}    \llabel{p QkP M37 3r0 iA 6EFs 1F6f 4mjOB5 WQ71}   \end{align} As above,    \begin{equation}    \left\Vert     \partial_{t}     \left(      \fractext{\bar\theta}{\sqrtg}     \right)    \right\Vert_{L^\infty(\Gamma_1)}    \les    P(\Vert \eta\Vert_{H^{2.5+\delta_0}})    \Vert v\Vert_{H^{2.5+\delta_0}}    \llabel{zu 5 GGT Ncl Bmk b5 jOOK 4yny My04WQ72}   \end{equation} and $|\det A^{1}|\les |\barpar \eta|^2  (\EE\barpar v)^2$. Therefore,   \begin{equation}     I_{4212}    \les     P(\Vert \eta\Vert_{H^{2.5+\delta_0}})          \Vert v\Vert_{H^{2.5+\delta_0}}     \Vert \EE \barpar v\Vert_{L^2(\Gamma_1)}^2     \les     P(\Vert \eta\Vert_{H^{2.5+\delta_0}})          \Vert v\Vert_{H^{2.5+\delta_0}}     \Vert \EE v\Vert_{H^{1.5}}^2    \llabel{oz 6m 6 Akz NnP JXh Bn PHRu N5Ly qWQ73}   \end{equation} as well as   \begin{equation}    \left|     \int_{\Gamma_1}     \frac{1}{\sqrt{g}}       (\det A^{2}+\det A^{3})    \right|    \les     P(\Vert \eta\Vert_{H^{2.5+\delta_0}})          \Vert \EE \barpar v\Vert_{L^2(\Gamma_1)}^2     \les     P(\Vert \eta\Vert_{H^{2.5+\delta_0}})          \Vert \EE v\Vert_{H^{1.5}}^2    .    \llabel{Sguz5 Nn W 2lU Yx3 fX4 hu LieH L30WQ74}   \end{equation} The rest is the same as in Section~\ref{sec06}.  We integrate by parts and write   \begin{align}    \begin{split}    \int_{\Gamma_1}    \frac{\bar\theta}{\sqrtg}    \det A^{1}    &=    \int_{\Gamma_1}    \frac{\bar\theta}{\sqrtg}    \bigl(     \partial_{1}\eta_{\mu}\partial_{2}\eta_{\lambda}     \partial_{1}\EE v^{\mu}     \partial_{2}\EE v^{\lambda}     -     \partial_{1}\eta_{\mu}\partial_{2}\eta_{\lambda}     \partial_{2}\EE v^{\mu}     \partial_{1}\EE v^{\lambda}    \bigr)    \\&    =    \int_{\Gamma_1}    \frac{\bar\theta}{\sqrtg}    \bigl(     -     \partial_{1}\eta_{\mu}\partial_{2}\eta_{\lambda}     \EE v^{\mu}     \partial_{1}\partial_{2}\EE v^{\lambda}     +     \partial_{1}\eta_{\mu}\partial_{2}\eta_{\lambda}     \EE v^{\mu}     \partial_{2}\partial_{1}\EE v^{\lambda}    \bigr)    -    \int_{\Gamma_1}      \bar\theta      Q_{\mu\lambda}^{i}(\bar\partial\eta,\bar\partial^2\eta)      \EE v^{\mu}      \partial_{i} \EE v^{\lambda}    \\&    =    -    \int_{\Gamma_1}      \bar\theta      Q_{\mu\lambda}^{i}(\bar\partial\eta,\bar\partial^2\eta)      \EE v^{\mu}      \partial_{i} \EE v^{\lambda}    \end{split}    \llabel{w g93Xwc gj 1 I9d O9b EPC R0 vc6A WQ75}   \end{align} where $Q_{\mu\lambda}^{i}(\bar\partial\eta,\bar\partial^2\eta)$ is a rational function, which is linear in $\bar\partial^2\eta$. Therefore,   \begin{equation*}    I_{4211}    =    \frac{d}{dt}    \int_{\Gamma_1}    \bar\theta    \tilde Q_{\mu\lambda}^{i}(\bar\partial\eta)\bar\partial^2\eta      \EE v^{\mu}      \partial_{i} \EE v^{\lambda}    ,    \end{equation*} and we obtain   \begin{equation}    I_{421}    \les    \frac{d}{dt}    \int_{\Gamma_1}    \bar\theta    \tilde Q_{\mu\lambda}^{i}(\bar\partial\eta)\bar\partial^2\eta      \EE v^{\mu}      \partial_{i} \EE v^{\lambda}     +     P(\Vert \eta\Vert_{H^{2.5+\delta_0}})          (\Vert v\Vert_{H^{2.5+\delta_0}}+1)     \Vert \EE v\Vert_{H^{1.5}}    .    \llabel{005Q VFy1ly K7 o VRV pbJ zZn xY dcWQ78}   \end{equation} Thus we have shown how to adapt the result in  Section~\ref{sec07} to the case of the curved domain. \gsdfgdsfgdsfg After covering $\Gamma$ with finitely many balls $\{ B_{r_\ell}(y_\ell) \}_{\ell=0}^N$, the procedure described above yields the desired estimates near the boundary. In order to obtain the full estimate, we need to bound the solution in the region of $\Omega$ not covered by $B = \bigcup_{\ell=0}^N B_{r_\ell}(y_\ell)$. This is done by covering $\Omega \backslash B$ with further open sets and again reducing the problem to estimates on $\mathbb{T}^2 \times [0,1]$. However, for these estimates no integrals on either $\Gamma_1$ or $\Gamma_2$ will appear. \gsdfgdsfgdsfg Using again cut-off functions and the local parameterization of $\Om$ described above,  the $L^2$ estimate for $\partial_{t}^2v$ in Section~\ref{sec08} is easily adapted to the present situation  since only an integer number of derivatives is used in those estimates. The later sections,  including the div-curl estimates and the Cauchy invariance property, are also easily adaptable.  This establishes Theorems \ref{T02} and \ref{T03}, except for the statement $\Vert \Gamma(t)\Vert_{H^{3 +\epsilon}} \leq C_*$, which we now prove. \gsdfgdsfgdsfg Let $y_0 \in \eta(\Omega)$. We choose coordinates  $(y^1,y^2,y^3)$ in the ambient Euclidean space such that, possibly after a rigid motion and relabeling of the coordinates, $y_0$ is identified with the origin and $\eta(\Omega)$ is locally given by a graph $y^3 = h(y^1, y^2)$. Denote by $\Sigma$ the portion of $\eta(\Omega)$ that is written as the graph of $h$. We can further assume that $\partial_{y^i}$, for $i=1,2$, are tangent to $\Sigma$ at $y_0=(0,0,0)$ and that $\partial_{y^1} h(y_0) = \partial_{y^2} h (y_0)= 0$. \gsdfgdsfgdsfg Recall that we denote by $\cH\colon \eta(\Om) \to \mathbb{R}$ the mean curvature of $\eta(\Om)$. In terms of local coordinates $(x^1,x^2,x^3)$ near $\eta^{-1}(y_0)$ we have the known formula $-\Delta_g \eta^\alpha = \cH \circ \eta \, n^\alpha \circ \eta$, where $n$ is the unit outer normal to $\eta(\Om)$. Contracting with $n^\alpha \circ \eta$, invoking \eqref{ZZ16} and \eqref{EQ12} (which is the part of Theorem \ref{T02} that has already been established) we have $\Vert \cH \Vert_{H^{1+\epsilon}} \leq C_*$ (here, and in what follows, we relabel the constant $C_*$ if necessary). On the other hand, setting  $w=h-y^3$, we have the following expression for the mean curvature expressed in  $y$-coordinates: \begin{gather} A^{ij}(\nabla w) \partial_i \partial_j w = \frac{1}{|\nabla w|}\left( \delta^{ij} - \frac{\partial^i w \partial^j w}{|\nabla w|^2} \right) \partial_i \partial_j w = \cH \circ h, \label{mean_curv_graph} \end{gather} where $\partial^i = \delta^{ik} \partial_k$. From the way we constructed $h$, we have $A^{ij}(y_0) = \delta^{ij}$. We already know that  $\Vert \Si \Vert_{H^{2+\epsilon}} \leq C_*$ since we have a bound for $\eta$, thus we may assume that $\Vert w \Vert_{2+\ep} \leq C_*$. It follows that $A^{ij}$ is uniformly elliptic near the origin and bounded in $C^{0,\beta}$ for some $0 < \beta < \epsilon$. Elliptic regularity then implies that $\Vert w \Vert_{H^{3+\epsilon}} \leq C_*$, as desired. \gsdfgdsfgdsfg We remark that the application of elliptic theory in the previous paragraph is not entirely immediate, and has to be carried out in steps due to the low regularity of the coefficients $A^{ij}$. First, one uses Schauder theory and the embedding $H^{1+\epsilon}(\Sigma) \subset C^{0,\beta}(\Si)$ to conclude that $w$ is in $C^{2,\beta}$. Then the coefficients $A^{ij}$ are in fact $C^{1,\beta}$. Using that the right hand side of \eqref{mean_curv_graph} is in $H^1$ we can then apply $L^p$ estimates to obtain $w \in H^3$. Thus, $A^{ij}$ is now in $H^2$, and we can interpolate between estimates for elliptic operators with coefficients in Sobolev spaces of integer order to finally conclude the result. \gsdfgdsfgdsfg \gsdfgdsfgdsfg \small \bibliographystyle{plain} \bibliography{References.bib} \gsdfgdsfgdsfg 
\def\cprime{$'$}
\begin{thebibliography}{10}

\bibitem{AlazardAboutGlobalExistence}
T.~Alazard.
\newblock About global existence and asymptotic behavior for two dimensional
  gravity water waves.
\newblock In {\em S\'eminaire {L}aurent {S}chwartz---\'{E}quations aux
  d\'eriv\'ees partielles et applications. {A}nn\'ee 2012--2013}, S\'emin.
  \'Equ. D\'eriv. Partielles, pages Exp. No. XVIII, 16. \'Ecole Polytech.,
  Palaiseau, 2014.

\bibitem{AlazardStabilizationSurfaceTension}
T.~Alazard.
\newblock Stabilization of the {W}ater-{W}ave {E}quations with {S}urface
  {T}ension.
\newblock {\em Ann. PDE}, 3(2):3:17, 2017.

\bibitem{AlazardCapillaryWaterWaves}
T.~Alazard and P.~Baldi.
\newblock Gravity capillary standing water waves.
\newblock {\em Arch. Ration. Mech. Anal.}, 217(3):741--830, 2015.

\bibitem{AlazardWaterWaveSurfaceTension}
T.~Alazard, N.~Burq, and C.~Zuily.
\newblock On the water-wave equations with surface tension.
\newblock {\em Duke Math. J.}, 158(3):413--499, 2011.

\bibitem{AlazardDispersiveSurfaceTension}
T.~Alazard, N.~Burq, and C.~Zuily.
\newblock Strichartz estimates for water waves.
\newblock {\em Ann. Sci. \'Ec. Norm. Sup\'er. (4)}, 44(5):855--903, 2011.

\bibitem{AlazardetalLowregularity}
T.~Alazard, N.~Burq, and C.~Zuily.
\newblock Low regularity {C}auchy theory for the water-waves problem: canals
  and wave pools.
\newblock In {\em Lectures on the analysis of nonlinear partial differential
  equations. {P}art 3}, volume~3 of {\em Morningside Lect. Math.}, pages 1--42.
  Int. Press, Somerville, MA, 2013.

\bibitem{AlazardCollectionWaterWaves}
T.~Alazard, N.~Burq, and C.~Zuily.
\newblock The water-wave equations: from {Z}akharov to {E}uler.
\newblock In {\em Studies in phase space analysis with applications to {PDE}s},
  volume~84 of {\em Progr. Nonlinear Differential Equations Appl.}, pages
  1--20. Birkh\"auser/Springer, New York, 2013.

\bibitem{AlazardCauchyWaterWaves}
T.~Alazard, N.~Burq, and C.~Zuily.
\newblock On the {C}auchy problem for gravity water waves.
\newblock {\em Invent. Math.}, 198(1):71--163, 2014.

\bibitem{AlazardCauchyTheoryWaterWaves}
T.~Alazard, N.~Burq, and C.~Zuily.
\newblock Cauchy theory for the gravity water waves system with non-localized
  initial data.
\newblock {\em Ann. Inst. H. Poincar\'e Anal. Non Lin\'eaire}, 33(2):337--395,
  2016.

\bibitem{AlazardDelortGlobal2dWater}
T.~Alazard and J.-M. Delort.
\newblock Global solutions and asymptotic behavior for two dimensional gravity
  water waves.
\newblock {\em Ann. Sci. \'Ec. Norm. Sup\'er. (4)}, 48(5):1149--1238, 2015.

\bibitem{AmbroseVortexSheets}
D.~M. Ambrose.
\newblock Well-posedness of vortex sheets with surface tension.
\newblock {\em SIAM J. Math. Anal.}, 35(1):211--244 (electronic), 2003.

\bibitem{AmbroseMasmoudiWaterWaves}
D.~M. Ambrose and N.~Masmoudi.
\newblock The zero surface tension limit of two-dimensional water waves.
\newblock {\em Comm. Pure Appl. Math.}, 58:1287--1315, 2005.

\bibitem{Beale_et_al_Growth}
J.~T. Beale, T.~Y. Hou, and J.~S. Lowengrub.
\newblock Growth rates for the linearized motion of fluid interfaces away from
  equilibrium.
\newblock {\em Comm. Pure Appl. Math.}, 46(9):1269--1301, 1993.

\bibitem{BieriWu1}
Lydia Bieri, Shuang Miao, Sohrab Shahshahani, and Sijue Wu.
\newblock On the {M}otion of a {S}elf-{G}ravitating {I}ncompressible {F}luid
  with {F}ree {B}oundary.
\newblock {\em Comm. Math. Phys.}, 355(1):161--243, 2017.

\bibitem{BourgainLi15}
Jean Bourgain and Dong Li.
\newblock Strong ill-posedness of the incompressible {E}uler equation in
  borderline {S}obolev spaces.
\newblock {\em Invent. Math.}, 201(1):97--157, 2015.

\bibitem{BourguignonBrezis}
J.~P. Bourguignon and H.~Brezis.
\newblock Remarks on the {E}uler equation.
\newblock {\em Journal of Functional Analysis}, 15:341--363, 1974.

\bibitem{FeffermanetallSplashSurfaceTension}
A.~Castro, D.~C{\'o}rdoba, C.~Fefferman, F.~Gancedo, and J.~G{\'o}mez-Serrano.
\newblock Finite time singularities for water waves with surface tension.
\newblock {\em J. Math. Phys.}, 53(11):115622, 26, 2012.

\bibitem{FeffermanetallSplash}
A.~Castro, D.~C{\'o}rdoba, C.~Fefferman, F.~Gancedo, and J.~G{\'o}mez-Serrano.
\newblock Finite time singularities for the free boundary incompressible
  {E}uler equations.
\newblock {\em Ann. of Math. (2)}, 178(3):1061--1134, 2013.

\bibitem{FeffermanStructural}
A.~Castro, D.~C{\'o}rdoba, C.~Fefferman, F.~Gancedo, and J.~G{\'o}mez-Serrano.
\newblock Structural stability for the splash singularities of the water waves
  problem.
\newblock {\em Discrete Contin. Dyn. Syst.}, 34(12):4997--5043, 2014.

\bibitem{FeffermanetallMuskat}
A.~Castro, D.~C{\'o}rdoba, C.~Fefferman, F.~Gancedo, and
  M.~L{\'o}pez-Fern{\'a}ndez.
\newblock Rayleigh-{T}aylor breakdown for the {M}uskat problem with
  applications to water waves.
\newblock {\em Ann. of Math. (2)}, 175(2):909--948, 2012.

\bibitem{ShkollerElliptic}
C.-H.~A. Chen, D.~Coutand, and S.~Shkoller.
\newblock Solvability and regularity for an elliptic system prescribing the
  curl, divergence, and partial trace of a vector field on {S}obolev-class
  domains.
\newblock {\em J. Math. Fluid Mech.}, 19(3):375--422, 2017.

\bibitem{ShkollerVortexSheets}
Ching-Hsiao~Arthur Cheng, Daniel Coutand, and Steve Shkoller.
\newblock On the motion of vortex sheets with surface tension in
  three-dimensional {E}uler equations with vorticity.
\newblock {\em Comm. Pure Appl. Math.}, 61(12):1715--1752, 2008.

\bibitem{ChristodoulouLindbladFree}
D.~Christodoulou and H.~Lindblad.
\newblock On the motion of the free surface of a liquid.
\newblock {\em Comm. Pure Appl. Math.}, 53(12):1536--1602, 2000.

\bibitem{CF}
Peter Constantin and Ciprian Foias.
\newblock {\em Navier-{S}tokes equations}.
\newblock Chicago Lectures in Mathematics. University of Chicago Press,
  Chicago, IL, 1988.

\bibitem{CoutandShkollerFreeBoundary}
D.~Coutand and S.~Shkoller.
\newblock Well-posedness of the free-surface incompressible {E}uler equations
  with or without surface tension.
\newblock {\em J. Amer. Math. Soc.}, 20(3):829--930, 2007.

\bibitem{MR2660719}
D.~Coutand and S.~Shkoller.
\newblock A simple proof of well-posedness for the free-surface incompressible
  {E}uler equations.
\newblock {\em Discrete Contin. Dyn. Syst. Ser. S}, 3(3):429--449, 2010.

\bibitem{CoutandShkollerSplash}
D.~Coutand and S.~Shkoller.
\newblock On the finite-time splash and splat singularities for the 3-{D}
  free-surface {E}uler equations.
\newblock {\em Commun. Math. Phys.}, 325:143--183, 2014.

\bibitem{CoutandSingularity}
Daniel Coutand.
\newblock Finite-time singularity formation for incompressible {E}uler moving
  interfaces in the plane.
\newblock {\em Arch. Ration. Mech. Anal.}, 232(1):337--387, 2019.

\bibitem{CraigHamiltonianWaterWaves}
W.~Craig.
\newblock On the {H}amiltonian for water waves.
\newblock {\em arXiv:1612.08971 [math.AP]}, page 10 pages, 2016.

\bibitem{PoryferreEmergingBottom}
Thibault de~Poyferr\'{e}.
\newblock A priori estimates for water waves with emerging bottom.
\newblock {\em Arch. Ration. Mech. Anal.}, 232(2):763--812, 2019.

\bibitem{IonescuGlobal3dCapillary}
Yu~Deng, Alexandru~D. Ionescu, Beno\^{\i}t Pausader, and Fabio Pusateri.
\newblock Global solutions of the gravity-capillary water-wave system in three
  dimensions.
\newblock {\em Acta Math.}, 219(2):213--402, 2017.

\bibitem{Disconzilineardynamic}
M.~M. Disconzi.
\newblock On a linear problem arising in dynamic boundaries.
\newblock {\em Evol. Equ. Control Theory}, 3(4):627--644, 2014.

\bibitem{DisconziEbinFreeBoundary2d}
M.~M. Disconzi and D.~G. Ebin.
\newblock On the limit of large surface tension for a fluid motion with free
  boundary.
\newblock {\em Comm. Partial Differential Equations}, 39(4):740--779, 2014.

\bibitem{DisconziEbinFreeBoundary3d}
M.~M. Disconzi and D.~G. Ebin.
\newblock The free boundary {E}uler equations with large surface tension.
\newblock {\em Journal of Differential Equations}, 261(2):821--889, 2016.

\bibitem{DisconziKukavicaIncompressible}
M.~M. Disconzi and I.~Kukavica.
\newblock A priori estimates for the free-boundary {E}uler equations with
  surface tension in three dimensions.
\newblock {\em arXiv: 1708.00086 [math.AP]}, 2017.
\newblock 40 pages.

\bibitem{Ebin_ill-posed}
D.~G. Ebin.
\newblock The equations of motion of a perfect fluid with free boundary are not
  well posed.
\newblock {\em Comm. Partial Differential Equations}, 12(10):1175--1201, 1987.

\bibitem{EbinMarsden}
D.~G. Ebin and J.~Marsden.
\newblock Groups of diffeomorphisms and the motion of an incompressible fluid.
\newblock {\em Annals of Math}, 92:102--163, 1970.

\bibitem{Ev}
Lawrence~C. Evans.
\newblock {\em Partial differential equations}, volume~19 of {\em Graduate
  Studies in Mathematics}.
\newblock American Mathematical Society, Providence, RI, second edition, 2010.

\bibitem{FeffermanIonescuLie}
C.~Fefferman, A.~D. Ionescu, and V.~Lie.
\newblock On the absence of splash singularities in the case of two-fluid
  interfaces.
\newblock {\em Duke Math. J.}, 165(3):417--462, 2016.

\bibitem{GermainMasmoudiShatahGlobalWaterWaves3D}
P.~Germain, N.~Masmoudi, and J.~Shatah.
\newblock Global solutions for the gravity water waves equation in dimension 3.
\newblock {\em Ann. of Math. (2)}, 175(2):691--754, 2012.

\bibitem{GermainMasmoudiShatahGlobalCapillary}
P.~Germain, N.~Masmoudi, and J.~Shatah.
\newblock Global existence for capillary water waves.
\newblock {\em Comm. Pure Appl. Math.}, 68(4):625--687, 2015.

\bibitem{GO}
Loukas Grafakos and Seungly Oh.
\newblock The {K}ato-{P}once inequality.
\newblock {\em Comm. Partial Differential Equations}, 39(6):1128--1157, 2014.

\bibitem{IfrimHunterTataru}
J.~K Hunter, M.~Ifrim, and D.~Tataru.
\newblock Two dimensional water waves in holomorphic coordinates.
\newblock {\em Comm. Math. Phys.}, 346(2):483--552, 2016.

\bibitem{IfrimTataruGlobalWater}
M.~Ifrim and D.~Tataru.
\newblock Two dimensional water waves in holomorphic coordinates {II}: {G}lobal
  solutions.
\newblock {\em Bull. Soc. Math. France}, 144(2):369--394, 2016.

\bibitem{IfrimTataru2dCapillary}
M.~Ifrim and D.~Tataru.
\newblock The {L}ifespan of {S}mall {D}ata {S}olutions in {T}wo {D}imensional
  {C}apillary {W}ater {W}aves.
\newblock {\em Arch. Ration. Mech. Anal.}, 225(3):1279--1346, 2017.

\bibitem{IfrimTataruGravityConstant}
Mihaela Ifrim and Daniel Tataru.
\newblock Two-dimensional gravity water waves with constant vorticity {I}:
  {C}ubic lifespan.
\newblock {\em Anal. PDE}, 12(4):903--967, 2019.

\bibitem{IgorMihaelaSurfaceTension}
M.~Ignatova and I.~Kukavica.
\newblock On the local existence of the free-surface {E}uler equation with
  surface tension.
\newblock {\em Asymptot. Anal.}, 100(1-2):63--86, 2016.

\bibitem{Iguchi_well_posedness_capillary_gravity}
T.~Iguchi.
\newblock Well-posedness of the initial value problem for capillary-gravity
  waves.
\newblock {\em Funkcial. Ekvac.}, 44(2):219--241, 2001.

\bibitem{Iguchi_et_al_FreeBoundary}
T.~Iguchi, N.~Tanaka, and A.~Tani.
\newblock On a free boundary problem for an incompressible ideal fluid in two
  space dimensions.
\newblock {\em Adv. Math. Sci. Appl.}, 9(1):415--472, 1999.

\bibitem{IonescuPusateriWaterWaves2d}
A.~D. Ionescu and F.~Pusateri.
\newblock Global solutions for the gravity water waves system in 2d.
\newblock {\em Invent. Math.}, 199(3):653--804, 2015.

\bibitem{IonescuPusateriGlobal2dwaterModel}
A.~D. Ionescu and F.~Pusateri.
\newblock Global analysis of a model for capillary water waves in two
  dimensions.
\newblock {\em Comm. Pure Appl. Math.}, 69(11):2015--2071, 2016.

\bibitem{IonescuPusateriGlobal2dwaterSurfaceTension}
Alexandru~D. Ionescu and Fabio Pusateri.
\newblock Global regularity for 2{D} water waves with surface tension.
\newblock {\em Mem. Amer. Math. Soc.}, 256(1227):v+124, 2018.

\bibitem{KP}
Tosio Kato and Gustavo Ponce.
\newblock Well-posedness of the {E}uler and {N}avier-{S}tokes equations in the
  {L}ebesgue spaces {$L^p_s({\bf R}^2)$}.
\newblock {\em Rev. Mat. Iberoamericana}, 2(1-2):73--88, 1986.

\bibitem{KPV}
Carlos~E. Kenig, Gustavo Ponce, and Luis Vega.
\newblock On the (generalized) {K}orteweg-de {V}ries equation.
\newblock {\em Duke Math. J.}, 59(3):585--610, 1989.

\bibitem{KPV91}
Carlos~E. Kenig, Gustavo Ponce, and Luis Vega.
\newblock Well-posedness of the initial value problem for the {K}orteweg-de
  {V}ries equation.
\newblock {\em J. Amer. Math. Soc.}, 4(2):323--347, 1991.

\bibitem{KukavicaTuffaha-Free2dEuler}
I.~Kukavica and A.~Tuffaha.
\newblock On the 2{D} free boundary {E}uler equation.
\newblock {\em Evol. Equ. Control Theory}, 1(2):297--314, 2012.

\bibitem{KukavicaTuffaha-RegularityFreeEuler}
I.~Kukavica and A.~Tuffaha.
\newblock A regularity result for the incompressible {E}uler equation with a
  free interface.
\newblock {\em Appl. Math. Optim.}, 69(3):337--358, 2014.

\bibitem{KukavicaTuffahaVicol-3dFreeEuler}
Igor Kukavica, Amjad Tuffaha, and Vlad Vicol.
\newblock On the local existence and uniqueness for the 3{D} {E}uler equation
  with a free interface.
\newblock {\em Appl. Math. Optim.}, 76(3):535--563, 2017.

\bibitem{KWZ}
Igor Kukavica, Fei Wang, and Mohammed Ziane.
\newblock Persistence of regularity for solutions of the {B}oussinesq equations
  in {S}obolev spaces.
\newblock {\em Adv. Differential Equations}, 21(1-2):85--108, 2016.

\bibitem{LannesWaterWaves}
D.~Lannes.
\newblock Well-posedness of the water-waves equations.
\newblock {\em J. Amer. Math. Soc.}, 18(3):605--654 (electronic), 2005.

\bibitem{LannesWaterWavesBook}
D.~Lannes.
\newblock {\em The water waves problem}, volume 188 of {\em Mathematical
  Surveys and Monographs}.
\newblock American Mathematical Society, Providence, RI, 2013.
\newblock Mathematical analysis and asymptotics.

\bibitem{LindbladFree1}
H.~Lindblad.
\newblock The motion of the free surface of a liquid.
\newblock In {\em S\'eminaire: \'{E}quations aux {D}\'eriv\'ees {P}artielles,
  2000--2001}, S\'emin. \'Equ. D\'eriv. Partielles, pages Exp. No. VI, 10.
  \'Ecole Polytech., Palaiseau, 2001.

\bibitem{Lindblad-LinearizedFreeBoundary}
H.~Lindblad.
\newblock Well-posedness for the linearized motion of an incompressible liquid
  with free surface boundary.
\newblock {\em Comm. Pure Appl. Math.}, 56(2):153--197, 2003.

\bibitem{LindbladNordgren-AprioriFreeBoundary}
H.~Lindblad and K.~H. Nordgren.
\newblock A priori estimates for the motion of a self-gravitating
  incompressible liquid with free surface boundary.
\newblock {\em J. Hyperbolic Differ. Equ.}, 6(2):407--432, 2009.

\bibitem{NalimovCauchyPoisson}
V.~I. Nalimov.
\newblock The {C}auchy-{P}oisson problem.
\newblock {\em Dinamika Splo\v sn. Sredy}, (Vyp. 18 Dinamika Zidkost. so
  Svobod. Granicami):104--210, 254, 1974.

\bibitem{Ogawa-Tani_FreeBoundarySurfaceTension}
M.~Ogawa and A.~Tani.
\newblock Free boundary problem for an incompressible ideal fluid with surface
  tension.
\newblock {\em Math. Models Methods Appl. Sci.}, 12(12):1725--1740, 2002.

\bibitem{Ogawa-Tani_FiniteDepth}
M.~Ogawa and A.~Tani.
\newblock Incompressible perfect fluid motion with free boundary of finite
  depth.
\newblock {\em Adv. Math. Sci. Appl.}, 13(1):201--223, 2003.

\bibitem{PusateriTwoPhaseOnePhaseLimitSurfaceTension}
F.~Pusateri.
\newblock On the limit as the surface tension and density ratio tend to zero
  for the two-phase {E}uler equations.
\newblock {\em J. Hyperbolic Differ. Equ.}, 8(2):347--373, 2011.

\bibitem{SchweizerFreeEuler}
B.~Schweizer.
\newblock On the three-dimensional {E}uler equations with a free boundary
  subject to surface tension.
\newblock {\em Ann. Inst. H. Poincar\'e Anal. Non Lin\'eaire}, 22(6):753--781,
  2005.

\bibitem{ShatahZengGeometry}
J.~Shatah and C.~Zeng.
\newblock Geometry and a priori estimates for free boundary problems of the
  {E}uler equation.
\newblock {\em Comm. Pure Appl. Math.}, 61(5):698--744, 2008.

\bibitem{ShatahZengInterface}
J.~Shatah and C.~Zeng.
\newblock Local well-posedness for fluid interface problems.
\newblock {\em Arch. Ration. Mech. Anal.}, 199(2):653--705, 2011.

\bibitem{T}
Roger Temam.
\newblock {\em Navier-{S}tokes equations}.
\newblock AMS Chelsea Publishing, Providence, RI, 2001.
\newblock Theory and numerical analysis, Reprint of the 1984 edition.

\bibitem{WZZZ}
C.~Wang, Z.~Zhang, W.~Zhao, and Y.~Zheng.
\newblock Local well-posedness and break-down criterion of the incompressible
  euler equations with free boundary.
\newblock {\em arXiv:1507.02478}, 2015.

\bibitem{WuWaterWaves2d}
S.~Wu.
\newblock Well-posedness in {S}obolev spaces of the full water wave problem in
  {$2$}-{D}.
\newblock {\em Invent. Math.}, 130(1):39--72, 1997.

\bibitem{WuWaterWaves}
S.~Wu.
\newblock Well-posedness in {S}obolev spaces of the full water wave problem in
  3-{D}.
\newblock {\em J. Amer. Math. Soc.}, 12(2):445--495, 1999.

\bibitem{WuAlmostGlobal}
S.~Wu.
\newblock Almost global wellposedness of the 2-{D} full water wave problem.
\newblock {\em Invent. Math.}, 177(1):45--135, 2009.

\bibitem{WuGlobal}
S.~Wu.
\newblock Global wellposedness of the 3-{D} full water wave problem.
\newblock {\em Invent. Math.}, 184(1):125--220, 2011.

\bibitem{YosiharaGravity}
H.~Yosihara.
\newblock Gravity waves on the free surface of an incompressible perfect fluid
  of finite depth.
\newblock {\em Publ. Res. Inst. Math. Sci.}, 18(1):49--96, 1982.

\bibitem{YosiharaGravitySurfaceTension}
H.~Yosihara.
\newblock Capillary-gravity waves for an incompressible ideal fluid.
\newblock {\em J. Math. Kyoto Univ.}, 23(4):649--694, 1983.

\end{thebibliography}
\end{document}